\documentclass[a4paper,icelandic,english,english,bibtotoc,a4paper]{scrartcl}
\usepackage[T1]{fontenc}
\usepackage[latin9]{inputenc}
\usepackage{prettyref}
\usepackage{float}
\usepackage{amsthm}
\usepackage{amsmath}
\usepackage{amssymb}
\usepackage{makeidx}
\makeindex
\usepackage{graphicx}
\usepackage{esint}
\usepackage{nomencl}

\providecommand{\makenomenclature}{\makeglossary}
\makenomenclature

\makeatletter

\pdfpageheight\paperheight
\pdfpagewidth\paperwidth

\newcommand{\lyxdot}{.}

\numberwithin{equation}{section}
\numberwithin{figure}{section}
\theoremstyle{plain}
\newtheorem{thm}{\protect\theoremname}[section]
  \theoremstyle{remark}
  \newtheorem{rem}[thm]{\protect\remarkname}
  \theoremstyle{definition}
  \newtheorem{defn}[thm]{\protect\definitionname}
  \theoremstyle{plain}
  \newtheorem{cor}[thm]{\protect\corollaryname}
  \theoremstyle{definition}
  \newtheorem{example}[thm]{\protect\examplename}
  \theoremstyle{plain}
  \newtheorem{prop}[thm]{\protect\propositionname}
  \theoremstyle{plain}
  \newtheorem{lem}[thm]{\protect\lemmaname}
  \theoremstyle{remark}
  \newtheorem*{rem*}{\protect\remarkname}

\def\documenttitle{Loss systems in a random environment}

\usepackage{tikz}
\usetikzlibrary{shapes,arrows}
\usepackage{fancyhdr}
\usepackage{cancel}
\usepackage[normalem]{ulem} 
\usepackage{xcolor}

\usepackage{makeidx}
\makeindex
\usepackage[nottoc]{tocbibind}

\usepackage[unicode=true, 
 bookmarks=true,bookmarksnumbered=false,bookmarksopen=false,
 breaklinks=true,pdfborder={0 0 0},backref=false,colorlinks=true]
 {hyperref}

\hypersetup{
 linkcolor=blue, citecolor=blue, urlcolor=blue,  filecolor=blue,
 pdftitle={\documenttitle},
 pdfauthor={Ruslan Krenzler, Hans Daduna},
 pdfkeywords={queueing systems} {random environment} {product form steady state}
{loss systems}{M/M/1/Inf} {M/M/m/Inf}{M/G/1/Inf}{embedded Markov chains}
 {inventory systems} {availability} {lost sales}{matrix invertibility},
 pdfsubject={queueing systems}
}

\newrefformat{thm}{Theorem \ref{#1}}
\newrefformat{cor}{Corollary \ref{#1}}
\newrefformat{ex}{Example \ref{#1}}
\newrefformat{prop}{Proposition \ref{#1}}
\newrefformat{prob}{Problem \ref{#1}}
\newrefformat{part}{Part \ref{#1}}
\newrefformat{sect}{Section \ref{#1}}
\newrefformat{def}{Definition \ref{#1}}
\newrefformat{rem}{Remark \ref{#1}}

\usepackage{pgfplots}
\usepackage{pgf}
\usetikzlibrary{shapes,arrows} 

\usepackage{ifthen}

\makeatletter
\pgfkeys{/tikz/queue head/.initial=east}
\pgfkeys{/tikz/queue size/.initial=infinite}
\pgfkeys{/tikz/queue ellipsis/.initial=false}

\def\queue@width{
	\pgf@x=\wd\pgfnodeparttextbox
	\pgfmathsetlength{\pgf@xa}{\pgfkeysvalueof{/pgf/minimum width}}%
	\ifdim\pgf@x<\pgf@xa
		\pgf@x=\pgf@xa
	\fi
}

\def\queue@height{
	\pgf@x=\ht\pgfnodeparttextbox
	\pgfmathsetlength{\pgf@xa}{\pgfkeysvalueof{/pgf/minimum height}}%
	\ifdim\pgf@x<\pgf@xa
		\pgf@x=\pgf@xa
	\fi
}

\pgfdeclareshape{queue-base}
{
	\inheritsavedanchors[from={rectangle}]
	\inheritanchorborder[from={rectangle}]
	\inheritanchor[from={rectangle}]{center}
	\inheritanchor[from={rectangle}]{north}
	\inheritanchor[from={rectangle}]{east}
	\inheritanchor[from={rectangle}]{south}
	\inheritanchor[from={rectangle}]{west}
	
	\savedmacro{\queuehead}{\pgfkeysvalueof{/tikz/queue head}}

	\saveddimen{\halfwidth}{
		\queue@width
		\divide \pgf@x by 2
	}
	\saveddimen{\halfheight}{
		\queue@height
		\divide \pgf@x by 2
	}

	\anchor{head}{
		\ifthenelse{\equal{\pgfkeysvalueof{/tikz/queue head}}{east}}{
				\northeast
				\advance \pgf@y by -\halfheight
		}{}
		\ifthenelse{\equal{\pgfkeysvalueof{/tikz/queue head}}{west}}{
				\southwest
				\advance \pgf@y by \halfheight
		}{}
		\ifthenelse{\equal{\pgfkeysvalueof{/tikz/queue head}}{north}}{
				\northeast
				\advance \pgf@x by -\halfwidth
		}{}
		\ifthenelse{\equal{\pgfkeysvalueof{/tikz/queue head}}{south}}{
				\southwest
				\advance \pgf@x by \halfwidth
		}{}
	}
	
	\anchor{tail}{
		\ifthenelse{\equal{\pgfkeysvalueof{/tikz/queue head}}{east}}{
				\southwest
				\advance \pgf@y by \halfheight
		}{}
		\ifthenelse{\equal{\pgfkeysvalueof{/tikz/queue head}}{west}}{
				\northeast
				\advance \pgf@y by -\halfheight
		}{}
		\ifthenelse{\equal{\pgfkeysvalueof{/tikz/queue head}}{north}}{
				\southwest
				\advance \pgf@x by \halfwidth
		}{}
		\ifthenelse{\equal{\pgfkeysvalueof{/tikz/queue head}}{south}}{
				\northeast
				\advance \pgf@x by -\halfwidth
		}{}
	}
}

\pgfdeclareshape{queue}
{
	\inheritsavedanchors[from={queue-base}]
	\inheritanchorborder[from={queue-base}]
	\inheritanchor[from={queue-base}]{center}
	\inheritanchor[from={queue-base}]{north}
	\inheritanchor[from={queue-base}]{east}
	\inheritanchor[from={queue-base}]{south}
	\inheritanchor[from={queue-base}]{west}
	\inheritanchor[from={queue-base}]{head}
	\inheritanchor[from={queue-base}]{tail}

	\savedmacro{\queuehead}{\pgfkeysvalueof{/tikz/queue head}}

	\saveddimen{\width}{
		\queue@width
	}

	\saveddimen{\height}{
		\queue@height
	}

	\backgroundpath{
		\pgfextractx\pgf@xa{\southwest}
		\pgfextracty\pgf@ya{\southwest}
		\pgfextractx\pgf@xb{\northeast}
		\pgfextracty\pgf@yb{\northeast}
		
 		\newboolean{vertical}
		\ifthenelse{\equal{\pgfkeysvalueof{/tikz/queue head}}{north}
		 	\OR  
			\equal{\pgfkeysvalueof{/tikz/queue head}}{south}}{
			\setboolean{vertical}{true}	
		}{
			\setboolean{vertical}{false}
		}
		
		\newboolean{ellipsis}
		\ifthenelse{\equal{\pgfkeysvalueof{/tikz/queue ellipsis}}{true}}{
			\setboolean{ellipsis}{true}	
		}{
			\setboolean{ellipsis}{false}
		}
				
		\newboolean{sizeisnumber}
		\ifthenelse{\equal{\pgfkeysvalueof{/tikz/queue size}}{infinite}
		\OR
		\equal{\pgfkeysvalueof{/tikz/queue size}}{bounded}}
		{
			\setboolean{sizeisnumber}{false}
		}{
			\setboolean{sizeisnumber}{true}
		}
		
		\ifthenelse{\boolean{vertical}}{
			\pgfpathmoveto{\pgfpoint{\pgf@xa}{\pgf@ya}}
			\pgfpathlineto{\pgfpoint{\pgf@xa}{\pgf@yb}}
			\pgfpathmoveto{\pgfpoint{\pgf@xb}{\pgf@ya}}
			\pgfpathlineto{\pgfpoint{\pgf@xb}{\pgf@yb}}	
		}{
			\pgfpathmoveto{\pgfpoint{\pgf@xa}{\pgf@ya}}
			\pgfpathlineto{\pgfpoint{\pgf@xb}{\pgf@ya}}
			\pgfpathmoveto{\pgfpoint{\pgf@xa}{\pgf@yb}}
			\pgfpathlineto{\pgfpoint{\pgf@xb}{\pgf@yb}}		
		}
		
		\newdimen \midy
		\midy=\pgf@ya
		\advance \midy by \pgf@yb
		\divide \midy by 2
		
		\newdimen \midx
		\midx=\pgf@xa
		\advance \midx by \pgf@xb
		\divide \midx by 2
		
		\ifthenelse{\boolean{vertical}}{
			\newdimen \cellheight
			\cellheight=\height

   	 		\ifthenelse{\boolean{sizeisnumber}}
			{
				\pgfmathparse{\pgfkeysvalueof{/tikz/queue size}}
				\divide\cellheight by \pgfmathresult
			}{
				\divide\cellheight by 6
			}
			
			\newdimen \currcelly
			\currcelly=\pgf@ya

			\ifthenelse{\equal{\pgfkeysvalueof{/tikz/queue size}}{infinite}
			 		\AND
			 		\equal{\pgfkeysvalueof{/tikz/queue head}}{north}}
        	{
       	 		\foreach \y in {2,3,4,5,6}{
					\advance \currcelly by \y\cellheight
					\pgfpathmoveto{\pgfpoint{\pgf@xa}{\currcelly}}
					\pgfpathlineto{\pgfpoint{\pgf@xb}{\currcelly}}
				}
				\ifthenelse{\boolean{ellipsis}}{
					\pgf@yc = \pgf@ya
					\advance \pgf@yc by \cellheight
					\filldraw (\midx, \pgf@yc) circle (1pt);
					\filldraw (\midx, \pgf@yc+5pt) circle (1pt);
					\filldraw (\midx, \pgf@yc-5pt) circle (1pt);
				}{}
       	 	}{}

			 \ifthenelse{\equal{\pgfkeysvalueof{/tikz/queue size}}{infinite}
			 		\AND
			 		\equal{\pgfkeysvalueof{/tikz/queue head}}{south}}
       	 	{
       	 		\foreach \y in {0,1,2,3,4}{
					\advance \currcelly by \y\cellheight
					\pgfpathmoveto{\pgfpoint{\pgf@xa}{\currcelly}}
					\pgfpathlineto{\pgfpoint{\pgf@xb}{\currcelly}}
				}
				\ifthenelse{\boolean{ellipsis}}{
					\pgf@yc = \pgf@yb
					\advance \pgf@yc by -\cellheight
					\filldraw (\midx, \pgf@yc) circle (1pt);
					\filldraw (\midx, \pgf@yc+5pt) circle (1pt);
					\filldraw (\midx, \pgf@yc-5pt) circle (1pt);
				}{}
        	}{}
        	
        	\ifthenelse{\equal{\pgfkeysvalueof{/tikz/queue size}}{bounded}}
       	 	{
       	 		\foreach \y in {0,1,2,4,5,6}{
					\advance \currcelly by \y\cellheight
					\pgfpathmoveto{\pgfpoint{\pgf@xa}{\currcelly}}
					\pgfpathlineto{\pgfpoint{\pgf@xb}{\currcelly}}
				}
				\pgf@yc = \pgf@ya
				\advance \pgf@yc by 3\cellheight
				\filldraw (\midx, \pgf@yc) circle (1pt);
				\filldraw (\midx, \pgf@yc+5pt) circle (1pt);
				\filldraw (\midx, \pgf@yc-5pt) circle (1pt);
        	}{}
        	
        	\ifthenelse{\boolean{sizeisnumber}}{
        		\pgfmathparse{\pgfkeysvalueof{/tikz/queue size}}
        		\foreach \y in {0,...,\pgfmathresult}{
					\advance \currcelly by \y\cellheight
					\pgfpathmoveto{\pgfpoint{\pgf@xa}{\currcelly}}
					\pgfpathlineto{\pgfpoint{\pgf@xb}{\currcelly}}
				}
        	}{}
		}{
			\newdimen \cellwidth
			\cellwidth=\width
			
   	 		\ifthenelse{\boolean{sizeisnumber}}
			{
				\pgfmathparse{\pgfkeysvalueof{/tikz/queue size}}
				\divide\cellwidth by \pgfmathresult
			}{
				\divide\cellwidth by 6
			}
			
			\newdimen \currcellx
			\currcellx=\pgf@xa
			\ifthenelse{\equal{\pgfkeysvalueof{/tikz/queue size}}{infinite}
			 		\AND
			 		\equal{\pgfkeysvalueof{/tikz/queue head}}{east}}
        	{
       	 		\foreach \x in {2,3,4,5,6}{
					\advance \currcellx by \x\cellwidth
					\pgfpathmoveto{\pgfpoint{\currcellx}{\pgf@ya}}
					\pgfpathlineto{\pgfpoint{\currcellx}{\pgf@yb}}
				}
				\ifthenelse{\boolean{ellipsis}}{
					\pgf@xc = \pgf@xa
					\advance \pgf@xc by \cellwidth
					\filldraw (\pgf@xc,\midy) circle (1pt);
					\filldraw (\pgf@xc+5pt,\midy) circle (1pt);
					\filldraw (\pgf@xc-5pt,\midy) circle (1pt);
				}{}
       	 	}{}

			 \ifthenelse{\equal{\pgfkeysvalueof{/tikz/queue size}}{infinite}
			 		\AND
			 		\equal{\pgfkeysvalueof{/tikz/queue head}}{west}}
       	 	{
	        	\foreach \x in {0,1,2,3,4}{
					\advance \currcellx by \x\cellwidth
					\pgfpathmoveto{\pgfpoint{\currcellx}{\pgf@ya}}
					\pgfpathlineto{\pgfpoint{\currcellx}{\pgf@yb}}
				}
				\ifthenelse{\boolean{ellipsis}}{
					\pgf@xc = \pgf@xb
					\advance \pgf@xc by -\cellwidth
					\filldraw (\pgf@xc,\midy) circle (1pt);
					\filldraw (\pgf@xc+5pt,\midy) circle (1pt);
					\filldraw (\pgf@xc-5pt,\midy) circle (1pt);
				}{}
        	}{}
        	
        	\ifthenelse{\equal{\pgfkeysvalueof{/tikz/queue size}}{bounded}}
       	 	{
	        	\foreach \x in {0,1,2,4,5,6}{
					\advance \currcellx by \x\cellwidth
					\pgfpathmoveto{\pgfpoint{\currcellx}{\pgf@ya}}
					\pgfpathlineto{\pgfpoint{\currcellx}{\pgf@yb}}
				}
				\pgf@xc = \pgf@xa
				\advance \pgf@xc by 3\cellwidth
				\filldraw (\pgf@xc,\midy) circle (1pt);
				\filldraw (\pgf@xc+5pt,\midy) circle (1pt);
				\filldraw (\pgf@xc-5pt,\midy) circle (1pt);
        	}{}
        	
        	\ifthenelse{\boolean{sizeisnumber}}{
        		\pgfmathparse{\pgfkeysvalueof{/tikz/queue size}}
        		\foreach \x in {0,...,\pgfmathresult}{
					\advance \currcellx by \x\cellwidth
					\pgfpathmoveto{\pgfpoint{\currcellx}{\pgf@ya}}
					\pgfpathlineto{\pgfpoint{\currcellx}{\pgf@yb}}
				}
        	}{}
        } 
	}
}

\definecolor{blocked.bg}{RGB}{255,125,125}
\definecolor{vratecolor}{named}{cyan}

\newcommand\etid{environment transition and interaction diagram}
\newcommand\Etid{Environment transition and interaction diagram}

\@ifundefined{showcaptionsetup}{}{%
 \PassOptionsToPackage{caption=false}{subfig}}
\usepackage{subfig}
\makeatother

\usepackage{babel}
  \addto\captionsenglish{\renewcommand{\corollaryname}{Corollary}}
  \addto\captionsenglish{\renewcommand{\definitionname}{Definition}}
  \addto\captionsenglish{\renewcommand{\examplename}{Example}}
  \addto\captionsenglish{\renewcommand{\lemmaname}{Lemma}}
  \addto\captionsenglish{\renewcommand{\propositionname}{Proposition}}
  \addto\captionsenglish{\renewcommand{\remarkname}{Remark}}
  \addto\captionsenglish{\renewcommand{\theoremname}{Theorem}}
  \addto\captionsicelandic{\renewcommand{\corollaryname}{Corollary}}
  \addto\captionsicelandic{\renewcommand{\definitionname}{Definition}}
  \addto\captionsicelandic{\renewcommand{\examplename}{Example}}
  \addto\captionsicelandic{\renewcommand{\lemmaname}{Lemma}}
  \addto\captionsicelandic{\renewcommand{\propositionname}{Proposition}}
  \addto\captionsicelandic{\renewcommand{\remarkname}{Remark}}
  \addto\captionsicelandic{\renewcommand{\theoremname}{Theorem}}
  \providecommand{\corollaryname}{Corollary}
  \providecommand{\definitionname}{Definition}
  \providecommand{\examplename}{Example}
  \providecommand{\lemmaname}{Lemma}
  \providecommand{\propositionname}{Proposition}
  \providecommand{\remarkname}{Remark}
\providecommand{\theoremname}{Theorem}

\begin{document}

\title{\documenttitle}

\author{Ruslan Krenzler, Hans Daduna%
\thanks{University of Hamburg, Department of Mathematics, Bundesstr. 55, 20146
Hamburg, Germany.%
}\\
}

\date{December 2, 2013}
\maketitle
\begin{abstract}
We consider a single server system with infinite waiting room in a
random environment. The service system and the environment interact
in both directions. Whenever the environment enters a prespecified
subset of its state space the service process is completely blocked:
Service is interrupted and newly arriving customers are lost. We prove
an if-and-only-if-condition for a product form steady state distribution
of the joint queueing-environment process. A consequence is a strong
insensitivity property for such systems.

We discuss several applications, e.g. from inventory theory and reliability
theory, and show that our result extends and generalizes several theorems
found in the literature, e.g. of queueing-inventory processes.

We investigate further classical loss systems, where due to finite
waiting room loss of customers occurs. In connection with loss of
customers due to blocking by the environment and service interruptions
new phenomena arise.

We further investigate the embedded Markov chains at departure epochs
and show that the behaviour of the embedded Markov chain is often
considerably different from that of the continuous time Markov process.
This is different from the behaviour of the standard M/G/1/$\infty$,where
the steady state of the embedded Markov chain and the continuous time
process coincide.

For exponential queueing systems we show that there is a product form
equilibrium of the embedded Markov chain under rather general conditions.
For systems with non-exponential service times more restrictive constraints
are needed, which we prove by a counter example where the environment
represents an inventory attached to an M/D/1 queue. Such integrated
queueing-inventory systems are dealt with in the literature previously,
and are revisited here in detail. \\

\end{abstract}


\tableofcontents{}

\section{Introduction\label{sect:intro}}

Product form networks of queues are common models for easy to perform
structural and quantitative first order analysis of complex networks
in Operations Research applications. The most prominent representatives
of this class of models are the Jackson \cite{jackson:57} and Gordon-Newell
\cite{gordon;newell:67} networks and their generalizations as BCMP
\cite{baskett;chandy;muntz;palacios:75} and Kelly \cite{kelly:76}
networks, for a short review see \cite{daduna:01a}.

Standard mathematical description of this class of models is by time
homogeneous Markovian vector processes, where each coordinate represents
the behaviour of one of the queues. Product form networks are characterized
by the fact that in steady state (at any fixed time $t$) the joint
distribution of the multi-dimensional (over nodes) queueing process
is the product of the stationary marginal distributions of the individual
nodes' (non Markovian) queueing processes. With respect to the research
described in this note the key point is that the coordinates of the
vector process represent objects of the same class, namely queueing
systems.\\

In Operations Research applications queueing systems constitute an
important class of models in very different settings. Nevertheless,
in many applications those parts of, e.g., a complex production system
which are modeled by queues interact with other subsystems which usually
can not be modeled by queues. We will describe two prototype situations
which will be considered in detail throughout the paper as introductory
examples. These examples deal with interaction of (i) a queueing system
with an inventory, and (ii) a queueing system with its environment,
which influences the availability of the server.\\

\textbf{Introductory example (i): Production-inventory system.} Typically,
there is a manufacturing system, (machine, modeled by a queueing system)
which assembles delivered raw material to a final product, consuming
in the production process some further material (we will call these
additional pieces of material ''items\textquotedbl{} henceforth)
which is hold in inventories.\\

\textbf{Introductory example (ii): Availability of a production system.}
The manufacturing system (machine, modeled by a queueing system) may
break down caused by influences of its environment or by wear out
of its server and has to be repaired.\\

Our present research is motivated by the observation that in both
situations we have to construct an integrated model with a common
structure: 
\begin{itemize}
\item A production processes modeled by queueing systems and 
\item an additional relevant part of the system with different character,
e.g., inventory control or availability control. 
\end{itemize}
In both models the components strongly interact, and although the
additional feature are quite different, we will extract similarities.
This motivates to construct a unified model which encompasses both
introductory examples and as we will show, many other examples in
different fields. For taking a general nomenclature we will subsume
in any case the ''additional relevant part of the system'' attached
to the queueing model as ''environment of the queue''.

Our construction of ''queueing systems in a random environment''
will result in a set of product form stationary distributions for
the Markovian joint queueing-environment process, i.e., in equilibrium
the coordinates at fixed time points decouple: The stationary distribution
of the joint queueing-environment process is the product of stationary
marginal distributions of the queue and the environment, which in
general cannot be described as a Markov process of their own. The
key point is that the coordinates of the vector process represent
very different classes of objects.

Product form stationary distributions for the introductory examples
have been found only recently.\\
 \textbf{(i)} Schwarz, Sauer, Daduna, Kulik, and Szekli \cite{schwarz;sauer;daduna;kulik;szekli:06}
discovered product forms for the steady state distributions of an
$M/M/1/\infty$ under standard order policies with lost sales. Further
contributions to product form results in this field are by Vineetha
\cite{vineetha:08}, Saffari, Haji, and Hassanzadeh \cite{saffari;haji;hassanzadeh:11},
and Saffari, Asmussen, and Haji \cite{saffari;asmussen;haji:13}.
An early paper of Berman and Kim \cite{berman;kim:99} can be considered
to contribute to integrated models with product form steady state.\\
 \textbf{(ii)} For classical product form networks of queues in \cite{sauer;daduna:03}
the influence of breakdown and repair of the nodes was studied and
it was proved that under certain conditions a product form equilibrium
for such networks of unreliable servers exists. The Markovian description
of the system encompasses coordinates describing (possibly many) queues
and an additional coordinate to indicate the reliability status of
the system, for more details see \cite{sauer:06}.\\
 Related research on queueing systems in a random environment is by
Zhu \cite{zhu:94}, Economou \cite{economou:05}, Tsitsiashvili, Osipova,
Koliev, Baum \cite{tsitsiashvili;osipova;koliev;baum:02}, and Balsamo,Marin
\cite{balsamo;marin:13}. There usually the environment is a Markov
process of its own, which is the case neither in our model nor in
the motivating results in \cite{schwarz;sauer;daduna;kulik;szekli:06}
and \cite{sauer;daduna:03}, which consider our introductory examples.\\

An important common aspect of the interaction in both introductory
examples leads to the term \textbf{loss system} for our general interacting
system: Whenever for the queue, respectively, 
\begin{itemize}
\item the inventory is depleted (i), 
\item the machine is broken down (ii), 
\end{itemize}
service at the production unit is interrupted due to stock out (i),
resp. no production capacity available (ii). Additionally, during
the time the interruption continues no new arrivals are admitted to
both systems, due to lost sales and because customers prefer to enter
some other working server.

Note, that this loss of customers is different from what is usually
termed loss systems in pure queueing theory, where loss of customers
happens, when the finite waiting space is filled up to maximal capacity.\\

Following the above description, in our present investigation of complex
systems we always start with a queueing system as one subsystem and
a general attached other subsystem (the environment) which imposes
side constraints on the queueing process and in general interacts
in both directions with the queue. In typical cases there will be
a part of the environment's state space, the states of which we shall
call ''blocking states'', with the following property: Whenever
the environment enters a blocking state, the service process will
be interrupted and no new arrivals are admitted to enter the system
and are lost to the system forever.\\
 The interaction of the components in this class of models is that
jumps of the queue may enforce the environment to jump instantaneously,
and in the other direction the evolving environment may interrupt
service and arrivals at the queue, by entering blocking states, and
when leaving the set of blocking states service is resumed and new
customers are admitted again.\\

We describe our exponential system in \prettyref{sect:LS-MM1Inf-model}
and start our detailed investigation in \prettyref{sect:LS-steady-state}.
Our main result (\prettyref{thm:LS-productform}) is that, although
production and environment strongly interact, asymptotically and in
equilibrium (at fixed time instants) the production process and the
environment process seem to decouple, which means that a product form
equilibrium emerges.

This shows that the mentioned independence results in \cite{schwarz;sauer;daduna;kulik;szekli:06},
\cite{sauer;daduna:03}, \cite{saffari;haji;hassanzadeh:11}, \cite{saffari;asmussen;haji:13},
and \cite{vineetha:08} do not depend on the specific properties of
the attached second subsystem. Furthermore, we will show that the
theorem can be interpreted as a strong insensitivity property of the
system: As long as ergodicity is maintained, the environment can change
drastically without changing the steady state distribution of the
queue length.

And, vice versa, it can be seen that the environment's steady state
will not change when the service capacity of the production will change.

We shall discuss this with related problems and some complements to
the theorem in more detail in \prettyref{sect:LS-steady-state} after
presenting our main result there. In \prettyref{sect:LS-finite-capacity}
we investigate the case of finite waiting space at the queue, so in
the classical loss system we introduce additional losses due to the
impact of the environment on the production process. Astonishingly,
there occur new structural problems when product form steady states
are found.

In \prettyref{sect:LS-applications} we present a bulk of applications
of our abstract modeling process to systems found in the literature.
We show especially, that our main theorem allows to generalize rather
directly many of the previous results. In \prettyref{sect:LS-EMC-exponential}
we consider the systems (which live in continuous time) at departure
instants only, which results in considering an embedded Markov chain.
We find that the behaviour of the embedded Markov chain is often considerably
different from that of the original continuous time Markov process
investigated in \prettyref{part:LS-continous-time}. Especially, it
is a non trivial task to decide whether the stationary distribution
of the embedded Markov chain (at departure instants) is of product
form as well.

For exponential queueing systems we show that there is a product form
equilibrium under rather general conditions. We provide this stationary
distributions explicitly in \prettyref{thm:LCS-EMC-product-form-exp-service},
showing that the marginal queue length distribution is the same as
in continuous time, and discuss the relation relation between the
respective marginal environment distributions, which are not equal
but related by a transformation which we explicitly give in \prettyref{lem:LS-EMC-theta-theta-hat-releationship}.

To emphasize the problems arising from the interaction of the two
components of integrated systems, we remind the reader, that for ergodic
$M/M/1/\infty$ queues the limiting and stationary distribution of
the continuous time queue length process and the Markov chains embedded
at either arrival instants or departure instants are the same. In
connection with this, we revisit some of Vineetha's \cite{vineetha:08}
queueing-inventory systems, using similarly embedded Markov chain
techniques.

A striking observation is moreover, that for a system which is ergodic
in the continuous time Markovian description the Markov chain embedded
at departure instants may be not ergodic. The reason for this is two-fold.
Firstly, the embedded Markov chain may have inessential states due
to the specified interaction rules. Secondly, even when we delete
all inessential states, the resulting single positive recurrent class
may be periodic. We study this problem in depth in \prettyref{sect:LS-EMC-steadystate}.

In \prettyref{sect:LS-EMC-MG1} we show that for systems with non-exponential
service times more restrictive constraints are needed, which we prove
by a counter example where the environment represents an inventory
attached to an $M/D/1$ queue. Such integrated queueing-inventory
systems are dealt with in the literature previously, e.g. in \cite{vineetha:08}.
Further applications are, e.g., in modeling unreliable queues.

In \prettyref{sect:LS-EMC-applications} we present further applications
and discuss the differences between the stationary distributions of
the continuous time process and the embedded Markov chain.

In \prettyref{sect:LS-EMC-useful-lemma} we provide some useful lemmata
which seem to be of interest for their own, because we can generalize
some standard results on invertibility of M-matrices. The invertibility
was a necessary ingredient of our proofs in the main body of the paper,
but the assumption of irreducibility of the M-matrices which is required
in the literature (see e.g. \cite{kanzow:05}, Lemma 4.12) does not
hold in our models.\\

\textbf{Related work:} We have cited literature related to our introductory
examples which deal with product form stationary distributions above.
Clearly, there are many investigations on queueing systems with unreliable
servers without this separability property, for a survey see the introductions
in \cite{sauer;daduna:03} and \cite{sauer:06}.

In classical Operations Research the fields of queueing theory and
inventory theory are almost disjoint areas of research. Recently,
research on integrated models has found some interest, a survey is
the review in \cite{krishnamoorthy;lakshmy;manikandan:11}.

In a more abstract setting, we can describe our present work as to
develop a framework for a birth-death process in a random environment,
where the birth-death process' development is interrupted from time
to time by some configurations occurring in the environment. On the
other side, in our framework the birth-death process influences the
development of the environment.

There are many investigations on birth-death processes in random environments,
we shall cite only some selected references. Best to our knowledge
our results below are complementary to the literature. A stream of
research on birth-death processes in a random environment exploits
the interaction of birth-death process and environment as the typical
structure of a quasi-birth-death process. Such ''QBD processes''
have two dimensional states, the ''level'' indicates the population
size, while the ''phase'' represents the environment. For more details
see Chapter 6 (Queues in a Random Environment) in \cite{neuts:81},
and Example C in \cite{neuts:89}{[}p. 202, 203{]}.

Related models are investigated in the theory of branching processes
in a random environment, see Section 2.9 in \cite{haccou;jagers;vatutin:05}
for a short review. An early survey with many references to old literature
is \cite{kesten:80}.

Another branch of research is optimization of queues under constraints
put on the queue by a randomly changing environment as described e.g.
in \cite{helm;waldmann:84}.

While the most of the annotated sources are concerned with conventional
steady state analysis, the work \cite{falin:96} is related to ours
two-fold: A queue (finite classical loss system) in a random environment
shows a product form steady state.\\

This paper is an extension and unification of the preprints \cite{krenzler;daduna:12}
and \cite{krenzler;daduna:13}.

\begin{center}
\begin{figure}[H]
\centering{}\begin{tikzpicture}
\begin{scope}[shift={(0,0.1)}]
 \node(envorment-label)[] at (3.25,1) {environment ($Y(t)$)};
 \draw[rounded corners=1ex,color=gray] (0,0) rectangle (6.5,1.5);
\end{scope}

\begin{scope}[shift={(1,-3.1)}]
 \node()[] at (2,1.5) {queuing system ($X(t)$) };
 \draw[rounded corners=1ex,color=gray] (-1,-1) rectangle (5.5,2);
 \node(input-node)[] at (-2,0) {$\lambda(n)$};
 \node(loss-decision)[draw, shape=circle] at (0,0) {};
 \node(input-loss)[] at (0,-2) {lost};
 \node(server)[shape=circle, draw,
	  label=above:server,
  	minimum width=33pt, minimum height=30pt] at (3.5,0){$\mu(n)$};
 \node(Q)[shape=queue, draw, queue head=east,
	    label=above:queue,
		minimum width=50pt,minimum height=20pt] at (2,0) {};
 \node(output-node)[] at (6.5,0) {};
 \draw (input-node)edge[->,thick] node[]{}(loss-decision);
 \draw (loss-decision)edge[->,thick] node[]{}(Q);
 \draw (loss-decision)edge[->,thick] node[]{}(input-loss);
 \draw (server)edge[->,thick] node[]{}(output-node);
\end{scope}

\begin{scope}[shift={(2,0)}]
 \draw [->, thick, dashed, anchor=east]
  (0.25,0)--(0.25,-1)
  node[left, align=center,midway]
  {
   may \\
   stop/resume
  };
  \draw [->, thick, dashed, anchor=east]
  (2,-1)--(2,0)
  node[right,align=center,midway]
  {
   may change \\
   environment state
  };
\end{scope}
\end{tikzpicture}\caption{\label{fig:LS-loss-system}Loss system}
\end{figure}
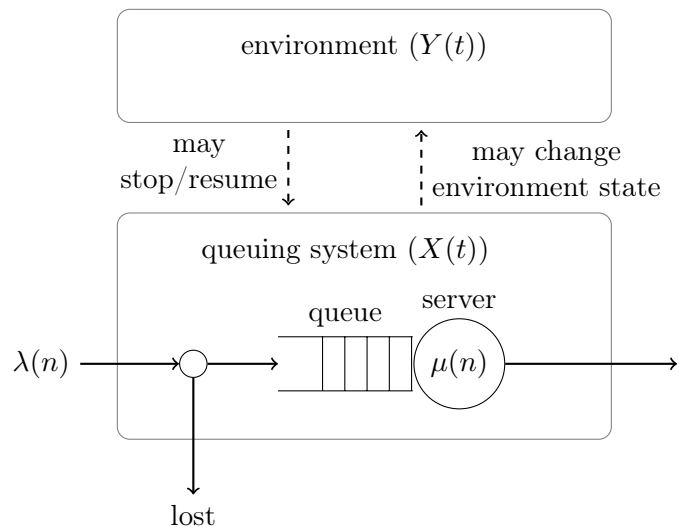

\par\end{center}

\paragraph*{Notations and conventions: }
\begin{itemize}
\item $\mathbb{R}_{0}^{+}=[0,\infty)$, $\mathbb{R}^{+}=(0,\infty)$, $\mathbb{N}$=\{1,2,3,\dots\},
$\mathbb{N}_{0}=\{0\}\cup\mathbb{N}$ 
\item All random variables and processes occurring henceforth are defined
on a common underlying probability space $(\Omega,{\cal F},P)$. 
\item For all processes considered in this paper we can and will assume
that their paths are right continuous with left limits (cadlag)\index{cadlag}. 
\item $1_{[expression]}$ is the indicator function which is $1$ if $expression$
is true and $0$ otherwise. 
\item For any quadratic matrix $V$ we define $diag(V)$ \index{$diag$}as
the matrix with the same diagonal as $V$, while all other entries
are $0$. \end{itemize}

\part{Loss systems in continuous time\label{part:LS-continous-time}}

\global\long\def\V{V}
\global\long\def\v{v}
\global\long\def\qsep{,}
\global\long\def\intpositive{>0}
\global\long\def\Rentry#1#2{R(#1,#2)}

\section{The exponential model}

\subsection{The $M/M/1/\infty$ model\label{sect:LS-MM1Inf-model} }

We consider a two-dimensional process $Z=(X,Y)=((X(t),Y(t)):t\in[0,\infty))$\marginpar{$Z=(X,Y)$}\index{Z (X,Y) two dimensional process in continous time@$Z=(X,Y)$, two
dimensional stochastic process in \foreignlanguage{icelandic}{continous}
time} with state space $E=\mathbb{N}_{0}\times K$. \index{$E$, state space}\marginpar{$E$}$K$\nomenclature{$K$}{$K$}\marginpar{$K$}\index{K enveironment states@$K$, environment states}
is a countable set, the environment space of the process, whereas
the queueing state space is $\mathbb{N}_{0}$.

We assume throughout that $Z=(X,Y)$ is non-explosive in finite times
and irreducible (unless specified otherwise).

According to our introductory example the environment space of the
process is partitioned into disjoint components $K:=K_{W}+K_{B}$.
In the framework of $K$ describing the inventory size $K_{B}$\marginpar{$K_{B}$}
\index{K B@$K_{B}$, set of the blocking environment states}describes
the status ''stock out'', in the reliability problem $K_{B}$ describes
the status ''server broken down''. So accordingly $K_{W}$ \marginpar{$K_{W}$}
\index{K w@$K_{W}$ , environment states, in which the server works.}
indicates for the inventory that there is stock on hand for production,
and ''server is up'' in the other system.

The general interpretation is that whenever the environment process
enters $K_{B}$ the service process is ''\textbf{B}LOCKED'', and the
service is resumed immediately whenever the environment process returns
to $K_{W}$, the server ''\textbf{W}ORKS'' again.

Whenever the environment process stays in $K_{B}$ new arrivals are
lost.

Obviously, it is natural to assume that the set $K_{W}$ is not empty,
while in certain frameworks $K_{B}$ may be empty, e.g. no break down
of the server in the second introductory example occurs.

The server in the system is a single server under First-Come-First-Served%
\footnote{Wer zuerst kommt, mahlt zuerst.%
} regime (FCFS)\index{FCFS} with an infinite waiting room.

The arrival stream of customers is Poisson with rate $\lambda(n)>0$,
when there are $n$ customers in the system.

The system develops over time as follows.

\textbf{1)} If the environment at time $t$ is in state $Y(t)=k\in K_{W}$
and if there are $X(t)=n$ customers in the queue then service is
provided to the customer at the head of the queue with rate $\mu(n)>0$.
The queue is organized according First-Come-First-Served regime (FCFS).
As soon as his service is finished he leaves the system and the environment
changes with probability $\Rentry km$\marginpar{$R$ - matrix}\index{R matrix@$R$ - stochastic matrix.}
to state $m\in K$, independent of the history of the system, given
$k$. We consider $R=(\Rentry km:k,m\in K)$ as a stochastic matrix
for the environment driven by the departure process.

\textbf{2)} If the environment at time $t$ is in state $Y(t)=k\in K_{B}$
no service is provided to customers in the queue and arriving customers
are lost.

\textbf{3)} Whenever the environment at time $t$ is in state $Y(t)=k\in K$
it changes with rate $\v({k,m})$ \marginpar{$\v(k,m)$} \index{v k,m@$\v(k,m)$}to
state $m\in K$, independent of the history of the system, given $k$.\\
 Note, that such changes occur independent from the service and arrival
process, while the changes of the environment's status under \textbf{1)}
are coupled with the service process.

From the above description we conclude that the non negative transition
rates of $(X,Y)$ are for $(n,k)\in E$ 
\begin{eqnarray*}
q((n,k)\qsep(n+1,k)) & = & \lambda(n),\qquad k\in K_{W},\\
q((n,k)\qsep(n-1,m)) & = & \mu(n)\Rentry km,\qquad k\in K_{W},n\intpositive,\\
q((n,k)\qsep(n,m)) & = & \v(k,m)\in\mathbb{R}_{0}^{+},~~~k\neq m,\\
q((n,k)\qsep(i,m)) & = & 0\,,\qquad\text{otherwise for}~~(n,k)\neq(i,m)\,.
\end{eqnarray*}
 Note, that the diagonal elements of $Q:=(q((n,k)\qsep(i,m)):(n,k),(i,m)\in E)$
\marginpar{$Q$} \index{Q generator matrx@$Q$, generator matrix}
are determined by the requirement that row sum is $0$.
\begin{rem}
It is allowed to have positive diagonal entries $\Rentry kk$. $R$
needs not be irreducible, there may exist closed subsets in $K$.

$\v(k,k)=-\sum_{m\in K\setminus\{k\}}\v(k,m)$ is required for all
$k\in K$ such that\marginpar{$\V$ - generator}\index{V generator matrix @$V$ - generator matrix}
\[
\V=(\v(k,m):k,m\in K)
\]
 is a generator matrix.

The Markov process associated with $\V$ may have absorbing states,
i.e., $\V$ then has zero rows. 
\end{rem}
\begin{rem}
\label{rem:LS-environment-interaction-diagram}We will visualize the
dynamics of the environment by \etid \index{etid@\etid} consisting
of colored nodes and colored arrows. We will use the following conventions:\end{rem}
\begin{itemize}
\item Square nodes describe the environment states from $K$.
\item Red nodes describe the blocking states, i.e., states from $K_{B}$.
\item Blue arrows describe possible environment changes independent from
the queueing system and correspond to the positive rates of the $V$
matrix.
\item Black arrows describe environment changes after services and correspond
to the positive entries of the $R$ matrix.
\end{itemize}
For example the diagrams \prettyref{fig:LS-r-S-example-diagram} and
\prettyref{fig:LS-r-Q-example-diagram} describe the behaviours of
two different lost-sales systems.

\subsection{Steady state distribution}

\label{sect:LS-steady-state} Our aim is to compute for an ergodic
system explicitly the steady state and limiting distribution of $(X,Y)$.
We can not expect that this will be possible in the general system
as described in Section \ref{sect:LS-MM1Inf-model}, but fortunately
enough we will be able to characterize those systems which admit a
product form equilibrium.
\begin{defn}
For a loss system $(X(t),Y(t))$ in a state space $E:=\mathbb{N}_{0}\times K$,
whose unique limiting distribution exists, we define \marginpar{$\pi$}\index{pi
, limiting distribution of the loss system in continuous time@$\pi$, limiting distribution of the loss system in continuous time} 

\begin{eqnarray*}
\pi & := & (\pi(n,k):(n,k)\in E:=\mathbb{N}_{0}\times K)\\
\pi(n,k) & := & \lim_{t\rightarrow\infty}P(X(t)=n,Y(t)=k)
\end{eqnarray*}

and the appropriate marginal limiting distributions \marginpar{$\xi$}
\index{xi
 - limiting distribution of queueing system@$\xi$ - limiting distribution of queueing system}
\begin{eqnarray*}
\xi:=(\xi(n):n\in\mathbb{N}_{0}) & \qquad\text{with}\qquad & \xi(n):=\lim_{t\rightarrow\infty}P(X(t)=n)
\end{eqnarray*}
\marginpar{$\theta$}\index{theta
 - limiting distribution of environment@$\theta$ - limiting distribution of environment}
\begin{eqnarray*}
\theta:=(\theta(k):k\in K) & \qquad\text{with}\qquad & \theta(k):=\lim_{t\rightarrow\infty}P(Y(t)=k)
\end{eqnarray*}
\end{defn}
\begin{rem}
\label{rem:LS-E-order} It will be convenient to order the state space
in the way which is common in matrix analytical investigations, where
$X$ is the level process and $Y$ is the phase process. Take on $\mathbb{N}_{0}$
the natural order and fix a total (linear) order $\preccurlyeq$ on
$K$ such that 
\begin{equation}
k\in K_{W}~\wedge~l\in K_{B}\Longrightarrow k\preccurlyeq l,\label{Eorder1}
\end{equation}
holds, and define on $E=\mathbb{N}_{0}\times K$ the lexicographic
order $\prec$ by 
\begin{equation}
(m,k),(n,l)\in E~~\text{then}~~{\big((m,k)}\prec{(n,l)}:\Longleftrightarrow\big[m<n~~\text{or}~~(m=n~~\text{and}~~k\preccurlyeq l)\big]\big)\,.\label{Eorder4}
\end{equation}

\end{rem}
\noindent \textbf{Some notation} which will be used henceforth: $I_{W}$
is a matrix which has ones on its diagonal elements $(k,k)$ with
$k\in K_{W}$ and $0$ otherwise. \index{I W@$I_{W}$}\marginpar{$I_{W}$}
That is

\[
(I_{W})_{km}=\delta_{km}1_{[k\in K_{W}]}\,,
\]
and using the ordering \eqref{Eorder1} we have the convenient notation
(which is not necessary, but makes reading more comfortable in the
proofs below) 
\[
I_{W}=\left(\begin{array}{c|cc}
 & K_{W} & K_{B}\\
\hline K_{W} & \left(\begin{array}{ccc}
1 &  & 0\\
 & \ddots\\
0 &  & 1
\end{array}\right) & 0\\
K_{B} & 0 & 0
\end{array}\right)
\]

\begin{thm}
\label{thm:LS-productform} 
\begin{itemize}
\item [(a)]Denote for $n\in\mathbb{N}_{0}$ 
\begin{eqnarray}
\tilde{Q}(n) & := & (\tilde{q}(n;k,m):k,m\in K)=\lambda(n)I_{W}(R-I)+\V\label{eq:LS-q-tilde-matrix-representation}
\end{eqnarray}
Then the matrices $\tilde{Q}(n)$ are generator matrices for some
homogeneous Markov processes and their entries are
\begin{eqnarray}
\tilde{q}(n;k,m) & = & \lambda(n)\Rentry km1_{[k\in K_{W}]}+\v(k,m)\qquad k\neq m\nonumber \\
\tilde{q}(n;k,k) & = & -(1_{[k\in K_{W}]}\lambda(n)(1-\Rentry kk)+\sum_{m\in K\backslash\{k\}}\v(k,m))\label{eq:LS-q-tilde-n-definition}
\end{eqnarray}

\item [(b)]For the process $(X,Y)$ the following properties are equivalent:

\begin{itemize}
\item [(i)]$(X,Y)$ is ergodic with product form steady state 
\begin{equation}
\pi(n,k)=\underbrace{C^{-1}\prod_{i=0}^{n-1}\frac{\lambda(i)}{\mu(i+1)}}_{=:\xi(n)}\theta(k)\,\label{eq:LS-product-form}
\end{equation}

\item [(ii)]The summability condition 
\begin{eqnarray}
C & := & \sum_{n=0}^{\infty}\prod_{i=0}^{n-1}\frac{\lambda(i)}{\mu(i+1)}<\infty\label{eq:LS-normalization1}
\end{eqnarray}
holds, and the equation 
\begin{equation}
\theta\cdot\tilde{Q}(0)=0\label{eq:LS-q-tilde-0-conditon}
\end{equation}
 admits a unique strictly positive stochastic solution $\theta=(\theta(k):k\in K)$
which solves also 
\begin{equation}
\forall n\in\mathbb{N}:\theta\cdot\tilde{Q}(n)=0\,.\label{eq:LS-q-tilde-n-conditon}
\end{equation}

\end{itemize}
\end{itemize}
\end{thm}
Before proving the theorem some remarks seem to be in order.
\begin{rem}
\label{rem:LS-BaD1}If $Z=(Z(t):t\geq0)$ is stationary, i.e., for
$t\geq0$ holds $P(Z(t)=(n,k))=\pi(n,k)$, for all $n,k\in E$, then
for any fixed time instant $t_{0}$ we have a product form distribution
\[
\pi(n,k)=\xi(n)\cdot\theta(k),\qquad(n,k)\in E
\]

Note, that this does not mean that the marginal Processes $X$ and
$Y$ are independent. Especially $X$ in general is not Markov for
its own, although its stationary distribution is identical to that
of a birth death process with birth rates $\lambda(n)$ and death
rates $\mu(n)$, i.e.,
\begin{equation}
\xi=\left(\xi(n):=C^{-1}\prod_{i=0}^{n-1}\frac{\lambda(i)}{\mu(i+1)}:n\in\mathbb{N}_{0}\right)\,,\label{eq:LS-BaD1}
\end{equation}

The observation \eqref{eq:LS-BaD1} is remarkable not only because
$X$ in general is not a birth-death process, but also because neither
the $\lambda(i)$ are the effective arrival rates (expected number
of arrivals per time unit) for queue length $i$ nor the $\mu(i+1)$
the effective service rates (expected maximal number of departures
per time unit) for queue length $i+1$. In case of a pure birth-death
process without an environment $\lambda(i)$, $\mu(i+1)$ \textbf{are}
the respective rates.

The conclusion is that both rates are diminished by the influence
of the environment by the same portion. It seems to be contra intuition
to us that the reduction of $\lambda(i)$ goes in parallel to that
of $\mu(i+1)$, while in the running system under queue length $i$
due to $Y$ entering $K_{B}$ arrivals at rate $\lambda(i)$ are interrupted
in parallel to services of rate $\mu(i)$.

The similar problem was noticed already for the case of queueing-inventory
processes with state independent service and arrival rates in Remark
2.8 in \cite{schwarz;sauer;daduna;kulik;szekli:06}, but in this setting
clearly the problem of $\lambda(i)$ versus $\mu(i+1)$ is still hidden.
\\

\end{rem}
\label{rem:LS-BaD2} The statement of the theorem can be interpreted
as a strong \textbf{insensitivity property} of the system: As long
as ergodicity is maintained, the environment can change drastically
without changing the \textbf{steady state of the queue length at any
fixed time point}. An intuitive interpretation of this result seems
to be hard. Especially, this insensitivity can not be a consequence
of the form of the control of the inventory or the availability.

We believe that there is intuitive explanation of a part of the result.
The main observation with respect to this is:

\emph{Whenever a customer is admitted to the queue, i.e. not lost,
he observes the service process as that in a conventional $M/M/1/\infty$
queue with state dependent service and arrival rates, as long as the
blocking periods are skipped over.}

\emph{Saying it the other way round, whenever the environment enters
$K_{B}$ and blocks the service process, the arrival process is blocked
as well, i.e. the queueing system is completely frozen and is revived
immediately when the environment enters $K_{W}$ next.}

Skipping the problem of $i$ versus $i+1$ discussed in \prettyref{rem:LS-BaD1}
this observation might explain the form of the marginal stationary
distribution of the customer process $X$, but it does by no means
explain the product form of limiting distribution $\pi(n,k)=\xi(n)\theta(k)$.\\

A similar observation was utilized in \cite{saffari;asmussen;haji:13}
in a queueing-inventory system (with state independent service and
arrival rates) to construct a related system which obviously has the
stationary distribution of $X$ and it is argued that from this follows
that the original system shows the same marginal queue length distribution.\\

\label{rem:LS-BaD3} The proven insensitivity does not mean, that
the time development of the queue length processes with fixed $\lambda(n)$
and $\mu(n)$ is the same under different environment behaviour. This
can be seen by considering the stationary sojourn time of admitted
customers, which is strongly dependent of the interruption time distributions
(= sojourn time distribution of $Y$ in $K_{B}$).

Similarly, multidimensional stationary probabilities for $(X(t_{1}),X(t_{2}),\dots,X(t_{n}))$
will clearly depend on the occurrence frequency of the event $(Y\in K_{B})$. 
\begin{proof}
\emph{of \prettyref{thm:LS-productform}} \textbf{(a)} Utilizing the
stochastic matrix property $R\mathbf{e}=\mathbf{e}$ and generator
property $V\mathbf{e}=0$ we get:

\begin{equation}
\tilde{Q}(n)\mathbf{e}=\left(\lambda(n)I_{W}(R-I)-\V\right)\mathbf{e}=\lambda(n)I_{W}(\underbrace{\underbrace{R\mathbf{e}}_{=\mathbf{e}}-\underbrace{I\mathbf{e}}_{=\mathbf{e}}}_{=0})+\underbrace{\V\mathbf{e}}_{=0}=0\label{eq:LS-Q-tilde-zero-sum}
\end{equation}

Using the fact that all entries of $R$ and all non-diagonal entries
of $V$ are non-negative we see that for all $k\neq m$ it holds 

\begin{eqnarray}
\left(\tilde{Q}(n)\right)_{km} & = & \left(\underbrace{\lambda(n)I_{W}R}_{\geq0}\right)_{km}-\underbrace{\left(\lambda(n)I_{W}I\right)_{km}}_{=0}+\underbrace{\left(\V\right)_{km}}_{\geq0}\geq0\,.\label{eq:LS-Q-tilde-non-negative-non-diagonal}
\end{eqnarray}

\prettyref{eq:LS-Q-tilde-zero-sum} and \prettyref{eq:LS-Q-tilde-non-negative-non-diagonal}
together show that the $\tilde{Q}(n)$ are generator matrices. The
explicit representation \prettyref{eq:LS-q-tilde-n-definition} of
the matrix $\tilde{Q}(n)$ is calculating directly.\\

\textbf{(b)} \textbf{(ii)} $\mathbf{\Rightarrow}$ \textbf{(i):} \\
 The global balance equations of the Markov process $(X,Y)$ are for
$(n,k)\in E$ 
\begin{eqnarray}
 &  & \pi(n,k)\left(1_{[k\in K_{W}]}\lambda(n)+\sum_{m\in K\backslash\{k\}}\v(k,m)+1_{[k\in K_{W}]}1_{[n\intpositive]}\mu(n)\right)\nonumber \\
 & = & \pi(n-1,k)1_{[k\in K_{W}]}1_{[n\intpositive]}\lambda(n-1)+\sum_{m\in K_{W}}\pi(n+1,m)\Rentry mk\mu(n+1)\nonumber \\
 &  & +\sum_{m\in K\backslash\{k\}}\pi(n,m)\v(m,k)\label{eq:LS-lba-general-2-2}
\end{eqnarray}

Inserting the proposed product form solution (\ref{eq:LS-product-form})
for $\pi(n,k)$ \index{pi(n,k)@$\pi(n,k)$} into the global balance
(\ref{eq:LS-lba-general-2-2}) equations, canceling $C^{-1}$ yields
\begin{eqnarray}
 &  & \theta(k)\prod_{i=0}^{n-1}\frac{\lambda(i)}{\mu(i+1)}\left(1_{[k\in K_{W}]}\lambda(n)+\sum_{m\in K\backslash\{k\}}\v(k,m)+1_{[k\in K_{W}]}1_{[n\intpositive]}\mu(n)\right)\nonumber \\
 & = & \theta(k)\prod_{i=0}^{n-2}\frac{\lambda(i)}{\mu(i+1)}1_{[k\in K_{W}]}1_{[n\intpositive]}\lambda(n-1)+\sum_{m\in K_{W}}\theta(m)\prod_{i=0}^{n}\frac{\lambda(i)}{\mu(i+1)}\Rentry mk\mu(n+1)\nonumber \\
 & + & \sum_{m\in K\backslash\{k\}}\theta(m)\prod_{i=0}^{n-1}\frac{\lambda(i)}{\mu(i+1)}\v(m,k)\,,\label{eq:LS-GB}
\end{eqnarray}
 and multiplication with $\prod_{i=0}^{n-1}\left(\frac{\lambda(i)}{\mu(i+1)}\right)^{-1}$
yields

\begin{eqnarray*}
 &  & \theta(k)\left(1_{[k\in K_{W}]}\lambda(n)+\sum_{m\in K\backslash\{k\}}\v(k,m)+1_{[k\in K_{W}]}1_{[n\intpositive]}\mu(n)\right)\\
 & = & \theta(k)\frac{\mu(n)}{\lambda(n-1)}1_{[k\in K_{W}]}1_{[n\intpositive]}\lambda(n-1)+\sum_{m\in K_{W}}\theta(m)\frac{\lambda(n)}{\mu(n+1)}\Rentry mk\mu(n+1)\\
 &  & +\sum_{m\in K\backslash\{k\}}\theta(m)\v(m,k)
\end{eqnarray*}

\begin{eqnarray*}
\Longleftrightarrow &  & \theta(k)\left(1_{[k\in K_{W}]}\lambda(n)+\sum_{m\in K\backslash\{k\}}\v(k,m)+1_{[k\in K_{W}]}1_{[n\intpositive]}\mu(n)\right)\\
 & = & \theta(k)\mu(n)1_{[k\in K_{W}]}1_{[n\intpositive]}+\sum_{m\in K_{W}}\theta(m)\lambda(n)\Rentry mk\\
 &  & +\sum_{m\in K\backslash\{k\}}\theta(m)\v(m,k)
\end{eqnarray*}

\begin{eqnarray*}
\Longleftrightarrow0 & = & -\theta(k)\left(1_{[k\in K_{W}]}\lambda(n)+\sum_{m\in K\backslash\{k\}}\v(k,m)\right)\\
 &  & +\sum_{m\in K_{W}}\theta(m)\lambda(n)\Rentry mk+\sum_{m\in K\backslash\{k\}}\theta(m)\v(m,k)
\end{eqnarray*}

\begin{eqnarray}
\Longleftrightarrow0 & = & \theta(k)\underbrace{\left\{ -\left(1_{[k\in K_{W}]}\lambda(n)(1-\Rentry kk)+\sum_{m\in K\backslash\{k\}}\v(k,m)\right)\right\} }_{=:\tilde{q}(n,k,k)}\label{eq:LS-equal-solution}\\
 &  & +\sum_{m\in K\backslash\{k\}}\theta(m)\underbrace{\left(\lambda(n)\Rentry mk1_{[m\in K_{W}]}+\v(m,k)\right)}_{=:\tilde{q}(n,m,k)}\nonumber 
\end{eqnarray}

\[
\Longleftrightarrow\theta\tilde{Q}(n)=0\,,
\]

which is (for all $n\in\mathbb{N}_{0}$) the condition \eqref{eq:LS-q-tilde-0-conditon}
and \eqref{eq:LS-q-tilde-n-conditon}.

By assumption (\ref{eq:LS-q-tilde-0-conditon}) there exists a stochastic
solution to $\theta\cdot\tilde{Q}(0)=0$, which according to requirement
(\ref{eq:LS-q-tilde-n-conditon}) is a solution of $\theta\tilde{Q}(n)=0$.

Setting $\theta:=\theta_{0}$ in \eqref{eq:LS-GB} provides a solution
of the global balance equations \eqref{eq:LS-lba-general-2-2}. Therefore,
the steady state equations of $(X,Y)$ admit a stochastic solution,
and so $(X,Y)$ is ergodic and we have identified the unique stochastic
solution of \eqref{eq:LS-lba-general-2-2}.\\

\textbf{(b)} \textbf{(i)} $\mathbf{\Rightarrow}$ \textbf{(ii):} \\
 Because $\pi$ is stochastic, summability \eqref{eq:LS-normalization1}
holds. Insert the stochastic vector of product form \eqref{eq:LS-product-form}
into \eqref{eq:LS-lba-general-2-2}. As shown in the part \textbf{(ii)}
$\mathbf{\Rightarrow}$ \textbf{(i)} of the proof this leads to \eqref{eq:LS-equal-solution}
and we have found a solution of \eqref{eq:LS-q-tilde-0-conditon}
which solves \eqref{eq:LS-q-tilde-n-conditon} for all $n\in\mathbb{N}$
as well. \end{proof}
\begin{cor}
\label{cor:LS-lambda-constant} If in the framework of \prettyref{thm:LS-productform}
the arrival stream is a Poisson-$\lambda$ stream (which is interrupted
when the environment process stays in $K_{B}$) then the stationary
distribution in case of ergodic $(X,Y)$ is of product form 
\begin{equation}
\pi(n,k)=C^{-1}\frac{\lambda^{n}}{\prod_{i=0}^{n-1}\mu(i+1)}\theta(k)\quad(n,k)\in E,\label{eq:LS-constant-lambda-product-form}
\end{equation}
 with normalization constant $C$. 

$\theta$ is the the unique stochastic strictly positive solution
of the equation \index{Q @ $\tilde{Q}$}\marginpar{$\tilde{Q}$}

\begin{equation}
\theta\underbrace{\left(\lambda\left(R_{W}-I_{W}\right)+\V\right)}_{=:\tilde{Q}}=0\,.\label{eq:LS-constant-lambda-theta-matrix-equation}
\end{equation}

with 
\begin{equation}
R_{W}:=I_{W}\cdot R=\left(\begin{array}{c|cc}
 & K_{W} & K_{B}\\
\hline K_{W} & R|_{K_{W}\times K_{W}} & R|_{K_{W}\times K_{B}}\\
K_{B} & 0 & 0
\end{array}\right)
\label{eq:LS-IWR}
\end{equation}

the matrix with $K_{W}$-rows the rows of $R$ and with $K_{B}$-rows
with only zeros.\end{cor}
\begin{proof}
Because of $\lambda(n)=\lambda$ for all $n$ holds $Q(0)=Q(n)$ and
the condition \eqref{eq:LS-q-tilde-n-conditon} is trivially valid.
Equation \eqref{eq:LS-constant-lambda-theta-matrix-equation} is the
condition \eqref{eq:LS-q-tilde-0-conditon} expressed via matrix representation
\eqref{eq:LS-q-tilde-matrix-representation} of $\tilde{Q(0)}$. \end{proof}
\begin{cor}
\label{cor:LS-PF-K-finite} If in the framework of \prettyref{thm:LS-productform}
the environment state space $K$ is finite, then the equations (\ref{eq:LS-q-tilde-0-conditon})
and (\ref{eq:LS-q-tilde-n-conditon}) always admit stochastic solutions,
and the stationary distribution of $(X,Y)$ is of product form 
\begin{equation}
\pi(n,k)=C^{-1}\prod_{i=0}^{n-1}\frac{\lambda(i)}{\mu(i+1)}\theta(k)\,,\label{eq:LS-finite-K-product-form}
\end{equation}
 whenever \eqref{eq:LS-q-tilde-0-conditon} and \eqref{eq:LS-q-tilde-n-conditon}
have a common solution. \end{cor}
\begin{rem}
\label{rem:LS-BaD4} The proof of Theorem \ref{thm:LS-productform}
reveals that the solution of the equation $\theta\cdot\tilde{Q}(0)=0$
(see \eqref{eq:LS-q-tilde-0-conditon}) does not depend on the values
$\mu(n)$. So, changing the service capacity of the queueing system
will not change the steady state of the environment, as long as the
system remains stable (ergodic). 
\end{rem}
The next examples comment on different forms of establishing product
form equilibrium which may arise in the realm of Theorem \ref{thm:LS-productform}.
\begin{example}
\label{ex:LS-non-constant-lambda} There exist non trivial loss systems
with non constant (i.e., state dependent) arrival rates $\lambda(n)$
in a random environment which have a product form steady state distribution.
This is verified by the following example. We have environment 
\[
K=\{1,2\}=K_{W}\,,
\]
 and for some $c,d\in(0,1)$ the routing matrix 
\[
R=\left(\begin{array}{cc}
1-c & c\\
d & 1-d
\end{array}\right)\,,
\]
 whereas $\V$ is the matrix of only zeros. It follows

\[
\tilde{Q}(n)=\left(\begin{array}{cc}
-\lambda(n)c & \lambda(n)c\\
\lambda(n)d & -\lambda(n)d
\end{array}\right)=\lambda(n)\left(\begin{array}{cc}
-c & c\\
d & -d
\end{array}\right)
\]
 It follows that $Q(n+1)=c_{n+1}Q(n)$ holds for suitable $c_{n+1},n\in\mathbb{N}_{0}$,
which immediately shows that \eqref{eq:LS-q-tilde-0-conditon} and
\eqref{eq:LS-q-tilde-n-conditon} have a common solution, which is
\[
\theta=\left(\frac{d}{d+c},\frac{c}{d+c}\right)\,.
\]
 
\end{example}
\begin{example}
\label{ex:LS-constant-lambda} There exist non trivial ergodic loss
systems in a random environment which have a product form steady state
distribution if and only if the arrival rates are independent of the
queue lengths, i.e. $\lambda(n)\equiv\lambda$ . This is verified
by the following example (which describes a queueing-inventory system
under $(r,S)$ policy with $(r=1,S=2)$, as will be seen in Section
\ref{sect:LS-inventories}, Definition \ref{def:LS-open}). We have
an environment 
\[
K=\{0,1,2\},\quad\text{with blocking set}~~K_{B}=\{0\}\,,
\]
 stochastic matrix $R$ and and the generator matrix $\V$ given as
\[
R=\left(\begin{array}{ccc}
1 & 0 & 0\\
1 & 0 & 0\\
0 & 1 & 0
\end{array}\right)\,,\qquad\V=\left(\begin{array}{ccc}
-\nu & 0 & \nu\\
0 & -\nu & \nu\\
0 & 0 & 0
\end{array}\right)\,.
\]

It follows 
\[
\tilde{Q}(n)=\left(\begin{array}{ccc}
-\nu & 0 & \nu\\
\lambda(n) & -(\lambda(n)+\nu) & \nu\\
0 & \lambda(n) & -\lambda(n)
\end{array}\right)
\]
 Clearly, if $\lambda(n)\equiv\lambda$ are equal, the equations 
\begin{equation}
\theta\cdot\tilde{Q}(n)=0,\quad n\in\mathbb{N}_{0}\,,\label{eq:LS-lambda-const-thata-q-tilde-n-eq-0}
\end{equation}
 have a common stochastic solution.\\

On the other hand, the solutions of \eqref{eq:LS-lambda-const-thata-q-tilde-n-eq-0}
are
\begin{equation}
\theta_{n}=(\theta_{n}(0),\theta_{n}(1),\theta_{n}(2))=C_{n}^{-1}\left(\frac{\lambda(n)}{\nu},1,\frac{\lambda(n)+\nu}{\lambda(n)}\right)\,,\quad n\in\mathbb{N}_{0}\,.\label{eq:LS-lambda-const-thata-q-tilde-n-eq-0-solution}
\end{equation}
 We conclude 
\[
\forall n\in\mathbb{N}_{0}:\theta_{n}=\theta_{n+1}\Longrightarrow\frac{\theta_{n}(0)}{\theta_{n}(1)}=\frac{\theta_{n+1}(0)}{\theta_{n+1}(1)}\Longleftrightarrow\lambda(n)=\lambda(n+1)\,.
\]
\end{example}
\begin{rem}
In Section \ref{sect:LS-finite-capacity} we will show in the course
of proving a companion of Theorem \ref{thm:LS-productform} for loss
systems with finite waiting room that more restrictive conditions
on the environment are needed. It turns out that the construction
in the proof of the Theorem \ref{thm:LS-productform-finite} will
provide us with more general constructions for examples as those given
here, see Remark \ref{rem:LS-another-example1} below.
\end{rem}

\subsection{Finite capacity loss systems\label{sect:LS-finite-capacity} }

In this section we study the systems from Section \ref{sect:LS-MM1Inf-model}
under the additional restriction that the capacity of the waiting
room is finite. That is, we now consider loss systems in the traditional
sense with the additional feature of losses due to the environment's
restrictions on customers' admission and service

Recall, that for the pure exponential single server queueing system
with state dependent rates and $N\geq0$ waiting places the state
space is $E=\{0,1,\dots,N,N+1\}$ and the queueing process $X$ is
ergodic with stationary distribution $\pi=(\pi(n):n\in E$ of the
form 
\begin{equation}
\pi(n)=C^{-1}\prod_{i=0}^{n-1}\frac{\lambda(i)}{\mu(i+1)}\,,\quad n\in E\,.\label{eq:LS-MM1N-1}
\end{equation}
 If the queueing system with infinite waiting room and the same rates
$\lambda(i)$, $\mu(i)$ is ergodic, the stationary distribution $\pi$
in \eqref{eq:LS-MM1N-1} is simply obtained by conditioning the stationary
distribution of this infinite system onto $E$. (Note, that ergodicity
in the finite waiting room case is granted by free, without referring
to the infinite system.)

We will show, that a similar construction by conditioning is in general
not possible for the loss system in a random environment. The structure
of the environment process will play a crucial role for enabling such
a conditioning procedure.

We take the interaction between the queue length process $X$ and
the environment process $Y$ of the same form as in Section \ref{sect:LS-MM1Inf-model},
with $R$ and $\V$ of the same form, and $\lambda(i)>0$ for $i=0,\dots,N$,
and $\mu(i)>0$ for $i=1,\dots,N+1$. The state space is $E:=\{0,\dots,N+1\}\times K$.
The non negative transition rates of $(X,Y)$ are for $(n,k)\in E$
\begin{eqnarray*}
q((n,k)\qsep(n+1,k)) & = & \lambda(n)\qquad k\in K_{W},n<N+1\\
q((n,k)\qsep(n-1,m)) & = & \mu(n)\Rentry km\qquad k\in K_{W},n\intpositive\\
q((n,k)\qsep(n,m)) & = & \v(k,m)\in\mathbb{R}_{0}^{+},~~~k\neq m\\
q((n,k)\qsep(i,m)) & = & 0\qquad\text{otherwise for}~~(n,k)\neq(i,m)\in E
\end{eqnarray*}

The first step of the investigation is nevertheless completely parallel
to Theorem \ref{thm:LS-productform}. 
\begin{thm}
\label{thm:LS-productform-finite} 
\begin{itemize}
\item [(a)]\textbf{ }Denote for $n\in\{0,\dots,N+1\}$ 
\begin{eqnarray}
\tilde{q}(n,k,k) & = & -(1_{[k\in K_{W}]}\cdot1_{[n\in\{0,\dots,N\}]}\lambda(n)(1-\Rentry kk)+\sum_{m\in K\backslash\{k\}}\v(k,m))\nonumber \\
\tilde{q}(n,k,m) & = & \lambda(n)\Rentry km1_{[k\in K_{W}]}\cdot1_{[n\in\{0,\dots,N\}]}+\v(k,m)\qquad k\neq m\label{eq:LS-q-tilde-n-definition-finite}
\end{eqnarray}
 and 
\[
\tilde{Q}(n)=(\tilde{q}(n,k,m):k,m\in K)\,.
\]
 Then the matrices $\tilde{Q}(n)$ are generator matrices for some
homogeneous Markov processes.
\item [(b)]\textbf{ }If the process $(X,Y)$ is ergodic denote its unique
steady state distribution by 
\[
\pi=(\pi(n,k):(n,k)\in E:=\{0,\dots,N+1\}\times K).
\]
Then the following three properties are equivalent:

\begin{itemize}
\item [(i)] $(X,Y)$ is ergodic on $E$ with product form steady state
\begin{equation}
\pi(n,k)={C^{-1}\prod_{i=0}^{n-1}\frac{\lambda(i)}{\mu(i+1)}}\theta(k)\,\quad n\in\{0,\dots,N+1\},k\in K\label{eq:LS-product-form-finite}
\end{equation}

\item [(ii)] The equation 
\begin{equation}
\theta\cdot\tilde{Q}(0)=0\label{eq:LS-q-tilde-0-conditon-finite}
\end{equation}
 admits a strict positive stochastic solution $\theta=(\theta(k):k\in K)$
which solves also 
\begin{equation}
\forall n\in\{0,\dots,N+1\}:\theta\cdot\tilde{Q}(n)=0\,.\label{eq:LS-q-tilde-n-conditon-finite}
\end{equation}

\item [(iii)] The equation 
\begin{equation}
\eta\cdot\V=0\label{eq:LS-upsilon-finite}
\end{equation}
 admits a strict positive stochastic solution.\\
The set $K_{W}\subseteq K$ is a closed set for the Markov chain on
state space $K$ with transition matrix $R$, i.e., 
\[
\forall k\in K_{W}:\quad\sum_{m\in K_{W}}\Rentry km=1,
\]
 and the restriction $\eta^{(W)}:=(\eta(m):m\in K_{W})$ of $\eta$
to $K_{W}$ solves the equation 
\begin{equation}
\eta^{(W)}=\eta^{(W)}\cdot R^{(W)},\label{eq:LS-R-W-finite}
\end{equation}
 where 
\begin{equation}
R^{(W)}:=(\Rentry km:k,m\in K_{W})\label{eq:LS-R-W-finite-1}
\end{equation}
 is the restriction of $R$ to $K_{W}$. 
\end{itemize}
\end{itemize}
\end{thm}
\begin{proof}
The proof of \textbf{(a)} is similar to that of Theorem \ref{thm:LS-productform}\textbf{(a)},
and in \textbf{(b)} the equivalence of \textbf{(i)} and \textbf{(ii)}
is proven in almost identical way as that of Theorem \ref{thm:LS-productform}
\textbf{(b)} (with the obvious slight changes due to having the $X$-component
finite) and are therefore omitted.\\

We next show

\textbf{(b)} \textbf{(i)} $\mathbf{\Rightarrow}$ \textbf{(iii):}
\\
 The global balance equations of the Markov process $(X,Y)$ are for
$(n,k)\in E$ 
\begin{eqnarray}
 &  & \pi(n,k)\left(1_{[k\in K_{W}]}\cdot1_{[n\in\{0,\dots,N\}]}\lambda(n)+\sum_{m\in K\backslash\{k\}}\v(k,m)+1_{[k\in K_{W}]}1_{[n\intpositive]}\mu(n)\right)\nonumber \\
 & = & \pi(n-1,k)1_{[k\in K_{W}]}1_{[n\intpositive]}\lambda(n-1)+\sum_{m\in K_{W}}\pi(n+1,m)\Rentry mk\mu(n+1)\cdot1_{[n\in\{0,\dots,N\}]}\nonumber \\
 &  & +\sum_{m\in K\backslash\{k\}}\pi(n,m)\v(m,k)\label{eq:LS-lba-general-2-2-finite}
\end{eqnarray}

Inserting the proposed product form solution (\ref{eq:LS-product-form-finite})
for $\pi(n,k)$ \index{pi(n,k)@$\pi(n,k)$} into the global balance
equations (\ref{eq:LS-lba-general-2-2-finite}) and proceeding in
the same way as in the proof of Theorem \ref{thm:LS-productform}
yields 

\begin{eqnarray}
0 & = & -\theta(k)\left(1_{[k\in K_{W}]}\cdot1_{[n\in\{0,\dots,N\}]}\lambda(n)+\sum_{m\in K\backslash\{k\}}\v(k,m)\right)\label{eq:LS-equal-solution-1-finite}\\
 &  & +\sum_{m\in K_{W}}\theta(m)\lambda(n)\Rentry mk\cdot1_{[n\in\{0,\dots,N\}]}+\sum_{m\in K\backslash\{k\}}\theta(m)\v(m,k)\nonumber 
\end{eqnarray}

For $n\to N+1$ \eqref{eq:LS-equal-solution-1-finite} turns to 
\begin{eqnarray}
\theta(k)\left(\sum_{m\in K\backslash\{k\}}\v(k,m)\right) & = & \sum_{m\in K\backslash\{k\}}\theta(m)\v(m,k)\,,\label{eq:LS-equal-solution-2-finite}
\end{eqnarray}
 which verifies \eqref{eq:LS-upsilon-finite} with $\eta:=\theta$.

For $n<N+1$ \eqref{eq:LS-equal-solution-1-finite} turns to

\begin{eqnarray*}
\theta(k)1_{[k\in K_{W}]}\lambda(n)+\underbrace{\theta(k)\sum_{m\in K\backslash\{k\}}\v(k,m)}_{(*)}=\sum_{m\in K_{W}}\theta(m)\lambda(n)\Rentry mk+\underbrace{\sum_{m\in K\backslash\{k\}}\theta(m)\v(m,k)}_{(**)}\,,
\end{eqnarray*}
 where from \eqref{eq:LS-equal-solution-2-finite} the expressions
$(**)$ and $(*)$ cancel and we arrive at 
\begin{eqnarray}
\theta(k)1_{[k\in K_{W}]}\lambda(n) & = & \sum_{m\in K_{W}}\theta(m)\lambda(n)\Rentry mk\,.\label{eq:LS-equal-solution-3-finite}
\end{eqnarray}
 Because $(X,Y)$ is ergodic, $\theta$ is strictly positive, and
we conclude (set $k\in K_{B}$ in \eqref{eq:LS-equal-solution-3-finite}
which makes the left side zero) 
\[
\Rentry mk=0,\quad\forall m\in K_{W},~k\in K_{B}\,,
\]
 which shows that $K_{W}$ is a closed set for the Markov chain governed
by $R$.

Now set $k\in K_{W}$ in \eqref{eq:LS-equal-solution-3-finite} which
makes the left side strictly positive and realize that this after
canceling $\lambda(n)$ is exactly \eqref{eq:LS-R-W-finite}.\\
 This part of the proof is finished.\\

\textbf{(b)} \textbf{(iii)} $\mathbf{\Rightarrow}$ \textbf{(ii):}
\\

For proving the reversed direction we reconsider the previous part
\textbf{(i)} $\mathbf{\Rightarrow}$ \textbf{(iii)} of the proof:
The strict positive stochastic solution of 
\begin{equation}
\eta\cdot\V=0\,,\label{eq:LS-V-finite-2}
\end{equation}
 which is given by assumption \eqref{eq:LS-upsilon-finite}, yields
the required solution for $n\to N+1$ of 
\[
\theta\cdot\tilde{Q}(N+1)=0\,.
\]

If $K_{W}\subseteq K$ is a closed set for the Markov chain on state
space $K$ with transition matrix $R$ we obtain 
\[
\Rentry mk=0,\quad\forall m\in K_{W},~k\in K_{B}\,,
\]
 and therefore for all $n\in\{0,1,\dots,N\}$ 
\[
\theta\cdot\tilde{Q}(n)=0
\]
 reduces for $k\in K_{B}$ to the respective expression in 
\[
\eta\cdot\V=0\,.
\]
 It remains for all $n\in\{0,1,\dots,N\}$ and for $k\in K_{W}$ to
show that for $k\in K_{W}$ the respective expression in 
\[
\theta\cdot\tilde{Q}(n)=0
\]
 is valid. This follows by considering 
\begin{eqnarray*}
\eta(k)1_{[k\in K_{W}]}\lambda(n)+\underbrace{\eta(k)\sum_{m\in K\backslash\{k\}}\v(k,m)}_{(*)}=\sum_{m\in K_{W}}\eta(m)\lambda(n)\Rentry mk+\underbrace{\sum_{m\in K\backslash\{k\}}\eta(m)\v(m,k)}_{(**)}\,,
\end{eqnarray*}
 and remembering that the expressions $(**)$ and $(*)$ cancel. The
residual terms are equal by the assumption \eqref{eq:LS-R-W-finite}.\\
 This finishes the proof. 
\end{proof}
The interesting insight is that from the existence of the product
form steady state $\pi$ on $E=\{0,\dots,N+1\}\times K$ implicitly
restrictions on the form of the movements of the environment emerge
which are not necessary in the case of infinite waiting rooms. (As
indicated above, such restrictions are not necessary too in the pure
queueing system framework.)

The proof of Theorem \ref{thm:LS-productform-finite} has brought
out the following additional, somewhat surprising, insensitivity property. 
\begin{cor}
\label{cor:LS-productform-finite} Whenever $(X,Y)$ is ergodic with
product form steady state 
\[
\pi(n,k)={C^{-1}\prod_{i=0}^{n-1}\frac{\lambda(i)}{\mu(i+1)}}\theta(k)\,\quad n\in\{0,\dots,N+1\},k\in K
\]
 for some (positive) parameter setting $({\lambda(i)}:i=0,1,\dots,N)$,
$({\mu(i)}:i=1,\dots,N+1)$ with an environment characterized by $(K,K_{B},\V,R)$,
then for this same environment $(X,Y)$ is ergodic with product form
steady state with the same $\theta$ for any (positive) parameter
setting for the arrival and service rates. \end{cor}
\begin{proof}
This becomes obvious at the step where we arrived at \eqref{eq:LS-equal-solution-3-finite}
and we see that the specific shape of the sequence of the ${\lambda(i)}$
do not matter. The specific $\mu(i)$ are canceled in the early steps
of the proof already. \end{proof}
\begin{example}
\label{ex:LS-productform-finite-1} We describe a class of examples
of environments which guarantee that the conditions of Theorem \ref{thm:LS-productform-finite}
are fulfilled. The construction is in three steps. 

Take for $\V$ a generator of an irreducible Markov process on $K$
with stationary distribution $\theta$, which fulfills for all $k\in K_{W}$
the \emph{partial balance condition} 
\begin{equation}
\theta(k)\sum_{m\in K_{W}}\v(k,m)=\sum_{m\in K_{W}}\theta(m)\v(m,k)\label{eq:LS-PB1}
\end{equation}
 and $\sup(-\v(k,k):k\in K_{W})<\infty.$ 

Denote by $\V^{(W)}$ the restriction of $\V$ onto $K_{W}$ which
has stationary distribution $\theta^{(W)}:=(\theta(k)/(\sum_{m\in K_{W}}\theta(m)):k\in K_{W})$,
see \cite{kelly:79}{[}Exercise 1.6.2, p. 27{]}. 

Take for $R^{(W)}$ (see \eqref{eq:LS-R-W-finite-1}) a uniformization
chain of $\V^{(W)}$, see \cite{keilson:79}{[}Chapter 2, Section
2.1{]}, e.g., (with $I$ the identity matrix on $K_{W}$) 
\[
R^{(W)}:=I+{\sup(-\v(k,k):k\in K_{W})}^{-1}\V^{(W)}\,,
\]
 which is stochastic and has equilibrium distribution $\theta^{(W)}:=(\theta(k)/(\sum_{m\in K_{W}}\theta(m)):k\in K_{W})$
as well.\\
 $(\Rentry km:k\in K_{B},m\in K)$ can be arbitrarily selected, e.g.
the identity matrix on $K_{B}$. 

This construction ensures that the restriction $\eta^{(W)}:=(\eta(m):m\in K_{W})$
of $\eta$ to $K_{W}$ solves the equation \eqref{eq:LS-R-W-finite}
\[
\eta^{(W)}=\eta^{(W)}\cdot R^{(W)}.
\]
\end{example}
\begin{rem}
The construction in Example \ref{ex:LS-productform-finite-1} may
seem to produce a narrow class of examples, but this is not so: All
reversible $\V$ fulfill the partial balance condition \eqref{eq:LS-PB1}. 
\end{rem}
\begin{rem}
\label{rem:LS-another-example1} The construction above produces another
example contributing to the discussion at the end of Section \ref{sect:LS-steady-state}
on the question which particular product forms can occur, and which
form of the environment and the arrival and service rate patterns
may interact to result in product form equilibrium for loss systems
with infinite waiting room. We only have to notice that the equations
for $n<N+1$ are exactly those which occur for all $n\in\mathbb{N}_{0}$
in the setting of Theorem \ref{thm:LS-productform}.

The cautious reader will already have noticed that the conditions
in \textbf{(b)(iii)} of Theorem \ref{thm:LS-productform-finite} provide
a similar more abstract example for the discussion on Theorem \ref{thm:LS-productform}
at the end of Section \ref{sect:LS-steady-state}. 
\end{rem}
\begin{rem}
\label{rem:LS-finite-queue-lost-sales1} We should point out that
in \cite{schwarz;sauer;daduna;kulik;szekli:06}{[}Section 6{]} queueing-inventory
models with finite waiting room are investigated with a resulting
''quasi product form'' steady state. The respective theorems there
do not fit into the realm of \prettyref{thm:LS-productform-finite}
because the state space is \textbf{not} a product space as in \prettyref{thm:LS-productform-finite},
where we have irreducibility on $E=\{0,1,\dots,N,N+1\}$.

The difference is that in \cite{schwarz;sauer;daduna;kulik;szekli:06}{[}Section
6{]} the element (in notation of the present paper) $(N+1,0)$ is
not a feasible state.

The results there can be considered as a truncation property of the
equilibrium of the system with infinite waiting room onto the feasible
state space under restriction to finite queues. 
\end{rem}

\section{Applications\label{sect:LS-applications}}

\subsection{Inventory models}

\label{sect:LS-inventories} In the following we describe an $M/M/1/\infty$-system
with inventory management as it is investigated in \cite{schwarz;sauer;daduna;kulik;szekli:06}. 
\begin{defn}
\label{def:LS-open} An M/M/1/$\infty$-system with inventory management
is a single server with infinite waiting room under FCFS regime and
an attached inventory.

There is a Poisson-$\lambda$-arrival stream, $\lambda>0$. Customers
request for an amount of service time which is exponentially distributed
with mean 1. Service is provided with intensity $\mu>0$. \\
 The server needs for each customer exactly one item from the attached
inventory. The on-hand inventory decreases by one at the moment of
service completion. If the inventory is decreased to the reorder point
$r\geq0$ after the service of a customer is completed, a replenishment
order is instantaneously triggered. The replenishment lead times are
i.i.d. with distribution function $B=(B(t);t\geq0)$. The size of
the replenishment depends on the policy applied to the system. We
consider two standard policies from inventory management, which lead
to an M/M/1/$\infty$-system with either $(r,Q)$-policy (size of
the replenishment order is always $Q>r$) or with $(r,S)$-policy
(replenishment fills the inventory up to maximal inventory size $S>r$).

During the time the inventory is depleted and the server waits for
a replenishment order to arrive, no customers are admitted to join
the queue (\textquotedbl{}lost sales\textquotedbl{}).\\
 All service, interarrival and lead times are assumed to be independent.\\
 Let $X(t)$ denote the number of customers present at the server
at time $t\geq0$, either waiting or in service (queue length) and
let $Y(t)$ denote the on-hand inventory at time $t\geq0$. Then $((X(t),Y(t)),t\geq0)$,
the \emph{queueing-inventory process} is a continuous-time Markov
process for the M/M/1/$\infty$-system with inventory management.
The state space of $(X,Y)$ is $E=\{(n,k):n\in\mathbb{N}_{0},k\in K\},$
where $K=\mathbb{N}_{0}$ or $K=\{0,1,\dots,\kappa\}$, where $\kappa<\infty$
is the maximal size of the inventory at hand. 
\end{defn}
The system described above generalizes the lost sales case of classical
inventory management where customer demand is not backordered but
lost in case there is no inventory on hand (see Tersine \cite{tersine:94}
p. 207). \\

The general \prettyref{thm:LS-productform} produces as special applications
the following results on product form steady states in integrated
queueing inventory systems from \cite{schwarz;sauer;daduna;kulik;szekli:06}.

\begin{center}
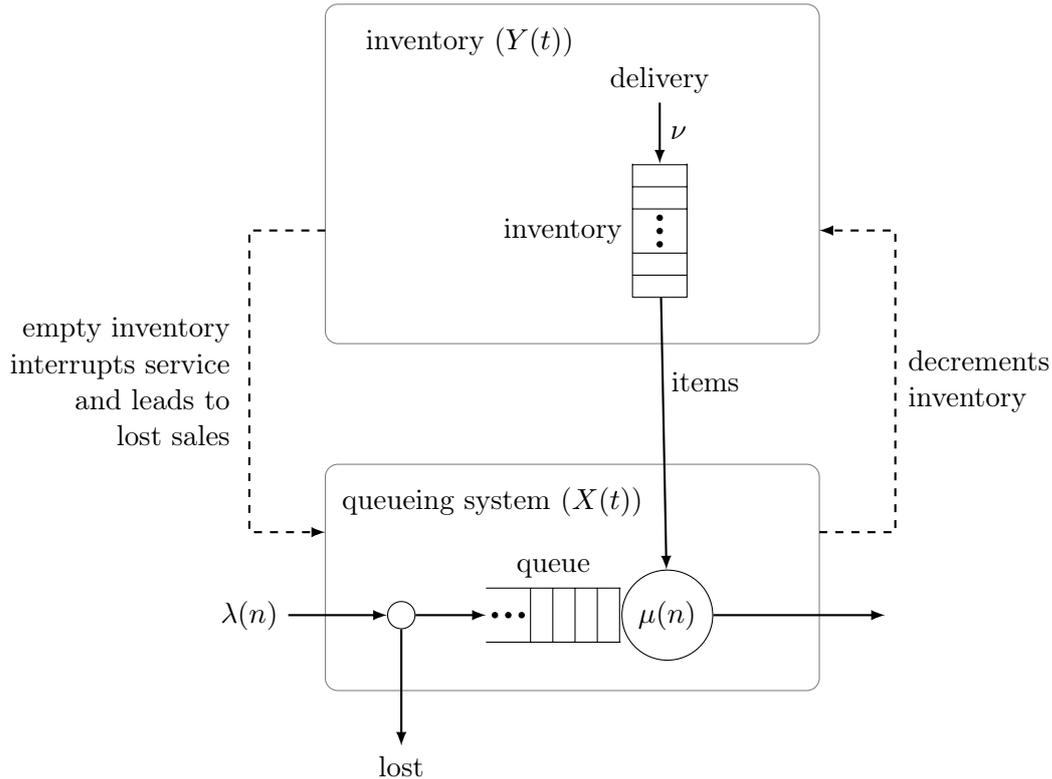
\begin{figure}
\begin{tikzpicture}
\begin{scope}[shift={(1,-3.1)}]
  \node()[] at (1.2,1.5) {queueing system ($X(t)$) };
  \draw[rounded corners=1ex,color=gray] (-1,-1) rectangle (5.5,2);
  \node(InputNode)[] at (-2,0) {$\lambda(n)$};
  \node(LossDecision)[draw, shape=circle] at (0,0) {};
  \node(InputLoss)[] at (0,-2) {lost};
  \node(Server)[shape=circle, draw,
  	minimum width=33pt, minimum height=30pt] at (3.5,0){$\mu(n)$};
  \node(Q)[shape=queue, draw, queue head=east, queue ellipsis=true,
	    label=above:queue,
		minimum width=50pt,minimum height=20pt] at (2,0) {};
  \node(OutputNode)[] at (6.5,0) {};
  \draw (InputNode)edge[arrows={-latex},thick] node[]{}(LossDecision);
  \draw[arrows={-latex}, thick]
    (LossDecision) -- (Q);
  \draw[arrows={-latex}, thick]
    (LossDecision) -- (InputLoss);
  \draw[arrows={-latex}, thick]
    (Server) -- (OutputNode);
\end{scope}

\begin{scope}[shift={(4.4,2)}]
  \node()[] at (-2.5,2.5) {inventory ($Y(t)$) };
  \draw[rounded corners=1ex,color=gray] (-4.4,-1.5) rectangle (2.1,3);
   \node(InventoryInputNode)[] at (0,2) {delivery}; 
   \node(InventoryQ)[shape=queue, draw, queue head=south, queue size=bounded,
	    label=left:inventory,
		minimum width=20pt,minimum height=50pt] at (0,0) {};
  \draw[arrows={-latex}, thick]
    (InventoryInputNode) -- node[auto]{$\nu$}(InventoryQ);

\end{scope}

  \draw[arrows={-latex}, thick, near start]
    (InventoryQ) --(Server);
  \node()[] at (5,0) {items};
  \draw[arrows={-latex}, thick, dashed]
    (6.5,-2) --(7.5,-2)--(7.5,2)--(6.5,2);
  \node()[align=left] at (8.6,0){decrements \\
    inventory};
  \draw[arrows={-latex}, thick, dashed]
    (0,2) --(-1,2)--(-1,-2)--(0,-2);
  \node()[align=right] at (-2.7,0){empty inventory\\
    interrupts service  \\ and leads to \\ lost sales};

\end{tikzpicture}\caption{$M/M/1/\infty$ inventory model with lost sales.}

\end{figure}

\par\end{center}

\begin{center}
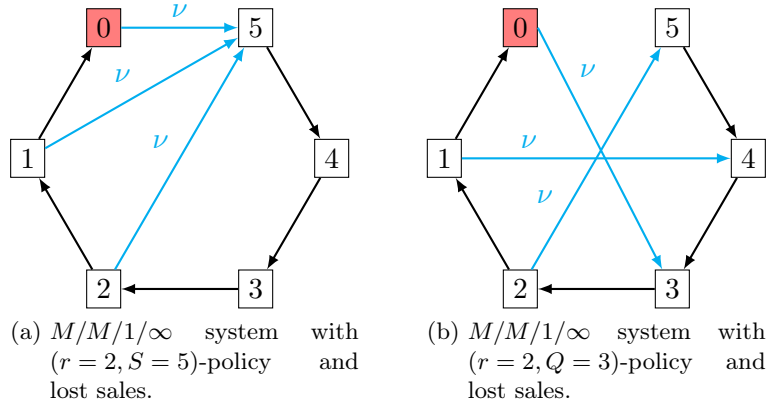
\begin{figure}
\begin{centering}
\subfloat[\label{fig:LS-r-S-example-diagram}$M/M/1/\infty$ system with ${(r=2,S=5)}$-policy
and lost sales.]{\begin{tikzpicture}
  \path
	(60:2) node (Y5)[shape=rectangle,draw] {$5$}
	(0:2) node (Y4)[shape=rectangle,draw] {$4$}
	(-60:2) node (Y3)[shape=rectangle,draw] {$3$}
	(-120:2) node (Y2)[shape=rectangle,draw] {$2$}
	(-180:2) node (Y1)[shape=rectangle,draw] {$1$}
	(-240:2) node (Y0)[shape=rectangle,draw,fill=blocked.bg] {$0$};
  \draw[arrows={-latex}, thick]
    (Y5) -- (Y4);
  \draw[arrows={-latex}, thick]
    (Y4) -- (Y3);
  \draw[arrows={-latex}, thick]
    (Y3) -- (Y2);
  \draw[arrows={-latex}, thick]
    (Y2) -- (Y1);
  \draw[arrows={-latex}, thick]
    (Y1) -- (Y0);
  \draw[arrows={-latex}, thick, color=vratecolor]
    (Y0) -- node[auto]{$\nu$}(Y5);
  \draw[arrows={-latex}, thick, color=vratecolor]
    (Y1) -- node[auto]{$\nu$}(Y5);
  \draw[arrows={-latex}, thick, color=vratecolor]
    (Y2) -- node[auto]{$\nu$}(Y5);
\end{tikzpicture}

}~~~~~~~\subfloat[\label{fig:LS-r-Q-example-diagram}$M/M/1/\infty$ system with ${(r=2,Q=3)}$-policy
and lost sales.]{\begin{tikzpicture}
  \path 	(60:2) node (Y5)[shape=rectangle,draw] {$5$}
	(0:2) node (Y4)[shape=rectangle,draw] {$4$}
	(-60:2) node (Y3)[shape=rectangle,draw] {$3$}
	(-120:2) node (Y2)[shape=rectangle,draw] {$2$}
	(-180:2) node (Y1)[shape=rectangle,draw] {$1$}
	(-240:2) node (Y0)[shape=rectangle,draw,fill=blocked.bg] {$0$};
  \draw[arrows={-latex}, thick]
    (Y5) -- (Y4);
  \draw[arrows={-latex}, thick]
    (Y4) -- (Y3);
  \draw[arrows={-latex}, thick]
    (Y3) -- (Y2);
  \draw[arrows={-latex}, thick]
    (Y2) -- (Y1);
  \draw[arrows={-latex}, thick]
    (Y1) -- (Y0);
  \draw[arrows={-latex},near start, thick, color=vratecolor]
    (Y0.east) -- node[auto]{$\nu$}(Y3);
  \draw[arrows={-latex},near start, thick, color=vratecolor]
    (Y1) -- node[auto]{$\nu$}(Y4);
  \draw[arrows={-latex},near start, thick, color=vratecolor]
    (Y2) -- node[auto]{$\nu$}(Y5);
\end{tikzpicture}}\caption{{\Etid} for lost sales environment systems. The environment process
counts the number of items in inventory.}

\par\end{centering}

\end{figure}

\par\end{center}
\begin{example}
\label{ex:LS-original-r-S}\cite{schwarz;sauer;daduna;kulik;szekli:06}
$M/M/1/\infty$ system with $(r,S)$-policy\marginpar{$(r,S)$-policy}\index{r-S-policy@$(r,S)$-policy},
 exp($\nu$)-distributed lead times, and lost sales\marginpar{lost sales}\index{lost sales}.
The inventory management process under $(r,S)$-policy fits into the
definition of the environment process by setting 
\begin{eqnarray*}
K=\{0,1,...,S\},\qquad &  & K_{B}=\{0\},\\
\Rentry 00=1,\,\,\,\,\Rentry k{k-1}=1,\quad1\leq k\leq S\,,\qquad &  & \v(k,m)=\begin{cases}
\nu, & \text{if}~~0\leq k\leq r,m=S\\
0, & \text{otherwise for}~~k\neq m\,.
\end{cases}
\end{eqnarray*}
 The queueing-inventory process $(X,Y)$ is ergodic iff ${\lambda}<{\mu}$.
The steady state distribution $\pi=(\pi(n,k):(n,k)\in E)$ of $(X,Y)$
has product form 
\[
\pi(n,k)=\left(1-\frac{\lambda}{\mu}\right)\left(\frac{\lambda}{\nu}\right)^{n}\theta(k),
\]
 where $\theta=(\theta(k):k\in K)$ with normalization constant $C$
is 
\begin{eqnarray}
\theta(k) & =\begin{cases}
C^{-1}(\frac{\lambda}{\nu}) & \qquad k=0,\\
C^{-1}(\frac{\lambda+\nu}{\lambda})^{k-1} & \qquad k=1,...,r,\\
C^{-1}(\frac{\lambda+\nu}{\lambda})^{r} & \qquad k=r+1,...,S.
\end{cases}\label{eq:LS-mm1-r-S-theta}
\end{eqnarray}
 
\end{example}
\begin{example}
\label{ex:LS-original-r-Q}\cite{schwarz;sauer;daduna;kulik;szekli:06}
$M/M/1/\infty$ system with $(r,Q)$-policy\marginpar{$(r,Q)$-policy}\index{r-Q-policy@$(r,Q)$-policy},
exp($\nu$)-distributed lead times, and lost sales. The inventory
management process under $(r,Q)$-policy fits into the definition
of the environment process by setting

\begin{eqnarray*}
K=\{0,1,...,r+Q\}\qquad &  & K_{B}=\{0\}\\
\Rentry 00=1,~~~\Rentry k{k-1}=1,\quad1\leq k\leq S\qquad &  & \v(k,m)=\begin{cases}
\nu, & \text{if}~~0\leq k\leq r,m=k+Q\\
0, & \text{otherwise for}~~k\neq m.
\end{cases}
\end{eqnarray*}

The queueing-inventory process $(X,Y)$ is ergodic iff ${\lambda}<{\mu}$.
The steady state distribution $\pi=(\pi(n,k):(n,k)\in E$ of $(X,Y)$
has product form 
\[
\pi(n,k)=\left(1-\frac{\lambda}{\mu}\right)\left(\frac{\lambda}{\nu}\right)^{n}\theta(k),
\]
 where $\theta=(\theta(k):k\in K)$ with normalization constant $C$
is

\begin{eqnarray*}
\theta(0) & = & C^{-1}\frac{\lambda}{\nu},\\
\theta(k) & = & C^{-1}\left(\frac{\lambda+\nu}{\lambda}\right)^{k-1},\qquad k=1,...,r,\\
\theta(k) & = & C^{-1}\left(\frac{\lambda+\nu}{\lambda}\right)^{r},\qquad k=r+1,...,Q,\\
\theta(k+Q) & = & C^{-1}\left(\frac{\lambda+\nu}{\lambda}\right)^{r}-\left(\frac{\lambda+\nu}{\lambda}\right)^{k-1},\qquad k=1,...,r.
\end{eqnarray*}
 
\end{example}
\begin{example}
\label{ex:LS-r-Q-r-S-revisit-extension}Recently Krishnamoorthy, Manikandan,
and Lakshmy \cite{krishnamoorthy;manikandan;lakshmy:13} analyzed
an extension of the $(r,S)$ and $(r,Q)$ inventory systems with lost
sales where the service time has a general distribution, and at the
end of the service the customer receives with probability $\gamma$
one item from the inventory while and with probability $(1-\gamma)$
the inventory level stays unchanged. The authors calculate the steady
state distribution of the system, which has a product form, and give
necessary and sufficient condition for stability. In the case of exponential
service time our model from \prettyref{sect:LS-MM1Inf-model} encompasses
this system.
\end{example}
\begin{example}
\label{inventory-production} This example is taken from \cite{krishnamoorthy;viswanath:13},
the notation is adapted to that used in Section \ref{sect:LS-MM1Inf-model}:
The authors study an inventory system under $(r,S)$-policy, which
provides items for a server who processes and forwards the items in
an on-demand production scheme. The processing time of each service
is exponentially-$\mu$ distributed. The demand occurs in a Poisson-$\lambda$
stream.

If demand arrives when the inventory is depleted it is rejected and
lost to the system forever (lost sales).

The complete system is a supply chain where new items are added to
the inventory through a second production process which is interrupted
whenever the inventory at hand reaches $S$. The production process
is resumed each time the inventory level goes down to $r$ and continues
to be on until inventory level reaches $S$ again. The times required
to add one item into the inventory (processing time + lead time) when
the production is on, are exponential-$\nu$ random variables.

All inter arrival times, service times, and production times are mutually
independent.

For a Markovian description we need to record the queue length of
not fulfilled demand $(\in\mathbb{N}_{0})$, the inventory on stock
$(\in\{0,1,\dots,S\})$, and a binary variable which indicates when
the inventory level is in $\{r+1,\dots,S\}$ whether the second production
process is on (=1) or off (= 0). (Note, that the second production
process is always on, when the inventory level is in $\{0,1,\dots,r\}$,
and is always off, when the inventory level is $S$.)

To fit this model into the framework of Section \ref{sect:LS-MM1Inf-model}
we define a Markov process $(X,Y)$ in continuous time with state
space 
\[
E:=\mathbb{N}_{0}\times K,\quad\text{with}~~K:=\{0,1,\dots,r\}\cup\{S\}\cup(\{r+1,\dots,S-1\}\times\{0,1\})\quad\text{and}~~K_{B}=\{0\}.
\]
 The environment therefore records the inventory size and the status
of the second production process, and blocking of the production system
occurs due to stock out with lost sales regime.
\end{example}
Starting from Example \ref{ex:LS-original-r-Q}, Saffari, Haji, and
Hassanzadeh \cite{saffari;haji;hassanzadeh:11} proved that under
$(r,Q)$ policy the integrated queueing-inventory $M/M/1/\infty$
system with hyper-exponential lead times (= mixtures of exponential
distributions) has a product-form distribution. The proof is done
by solving directly the steady state equations. In \cite{saffari;asmussen;haji:13},
Saffari, Asmussen, and Haji generalized this result to general lead
time distributions. The proof of product form uses some intuitive
arguments from related simplified systems and the marginal probabilities
for the inventory position are derived using regenerative arguments.\\

In the following example we show that our models encompasses queueing-inventory
systems with general replenishment lead times under $(r,S)$ policy.
This will allow us directly to conclude that for the ergodic system
the steady state has product form and this will enable us to generalize
the theorem (here Example \ref{ex:LS-original-r-S}) of \cite{schwarz;sauer;daduna;kulik;szekli:06}
to incorporate generally distributed lead times.

In a second step we will show that the results of Saffari, Haji, and
Hassanzadeh \cite{saffari;haji;hassanzadeh:11} and of Saffari, Asmussen,
and Haji \cite{saffari;asmussen;haji:13} for queueing-inventory systems
under $(r,Q)$ policy can be obtained by our method as well and can
even be slightly generalized.\\

We will consider lead time distributions of the following phase-type
which are sufficient versatile to approximate any distribution on
$\mathbb{R}_{+}$ arbitrary close.
\begin{defn}
[Phase-type distributions] \label{def:LS-theo8.2} For $k\in\mathbb{N}$
and $\beta>0$ let 
\[
\Gamma_{\beta,k}(s)=1-e^{-\beta s}\sum_{i=0}^{k-1}\frac{(\beta s)^{i}}{i!},\quad s\geq0,
\]
 denote the cumulative distribution function of the $\Gamma$--distribution
with parameters $\beta$ and $k.$ $k$ is a positive integer and
serves as a phase-parameter for the number of independent exponential
phases, each with mean $\beta^{-1},$ the sum of which constitutes
a random variable with distribution $\Gamma_{\beta,k}.$ ($\Gamma_{\beta,k}$
is called a $k$--stage Erlang distribution with shape parameter $\beta.$)

We consider the following class of distributions on $\mathbb{R}_{+}$,
which is dense with respect to the topology of weak convergence of
probability measures in the set of all distributions on $(\mathbb{R}_{+},\mathbb{B}_{+})$\,
(\cite{schassberger:73}, section I.6). For $\beta\in(0,\infty),$
$L\in\mathbb{N},$ and probability $b$ on $\{1,\dots,L\}$ with $b(L)>0$
let the cumulative distribution function 
\begin{equation}
B(s)=\sum_{\ell=1}^{L}b(\ell)\Gamma_{\beta,\ell}(s),\quad s\geq0,\label{eq:LS-ph-distr}
\end{equation}

denote a phase-type distribution function. With varying $\beta,$
$L,$ and $b$ we can approximate any distribution on $(\mathbb{R}_{+},\mathbb{B}_{+})$
sufficiently close. 
\end{defn}
To incorporate replenishment lead time distributions of phase-type
we apply the supplemented variable technique. This leads to enlarging
the phase space of the system, i.e. the state space of the inventory
process $Y$. Whenever there is an ongoing lead time, i.e., when inventory
at hand is less than $r+1$ we count the number of residual successive
i.i.d. exp($\beta$)-distributed lead time phases which must expire
until the replenishment arrives at the inventory.

The state space of $(X,Y)$ then is $E=\mathbb{N}_{0}\times K$ with
\[
K=\{r+1,r+2,...,S\}\cup\left(\{0,1,...r\}\times\{L,\dots,1\}\right)\,,
\]
 and $(X,Y)$ is irreducible on $E$.
\begin{prop}
\label{prop:LS-rS-PH}$M/M/1/\infty$ system with $(r,S)$-policy,
phase-type replenishment lead time, state dependent service rates
$\mu(n)$, and lost sales.

The lead time distribution has a distribution function $B$ from \eqref{eq:LS-ph-distr}.
We assume that $(X,Y)$ is positive recurrent and denote its steady
state distribution by 
\[
\pi=(\pi(n,k):n\in\mathbb{N}_{0}\times K).
\]
 The steady state $\pi$ of $(X,Y)$ is of product form. With normalization
constant $C$ 
\begin{equation}
\pi(n,k)={C^{-1}\prod_{i=0}^{n-1}\frac{\lambda}{\mu(i+1)}}\cdot\theta(k)\,\label{eq:LS-PFPH}
\end{equation}
where $\theta=(\theta(k):k\in K)$ is for $r=0$ 
\begin{eqnarray}
\theta(j,\ell) & = & G^{-1}\left(\frac{\lambda+\beta}{\lambda}\right)^{j-1}\sum_{i=\ell}^{L}b(i)\left(\frac{\beta}{\lambda+\beta}\right)^{i-\ell}{{i-\ell+r-j} \choose {r-j}},\label{eq:LS-PFPH1}\\
 &  & j=1,2,\dots,r,\ell=1,\dots,L\nonumber \\
\theta(0,\ell) & = & G^{-1}\frac{\lambda}{\beta}\left[\sum_{i=\ell}^{L}\left(\sum_{g=i}^{L}b(g)\right)\left(\frac{\beta}{\lambda+\beta}\right)^{i-\ell}{{i-\ell+r-1} \choose {r-1}}\right].\label{eq:LS-PFPH2}\\
 &  & \ell=1,\dots,L,\nonumber \\
\theta(r+1) & = & \theta(r+2)=\dots=\theta(S)=G^{-1}\left(\frac{\lambda+\nu}{\lambda}\right)^{r},\label{eq:LS-PFPH3}
\end{eqnarray}
 where the normalization constant $G$ is chosen such that 
\[
\sum_{k\in K}\theta(k)=1.
\]
 For $r=0$ we obtain $\theta=(\theta(k):k\in K)$ with normalization
constant $G$ as 
\begin{eqnarray}
\theta(0,\ell) & = & G^{-1}\frac{\lambda}{\beta}\left[\sum_{i=\ell}^{L}b(i)\right],\quad\ell=1,\dots,L,\ \label{eq:LS-PFPH2-0}\\
\theta(1) & = & \theta(2)=\dots=\theta(S)=G^{-1},\label{eq:LS-PFPH3-0}
\end{eqnarray}
 \end{prop}
\begin{proof}
The inventory management process under $(r,S)$-policy with distribution
function $B$ of the lead times fits into the definition of the environment
process by setting

\[
K=\{r+1,r+2,...,S\}\cup\left(\{0,1,...r\}\times\{L,\dots,1\}\right),\quad K_{B}=\{0\}\times\{L,\dots,1\}\,.
\]

The non negative transition rates of $(X,Y)$ are for $(n,k)\in E$
\begin{eqnarray*}
q((n,k)\qsep(n+1,k)) & = & \lambda\qquad k\in K_{W},n\geq0\\
q((n,k)\qsep(n-1,m)) & = & \mu(n)\Rentry km\qquad k\in K_{W},m\in K,n\intpositive,\\
q((n,k)\qsep(n,m)) & = & \v(k,m)\in\mathbb{R}_{0}^{+},~~~k\neq m,~~k,m\in K,\\
q((n,k)\qsep(i;m)) & = & 0\qquad\text{otherwise for}~~(n,k)\neq(i;m)\in E;
\end{eqnarray*}
 where 
\begin{eqnarray*}
\Rentry k{k-1}=1 & \text{if} & k\in\{r+2,...,S\},\\
\Rentry{r+1}{(r,\ell)}=b{(\ell)} & \text{if} & \ell\in\{L,\dots,1\},\\
\Rentry{(j,\ell)}{(j-1,\ell)}=1 & \text{if} & (j,\ell)\in\{1,...r\}\times\{L,\dots,1\}\,,\\
\Rentry kj=0 & \text{if} & k,j\in K,~\text{otherwise},
\end{eqnarray*}
 and 
\begin{eqnarray*}
\v((j,\ell),(j,\ell-1))=\beta & \text{if} & j\in\{0,1,...r\},\ell\in\{L,\dots,2\}\\
\v((j,1),S)=\beta & \text{if} & j\in\{0,1,...r\},\\
\v{(k,j)}=0 & \text{if} & k,j\in K,~\text{otherwise},
\end{eqnarray*}

Because $\lambda(n)=\lambda$ for all $n$, Theorem \ref{thm:LS-productform}
applies and we know that the steady state of the ergodic system is
of product form 
\begin{equation}
\pi(n,k)=C^{-1}\frac{\lambda^{n}}{\prod_{i=0}^{n-1}\mu(i+1)}\theta(k)\quad n.k)\in E,\label{eq:LS-PF1}
\end{equation}
 according to \prettyref{cor:LS-lambda-constant}. We have to solve
\eqref{eq:LS-q-tilde-n-conditon} which is independent of $n$ in
the present setting. By definition this is (with $\Rentry kk=0,\,\forall k\in K\backslash\{0\},\,\Rentry 00=1$)
\begin{eqnarray}
 &  & \theta(k)\left(1_{[k\in K_{W}]}\lambda+\sum_{m\in K\backslash\{k\}}\v(k,m)\right)\label{eq:LS-PF2}\\
 & = & \sum_{m\in K_{W}\backslash\{k\}}\theta(m)\left(\lambda(n)\Rentry mk+\v(m,k)\right)+\sum_{m\in K_{B}\backslash\{k\}}\theta(m)\v(m,k)\nonumber 
\end{eqnarray}
 \textbf{(I)} For $r>0$, \eqref{eq:LS-PF2} translates into 
\begin{eqnarray}
\theta(S)\cdot\lambda & = & \sum_{j=0}^{r}\theta(j,1)\cdot\beta,\label{eq:LS-PFspecial-1}\\
\theta(k)\cdot\lambda & = & \theta(k+1)\cdot\lambda,\qquad k=r+1,\dots,S-1\label{eq:LS-PFspecial-2}\\
\theta(r,\ell)\cdot(\lambda+\beta) & = & \theta(r+1)\cdot\lambda b(\ell)+\theta(r,\ell+1)\cdot\beta,\qquad1\leq\ell<L,\label{eq:LS-PFspecial-3}\\
\theta(r,L)\cdot(\lambda+\beta) & = & \theta(r+1)\cdot\lambda b(L),\label{eq:LS-PFspecial-4}\\
\theta(j,L)\cdot(\lambda+\beta) & = & \theta(j+1,L)\cdot\lambda,\qquad1\leq j<r\label{eq:LS-PFspecial-5}\\
\theta(j,\ell)\cdot(\lambda+\beta) & = & \theta(j+1,\ell)\cdot\lambda+\theta(j,\ell+1)\cdot\beta,~~1\leq j<r,1\leq\ell<L,\label{eq:LS-PFspecial-6}\\
\theta(0,\ell)\cdot\beta & = & \theta(1,\ell)\cdot\lambda+\theta(0,\ell+1)\cdot\beta,\qquad1\leq\ell<L,\label{eq:LS-PFspecial-7}\\
\theta(0,L)\cdot\beta & = & \theta(1,L)\cdot\lambda.\label{eq:LS-PFspecial-8}
\end{eqnarray}
 From \eqref{eq:LS-PFspecial-2} follows 
\begin{equation}
\theta(S)=\theta(S-1)=\dots=\theta(r+1),\label{eq:LS-PFspecial-L1}
\end{equation}
 and from \eqref{eq:LS-PFspecial-4} and \eqref{eq:LS-PFspecial-5}
follows 
\begin{equation}
\theta(j,L)=\theta(r+1)b(L)\left(\frac{\lambda}{\lambda+\beta}\right)^{r+1-j}.\label{eq:LS-PFspecial-L4+5}
\end{equation}
 From \eqref{eq:LS-PFspecial-L4+5} (for j=r) and \eqref{eq:LS-PFspecial-3}
follows directly 
\begin{equation}
\theta(r,\ell)=\theta(r+1)\frac{\lambda}{\lambda+\beta}\sum_{i=\ell}^{L}b(i)\left(\frac{\beta}{\lambda+\beta}\right)^{i-\ell},~~1\leq\ell<L\,.\label{eq:LS-PFspecial-L3}
\end{equation}
 Up to now we obtained the expressions for the north and west border
line of the array $(\theta(j,\ell):1\leq j\leq r,1\leq\ell\leq L)$
which can be filled step by step via \eqref{eq:LS-PFspecial-6}. The
proposed solution is 
\begin{equation}
\theta(r-h,\ell)=\theta(r+1)\left(\frac{\lambda}{\lambda+\beta}\right)^{h+1}\sum_{i=\ell}^{L}b(i)\left(\frac{\beta}{\lambda+\beta}\right)^{i-\ell}{{i-\ell+h} \choose {h}},\label{eq:LS-PFspecial-L6}
\end{equation}
 for $h=0,1,\dots,r-1,\ell=1,\dots,L$ fits with \eqref{eq:LS-PFspecial-L3}
($h=0$ with ${{i-l} \choose {0}}=1$) and \eqref{eq:LS-PFspecial-L4+5}.
Inserting \eqref{eq:LS-PFspecial-L6} into \eqref{eq:LS-PFspecial-6}
verifies \eqref{eq:LS-PFspecial-L6} by a two-step induction with
help by the elementary formula ${{a} \choose {n}}+{{a} \choose {n-1}}={{a+1} \choose {n}}$.\\

For computing the residual boundary probabilities $\theta(0,\ell)$
we need some more effort. The proposed solution is for $\ell=1,\dots,L,$
\begin{equation}
\theta(0,\ell)=\theta(r+1)\left(\frac{\lambda}{\lambda+\beta}\right)^{r}\frac{\lambda}{\beta}\left[\sum_{i=\ell}^{L}\sum_{g=i}^{L}b(g)\left(\frac{\beta}{\lambda+\beta}\right)^{i-\ell}{{i-\ell+r-1} \choose {r-1}}\right].\label{eq:LS-PFspecial-L7+8}
\end{equation}
 From \eqref{eq:LS-PFspecial-8} and \eqref{eq:LS-PFspecial-L4+5}
we obtain 
\begin{equation}
\theta(0,L)=\theta(r+1)\left(\frac{\lambda}{\lambda+\beta}\right)^{r}\frac{\lambda}{\beta}b(L),\label{eq:LS-PFspecial-L8}
\end{equation}
 which fits into \eqref{eq:LS-PFspecial-L7+8}, and it remains to
check the recursion \eqref{eq:LS-PFspecial-7}. This amounts to compute
\begin{eqnarray*}
 &  & \theta(1,\ell)\cdot\frac{\lambda}{\beta}+\theta(0,\ell+1)\\
 & = & \theta(r+1)\left(\frac{\lambda}{\lambda+\beta}\right)^{r}\sum_{i=\ell}^{L}b(i)\left(\frac{\beta}{\lambda+\beta}\right)^{i-\ell}{{i-\ell+r-1} \choose {r-1}}\cdot\frac{\lambda}{\beta}+\\
 &  & +\theta(r+1)\left(\frac{\lambda}{\lambda+\beta}\right)^{r}\frac{\lambda}{\beta}\left[\sum_{i=\ell+1}^{L}\sum_{g=i}^{L}b(g)\left(\frac{\beta}{\lambda+\beta}\right)^{i-(\ell+1)}{{i-(\ell+1)+r-1} \choose {r-1}}\right]\\
 & = & \theta(r+1)\left(\frac{\lambda}{\lambda+\beta}\right)^{r}\frac{\lambda}{\beta}\left[\sum_{i=\ell+1}^{L}\left\{ \sum_{g=i}^{L}b(g)\left(\frac{\beta}{\lambda+\beta}\right)^{i-(\ell+1)}{{\overbrace{i-(\ell+1)}^{=(i-1)-\ell}+r-1} \choose {r-1}}\right.\right.\\
 &  & \qquad\qquad\qquad\qquad\qquad\left.+b(i-1)\left(\frac{\beta}{\lambda+\beta}\right)^{(i-1)-\ell}{{(i-1)-\ell+r-1} \choose {r-1}}\right\} +\\
 &  & \qquad\qquad\qquad\qquad\qquad\qquad\qquad\left.+b(L)\left(\frac{\beta}{\lambda+\beta}\right)^{L-\ell}{{L-\ell+r-1} \choose {r-1}}\right]\\
 & = & \theta(r+1)\left(\frac{\lambda}{\lambda+\beta}\right)^{r}\frac{\lambda}{\beta}\left[\sum_{i=\ell+1}^{L}\left\{ \sum_{g=i-1}^{L}b(g)\left(\frac{\beta}{\lambda+\beta}\right)^{(i-1)-\ell}{{(i-1)-\ell+r-1} \choose {r-1}}\right\} \right.\\
 &  & \qquad\qquad\qquad\qquad\qquad\qquad\qquad\left.+b(L)\left(\frac{\beta}{\lambda+\beta}\right)^{L-\ell}{{L-\ell+r-1} \choose {r-1}}\right]\\
 & = & \theta(r+1)\left(\frac{\lambda}{\lambda+\beta}\right)^{r}\frac{\lambda}{\beta}\left[\sum_{i=\ell}^{L-1}\left\{ \sum_{g=i}^{L}b(g)\left(\frac{\beta}{\lambda+\beta}\right)^{i-\ell}{{i-\ell+r-1} \choose {r-1}}\right\} \right.\\
 &  & \qquad\qquad\qquad\qquad\qquad\qquad\qquad\left.+b(L)\left(\frac{\beta}{\lambda+\beta}\right)^{L-\ell}{{L-\ell+r-1} \choose {r-1}}\right]\\
 & = & \theta(0,\ell).
\end{eqnarray*}
 Setting 
\[
\theta(r+1)=G^{-1}\left(\frac{\lambda+\nu}{\lambda}\right)^{r}
\]
 completes the proof in case of $r=0$.

\textbf{(II)} For $r>0$, \eqref{eq:LS-PF2} translates into 
\begin{eqnarray}
\theta(S)\cdot\lambda & = & \theta(0,1)\cdot\beta,\label{eq:LS-PFspecial-1-0}\\
\theta(k)\cdot\lambda & = & \theta(k+1)\cdot\lambda,\qquad k=1,\dots,S-1\label{eq:LS-PFspecial-2-0}\\
\theta(0,\ell)\cdot\beta & = & \theta(1)\cdot\lambda\cdot b(L)+\theta(0,\ell+1)\cdot\beta,\qquad1\leq\ell<L,\label{eq:LS-PFspecial-7-0}\\
\theta(0,L)\cdot\beta & = & \theta(1)\cdot\lambda\cdot b(L).\label{eq:LS-PFspecial-8-0}
\end{eqnarray}
 From \eqref{eq:LS-PFspecial-2-0} follows 
\begin{equation}
\theta(S)=\theta(S-1)=\dots=\theta(1),\label{eq:LS-PFspecial-L1-0}
\end{equation}
 and we will show that 
\begin{equation}
\theta(0,\ell)=\theta(1)\cdot\left(\frac{\lambda}{\beta}\right)\left[\sum_{i=\ell}^{L}b(i)\right],\quad\ell=1,\dots,L,\label{eq:LS-PFspecial-L2-0}
\end{equation}
 holds. For $\ell=L$ this is immediate from \eqref{eq:LS-PFspecial-8-0},
and for $\ell<L$ it follows by induction from \eqref{eq:LS-PFspecial-7-0}.
Setting $\theta(1)=G^{-1}$ completes the proof in case of $r=0$. \end{proof}
\begin{rem}
For $r>0$ we can write \eqref{eq:LS-PFPH2} as 
\begin{eqnarray*}
\theta(0,\ell) & = & G^{-1}\frac{\lambda}{\beta}\left[\sum_{i=\ell}^{L}\left(\sum_{g=i}^{L}b(g)\right)\left(\frac{\beta}{\lambda+\beta}\right)^{i-\ell}{{i-\ell+r-1} \choose {i-\ell}}\right].\\
 &  & \ell=1,\dots,L,
\end{eqnarray*}
 and can extend this formula to the case $r=0$. This yields with
${{-1} \choose {0}}=1$ explicitly 
\begin{eqnarray*}
\theta(0,\ell) & = & G^{-1}\frac{\lambda}{\beta}\left[\sum_{i=\ell}^{L}b(i)\right],\quad\ell=1,\dots,L,
\end{eqnarray*}
 \end{rem}
\begin{cor}
In steady state the marginal probabilities for the inventory at hand
have the following simple representation.

Denote by $\nu^{-1}$ the expected lead time.

Let $U$ denote a random variable distributed according to $b=(b(\ell):1\leq\ell\leq L)$,
and let $U_{e}$ denote a random variable distributed according to
the ''equilibrium distribution'' of $U$, resp. $b$, i,e. 
\[
P(U_{e}=i)=\frac{1}{E(U)}\sum_{g=i}^{L}b(g),\quad1\leq i\leq L.
\]

Let $W(u,\alpha)$ denote a random variable distributed according
to a negative binomial distribution $Nb^{0}(u,\alpha)$ with parameters
$u\in\mathbb{N}$ and $\alpha\in(0,1)$, i.e., 
\[
P(W(u,\alpha)=i)={i+u-1 \choose u-1}\alpha^{u}(1-\alpha)^{i},\quad i\in\mathbb{N}.
\]
 Let $I$ denote a random variable distributed according to the marginal
steady state probability for the inventory size. Then for $j=1,\dots,r$
\begin{equation}
P(I=j)=G^{-1}\left(\frac{\lambda+\beta}{\lambda}\right)^{r}\cdot P(W(r+1-j,\frac{\lambda}{\lambda+\beta})<U),\label{eq:LS-inventory1}
\end{equation}
 and 
\begin{equation}
P(I=0)=G^{-1}\frac{\lambda}{\nu}\left(\frac{\lambda+\beta}{\lambda}\right)^{r}\cdot P(W(r,\frac{\lambda}{\lambda+\beta})<U_{e}).\label{eq:LS-inventory0}
\end{equation}
 For $j=r+1,\dots,S$ \eqref{eq:LS-PFPH3} applies directly: 
\[
P(I=r+1)=\dots=P(I=S)=G^{-1}\left(\frac{\lambda+\nu}{\lambda}\right)^{r}.
\]
\end{cor}
\begin{proof}
For $j=1,\dots,r$ we have 
\begin{eqnarray*}
 &  & P(I=j)=G^{-1}\left(\frac{\lambda+\beta}{\lambda}\right)^{j-1}\sum_{\ell=1}^{L}\sum_{i=\ell}^{L}b(i)\left(\frac{\beta}{\lambda+\beta}\right)^{i-\ell}{{i-\ell+r-j} \choose {r-j}}\\
 & = & G^{-1}\left(\frac{\lambda+\beta}{\lambda}\right)^{j-1}\sum_{i=1}^{L}b(i)\sum_{\ell=1}^{i}\left(\frac{\beta}{\lambda+\beta}\right)^{i-\ell}{{i-\ell+r-j} \choose {r-j}}\\
 & = & G^{-1}\left(\frac{\lambda+\beta}{\lambda}\right)^{j-1}\sum_{i=1}^{L}b(i)\sum_{g=0}^{i-1}\left(\frac{\beta}{\lambda+\beta}\right)^{g}{{g+r-j} \choose {r-j}}\\
 & = & G^{-1}\left(\frac{\lambda+\beta}{\lambda}\right)^{j-1}\left(\frac{\lambda+\beta}{\lambda}\right)^{r+1-j}\\
 &  & \qquad\qquad\qquad\sum_{i=1}^{L}b(i)\sum_{g=0}^{i-1}{{g+(r+1-j)-1} \choose {(r+1-j)-1}}\left(\frac{\lambda}{\lambda+\beta}\right)^{r+1-j}\left(\frac{\beta}{\lambda+\beta}\right)^{g}\\
 & = & G^{-1}\left(\frac{\lambda+\beta}{\lambda}\right)^{r}\sum_{i=1}^{L}b(i)\cdot P(W(r+1-j,\frac{\lambda}{\lambda+\beta})<i),
\end{eqnarray*}
 and for $j=0$ we have 
\begin{eqnarray*}
 &  & P(I=0)=G^{-1}\left(\frac{\lambda}{\beta}\right)\sum_{\ell=1}^{L}\sum_{i=\ell}^{L}\sum_{g=i}^{L}b(g)\left(\frac{\beta}{\lambda+\beta}\right)^{i-\ell}{{i-\ell+r-1} \choose {r-1}}\\
 & = & G^{-1}\left(\frac{\lambda}{\beta}\right)\sum_{i=1}^{L}\sum_{\ell=1}^{i}\left(\frac{\beta}{\lambda+\beta}\right)^{i-\ell}{{i-\ell+r-1} \choose {r-1}}\sum_{g=i}^{L}b(g)\\
 & = & G^{-1}\underbrace{\left(\frac{\lambda}{\beta}\cdot E(V)\right)}_{=\lambda/\nu}\left(\frac{\lambda+\beta}{\lambda}\right)^{r}\sum_{i=1}^{L}\underbrace{\left(\frac{1}{E(U)}\sum_{g=i}^{L}b(g)\right)}_{=:P(V_{e}=i)}\sum_{f=0}^{i-1}{{f+r-1} \choose {r-1}}\left(\frac{\lambda}{\lambda+\beta}\right)^{r}\left(\frac{\beta}{\lambda+\beta}\right)^{f}\\
 & = & G^{-1}\left(\frac{\lambda}{\nu}\right)\left(\frac{\lambda+\beta}{\lambda}\right)^{r}P(W(r+1-1,\frac{\lambda}{\lambda+\beta})<U_{e}).
\end{eqnarray*}
 
\end{proof}
We now revisit the results from \cite{saffari;haji;hassanzadeh:11}
and \cite{saffari;asmussen;haji:13} for queueing-inventory systems
under $(r,Q)$ policy. We allow additionally the service rate of the
server to depend on the queue length of the system. We assume that
the lead time distribution is of phase type.

We enlarge the phase space of the system, i.e. the state space of
the inventory process $Y$. Whenever there is an ongoing lead time,
i.e., when inventory at hand is less than $r+1$, we count the number
of residual successive i.i.d. exp($\beta$)-distributed lead time
phases which must expire until the replenishment arrives at the inventory.

The state space of $(X,Y)$ then is $E=\mathbb{N}_{0}\times K$ with
\[
K=\{r+1,r+2,...,r+Q\}\cup\left(\{0,1,...r\}\times\{L,\dots,1\}\right),
\]
 and $(X,Y)$ is irreducible on $E$. 
\begin{prop}
\label{prop:LS-lemma-rQ-PH} $M/M/1/\infty$ system with $(r,Q)$-policy,
phase-type replenishment lead time, state dependent service rates
$\mu(n)$, and lost sales.

The lead time distribution has a distribution function $B$ from \eqref{eq:LS-ph-distr}.
We assume that $(X,Y)$ is positive recurrent and denote its steady
state distribution by 
\[
\pi=(\pi(n,k):n\in\mathbb{N}_{0}\times K).
\]
 The steady state $\pi$ of $(X,Y)$ is of product form. With normalization
constant $C$ 
\begin{equation}
\pi(n,k)={C^{-1}\prod_{i=0}^{n-1}\frac{\lambda}{\mu(i+1)}}\cdot\theta(k)\,,\,\label{eq:LS-PFPH-SAH}
\end{equation}

where $\theta=(\theta(k):k\in K)$ can be obtained from formula (3)
in \cite{saffari;asmussen;haji:13}, and the subsequent formulas (4)
- (10) there. \end{prop}
\begin{proof}
The proof is in its first part similar to that of \prettyref{prop:LS-rS-PH}
because the inventory management process under $(r,Q)$-policy with
distribution function $B$ of the lead times fits into the definition
of the environment process by setting 
\[
K=\{r+1,r+2,...,r+Q\}\cup\left(\{0,1,...r\}\times\{L,\dots,1\}\right),\quad K_{B}=\{0\}\times\{L,\dots,1\}\,.
\]

Because $\lambda(n)=\lambda$ for all $n$, Theorem \ref{thm:LS-productform}
applies and we know that the steady state of the ergodic system is
of product form 
\begin{equation}
\pi(n,k)=C^{-1}\frac{\lambda^{n}}{\prod_{i=0}^{n-1}\mu(i+1)}\cdot\theta(k)\quad(n,k)\in E,\label{eq:LS-PF1-SAH}
\end{equation}
 according to \prettyref{cor:LS-lambda-constant}. Thus the product
form statement is proven with the required marginal queue length distribution.

In a second part we have to compute the $\theta(k)$ which is to solve
\eqref{eq:LS-q-tilde-n-conditon}. This equation is independent of
$n$, especially independent of the $\mu(n)$.

Therefore the solution in the case of state independent service rates
$(\mu(n)\to\mu)$ from \cite{saffari;asmussen;haji:13} must be the
solution in the present slightly more general setting as well. 
\end{proof}

\subsection{Unreliable servers}

\label{sect:unreliable-server} In \cite{sauer;daduna:03} networks
of queues with unreliable servers were investigated which admit product
form steady states in twofold way: The joint queue length vector of
the system (which in general is not a Markov process) is of classical
product form as in Jackson's Theorem and the availability status of
the nodes as a set valued supplementary variable process constitutes
an additional product factor attached to the joint queue length vector.

We show for the case of a single server which is unreliable and breaks
down due to influences from an environment that a similar product
form result follows from our Theorem \ref{thm:LS-productform}. We
allow for a much more complicated breakdown and repair process as
that investigated in \cite{sauer;daduna:03}.\\

\begin{example}
There is a single exponential server with with Poisson-$\lambda$
arrival stream and state dependent service rates $\mu(n)$. The server
acts in a random environment which changes over time. The server breaks
down with rates depending on the state of the environment and is repaired
after a breakdown with repair intensity depending on the state of
the environment as well. Whenever the server is broken down, new arrivals
are not admitted and are lost to the system forever.

The system is described by a two-dimensional Markov process $(X,Y)=((X(t),Y(t)):t\in[0,\infty))$
with state space $E=\mathbb{N}_{0}\times K$. $K$ is the (countable)
environment space of the process, whereas $\mathbb{N}_{0}$ denotes
the queue length. $(X,Y)$ is assumed to be irreducible.

The environment space of the process is partitioned into disjoint
nonempty components $K:=K_{W}+K_{B}$, and whenever $Y$ enters $K_{B}$
the server breaks down immediately, and will be repaired when $Y$
enters $K_{W}$ again.\\

The non negative transition rates of $(X,Y)$ are for $(n,k)\in E$
\begin{eqnarray}
q((n,k)\qsep(n+1,k)) & = & \lambda\qquad k\in K_{W},\nonumber \\
q((n,k)\qsep(n-1,m)) & = & \mu(n)\Rentry km\qquad k\in K_{W},n\intpositive,\label{eq:LS-breakdown-service}\\
q((n,k)\qsep(n,m)) & = & \v(k,m)\in\mathbb{R}_{0}^{+},~~~k\neq m,\nonumber \\
q((n,k)\qsep(l,m)) & = & 0\qquad\text{otherwise for}~~(n,k)\neq(l,m),\nonumber 
\end{eqnarray}
 and from Theorem \ref{thm:LS-productform} we directly obtain in
case of ergodicity the product form steady state distribution.

An interesting observation is, that we can model general distributions
for the successive times the system is functioning and similarly for
the repair times.

By suitably selected structures for the $\v(\cdot,\cdot)$ we can
incorporate dependent up and down times.\\

The distinctive feature which sets the difference to the breakdown
mechanism in \cite{sauer;daduna:03} is that breakdowns can be directly
connected with expiring service times via the stochastic matrix $R$,
which is visible from \eqref{eq:LS-breakdown-service}. This widens
applicability of the mechanism considerably. 
\end{example}

\subsection{Active-sleep model for nodes in wireless sensor networks}

\label{sect:sensor-nodes} Modeling of wireless sensor networks (WSN)
is a challenging task due to specific restrictions imposed on the
network structure and the principles the nodes have to follow to survive
without the possibility of external renewal of a node or repair. A
specific task is that usually battery power cannot be renewed which
strongly requires to control energy consumption. A typical way to
resolve this problem is to reduce energy consumption by laying a node
in sleep status whenever this is possible. In sleep status all activities
of the node are either completely or almost completely interrupted.

In active mode the node undertakes several activities: Gathering data
and putting the resulting data packets into its queue, receiving packets
from other nodes which are placed in its queue (and relaying these
packets when they arrive at the head of the node's queue), and processing
the packets in the queue (usually in a FCFS regime).

The modeling approach to undertake analytical performance analysis
of WSN in the literature is to first investigate a single (''referenced'')
node and thereafter to combine by some approximation procedure the
results to investigate the behaviour of interacting nodes, for a review
see \cite{wang;dang;wu:07}. More recent and a more detailed study
of a specific node model is \cite{li:11}, other typical examples
for the described procedure are \cite{liu;tong-lee:05}, \cite{zhang;li:11}.

We will report here only on the first step of the procedure and follow
mainly the model of a node found in \cite{liu;tong-lee:05}. The functioning
of the referenced node is governed by the following principles which
incorporate three processes. 
\begin{itemize}
\item Length of the packet queue of the node ($\in\mathbb{N}_{0}$), 
\item mode of the node (active $=A$, sleep $=S$) 
\item status of the nodes with which the referenced node is able to communicate;
these nodes are called the ''outer environment'' and their behaviour
with respect to the referenced node is summarized in a binary variable
(on $=1$, off $=0$), where ''on'' $=1$ indicates that there is
another active node in the neighborhood of the referenced node, while
''off'' $=0$ indicates that all nodes in the neighborhood of the
referenced node are in sleep mode. 
\end{itemize}
It follows that the referenced node can communicate with other nodes
if and only if the outer environment is on $=1$ \textbf{and} the
node itself is active $=A$.

Transforming the described behaviour into the formalism of Section
\ref{sect:LS-MM1Inf-model}, we end with an environment 
\[
K=\{A,S\}\times\{0,1\},\quad K_{W}=\{(A,1)\}.
\]
 and state space $E:=\mathbb{N}_{0}\times K$ of the joint process
$(X,Y)$.

In \cite{liu;tong-lee:05} the authors assume that when the referenced
node is active $(=A)$ and outer environment on $(=1)$, the stream
of packets arriving at the packet queue of the node is the superposition
of data gathering and receiving packets from other nodes. The superposition
process is a Poisson-$\lambda$ process, and processing a packet in
the queue needs an exponential-$\mu$ distributed time.

For simplicity of the model we assume here that whenever the node
is in sleep mode or the outer environment is off, all activities of
the node are frozen. This is different from \cite{liu;tong-lee:05},
who allow during this periods data gathering by the node.

Because the overwhelming part of battery capacity reduction stems
from the other two activities (to relay and processing packets), with
respect to battery control the additional simplification can be justified
for raw first approximations, especially if this is rewarded by obtaining
simple to evaluate closed form expressions.

This reward is obtained by following \cite{liu;tong-lee:05} and assuming
that the on-off $(1-0)$ process of the outer environment is an alternating
renewal process with exponential-$\alpha$, resp., exponential-$\beta$
holding times, and that the active-sleep $(A-S)$ process of the node
is an alternating renewal process with exponential-$a$, resp., exponential-$s$
holding times.

Fixing the usual overall independence assumption for all these holding
times and the processing and inter arrival times, we see that this
model fits precisely into the framework of our general model from
Section \ref{sect:LS-MM1Inf-model} and that the Theorem \ref{thm:LS-productform}
provides us with an explicit steady state distribution.\\

Results on battery consumption obtained from the steady state distribution
(similarly obtained to that in \cite{liu;tong-lee:05}) clearly will
then produce lower bounds for the energy consumption (which are weaker
than those obtained in \cite{liu;tong-lee:05} - which are obtained
with expense of more computational effort).

\subsection{Tandem system with finite intermediate buffer}

\label{sect:tandem}Modeling multi-stage production lines by serial
tandem queues is standard technique. In the simplest case with Poisson
arrivals and with exponential production times for one unit in each
stage the model fits into the realm of Jackson network models as long
as the buffers between the stages have infinite capacity. Consequently,
ergodic systems under this modeling approach have a product form steady
state distribution.

With respect to steady state analysis the picture changes completely
if the buffers between the stages have only finite capacity, no simple
solutions are available. Direct numerical analysis or simulations
are needed, or we have to resort to approximations. A common procedure
is to use product form approximations which are developed by decomposition
methods. A survey on general networks with blocking is \cite{balsamo;denittopersone;onvural:01},
special emphasis to modeling manufacturing flow lines is given in
the survey \cite{dallery;gershwin:92}.

The same class of problems and solutions are well known in teletraffic
networks where finite buffers are encountered, for surveys see \cite{onvural:90}
and \cite{perros:90}.

A systematic study of how to use product form networks as upper or
lower bounds (in a specified performance metric) is given in \cite{dijk:93}.
A closed 3-station model which is related to the one given below is
discussed in \cite{dijk:93}{[}Section 4.5.1{]}, where product form
lower and upper bounds are proposed.\\

Van Dijk \cite[p.44]{dijk:11} describes a tandem system with $\mu(n)=\mu$,
$\nu(k)=\nu$ which leads to a product form. We extend this model
by allowing more general service rates $\mu(n)$ and $\nu_{k}$.

\begin{center}
\begin{figure}[H]
\centering{}\begin{tikzpicture}
\begin{scope}[shift={(1,-3.1)}]
 \node()[] at (2,1.5) {queueing system ($X(t)$) };
 \draw[rounded corners=1ex,color=gray] (-1,-1) rectangle (5.5,2);
 \node(input-node)[] at (-2,0) {$\lambda(n)$};
 \node(loss-decision)[draw, shape=circle] at (0,0) {};
 \node(input-loss)[] at (0,-2) {lost};
 \node(server)[shape=circle, draw,
	  label=above:server 1,
  	minimum width=33pt, minimum height=30pt] at (3.5,0){$\mu(n)$};
 \node(Q)[shape=queue, draw, queue head=east,
	    label=above:queue,
		queue ellipsis=true,
		minimum width=50pt,minimum height=20pt] at (2,0) {};
 \node(output-node)[] at (6.5,0) {};
 \draw (input-node)edge[->,thick] node[]{}(loss-decision);
 \draw (loss-decision)edge[->,thick] node[]{}(Q);
 \draw (loss-decision)edge[->,thick] node[]{}(input-loss);

\end{scope}

\begin{scope}[shift={(8,-3.1)}]
  \draw[rounded corners=1ex,color=gray] (0,-1) rectangle (5.5,2);
  \node()[] at (2.5,1.5) {items in buffer ($Y(t)$) };
  \node(BufferQ)[shape=queue, draw, queue head=east, queue size=bounded,
	    label=above:buffer,
		minimum width=50pt,minimum height=20pt] at (2,0) {};
  \node(BufferServer)[shape=circle, draw,
	label=above:server 2,
    minimum width=33pt, minimum height=30pt] at (3.5,0){$\nu_n$};
  \node(BufferOutputNode)[] at (6.5,0){};
  \draw (BufferServer)edge[->,thick] node[]{}(BufferOutputNode);
\end{scope}

\draw (server)edge[->,thick] node[]{}(BufferQ);

\begin{scope}[shift={(2,0)}]
 \draw [->, thick, dashed, anchor=east]
  (8,-1)--(8,0)--(0.25,0)--(0.25,-1)
  node[align=center,midway]{};
  \node()[align=center] at (4,0.5){
   full buffer interrupts service and arrival at server 1 / resumes otherwise
  };

  \draw [->, thick, dashed, anchor=east]
    (0.25,-4.1)--(0.25,-5)--(8,-5)--(8,-4.1)
    node[align=center,midway]{};
  \node()[align=center] at (4,-5.5){
    increments buffer contents
  };
\end{scope}
\end{tikzpicture}\caption{\label{fig:LS-finite-buffer-tandem}Tandem system with finite intermediate
buffer of size $N$.}
\end{figure}
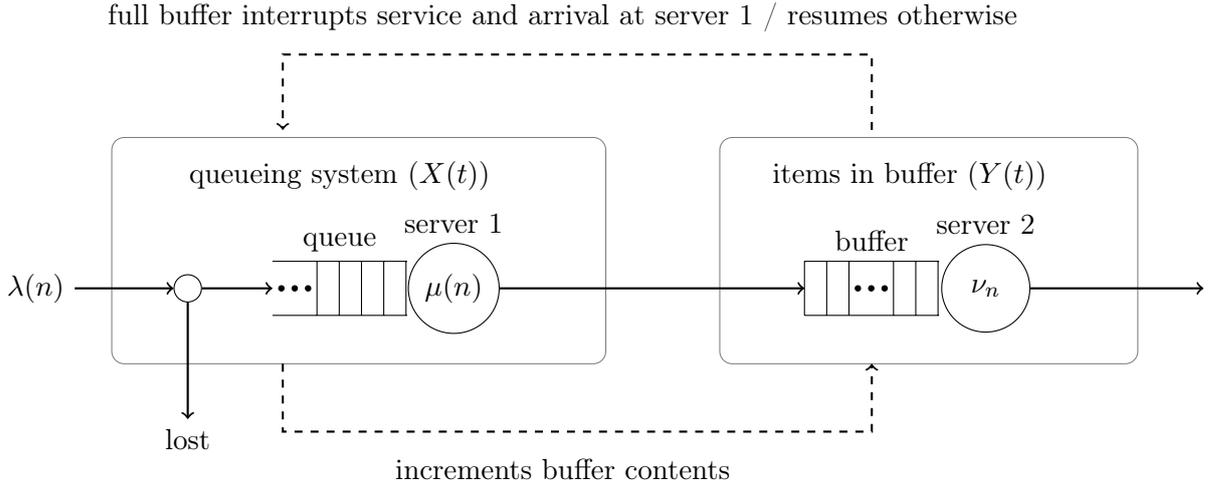

\par\end{center}

\begin{center}
\begin{figure}[H]
\centering{}\begin{tikzpicture}
  \path
	(0,0) node (B0)[shape=rectangle,draw,, minimum width=1.5cm] {$0$}
	(2.5,0) node (B1)[shape=rectangle,draw, minimum width=1.5cm] {$1$}
	(5,0) node (B2)[shape=rectangle,draw, minimum width=1.5cm] {$2$}
	(7.5,0) node (BIntermediate)[shape=rectangle, minimum width=1.5cm, minimum height=0.5cm] {...}
	(10,0) node (BN-1)[shape=rectangle,draw, minimum width=1.5cm] {$N-1$}
	(12.5,0) node (BN)[shape=rectangle,draw, fill=blocked.bg, minimum width=1.5cm] {$N$};
  \draw
    (B0.north) edge[out=45,arrows={-latex}, thick] (B1);
  \draw
    (B1.north) edge[out=45,arrows={-latex}, thick] (B2);
  \draw
    (B2.north) edge[out=45,arrows={-latex}, thick] (BIntermediate.north west);
  \draw
    (BIntermediate.north east) edge[out=45,arrows={-latex}, thick] (BN-1);
  \draw
    (BN-1.north) edge[out=45,arrows={-latex}, thick] (BN);
  \draw
    (BN.south) edge[out=-135, in=-45, arrows={-latex}, thick, color=vratecolor]
    node[auto]{$\nu_N$} (BN-1);
  \draw
    (BN-1.south) edge[out=-135, in=-45, arrows={-latex}, thick, color=vratecolor]
    node[auto]{$\nu_{N-1}$} (BIntermediate.south east);
  \draw
    (BIntermediate.south west) edge[out=-135, in=-45, arrows={-latex}, thick, color=vratecolor]
     node[auto]{$\nu_{3}$} (B2);
  \draw
    (B2.south) edge[out=-135, in=-45, arrows={-latex}, thick, color=vratecolor]
    node[auto]{$\nu_{2}$} (B1);
  \draw
    (B1.south) edge[out=-135, in=-45, arrows={-latex}, thick, color=vratecolor]
    node[auto]{$\nu_{1}$} (B0);
\end{tikzpicture}\caption{\label{fig:LS-tamdem-buffer-diagram}{\Etid}  of $M/M/1/\infty$
tandem system with finite intermediate buffer of size $N$.}
\end{figure}

\par\end{center}

We consider a two-stage single server tandem queueing system where
the first station has ample waiting space while the buffer between
the stages has only $N\geq0$ waiting places, $N<\infty$, i.e. there
can at most $N+1$ units be stored in the system which have been processed
at the first stage. It follows that for the system must be determined
a blocking regime, which enforces the first station to stop production
when the intermediate buffer reaches its capacity $N+1$. We apply
the \textbf{blocking-before-service} regime \cite{perros:90}{[}p.
455{]}: Whenever the second station is full, the server at the first
station does not start serving the next customer. When a departure
occurs from the second station, the first station is unblocked immediately
and resumes its service. Additionally, we require that the first station,
when blocked, does not accept new customers, i.e., it is completely
blocked.

The arrival stream is Poisson-$\lambda$, service rates are state
dependent with $\mu(n)$ at the first station and $\nu_{k}$ at the
second. To emphasize the modeling of the second server as an environment
we use the notation $\nu_{k}$ instead of well known from the literature
notation $\nu(k)$ for the service rate of the second server with
$k$ present customers. The standard independence assumption are assumed
to hold, service at both stations is on FCFS basis. 

This makes the joint queue length process $(X,Y)$ Markov with state
space $E:=\mathbb{N}_{0}\times\{0,1,\dots,N,N+1\}$. The non negative
transition rates are 
\begin{eqnarray*}
q((n,k)\qsep(n+1,k)) & = & \lambda,\qquad k\leq N\,,\\
q((n,k)\qsep(n-1,k+1)) & = & \mu(n),\qquad n\intpositive,k\leq N\,,\\
q((n,k)\qsep(n,k-1)) & = & \nu_{k},\qquad1\leq k\leq N+1\,,\\
q((n,k)\qsep(j,m)) & = & 0,\qquad\text{otherwise for}~~(n,k)\neq(j,m).
\end{eqnarray*}

We fit this model into the formalism of Section \ref{sect:LS-MM1Inf-model}
by setting 
\begin{eqnarray*}
K=\{0,1,...,N+1\},\qquad &  & K_{B}=\{N+1\},\\
\Rentry k{k+1}=1,0\leq k\leq N,\,\Rentry{N+1}{N+1}=1\,,~ &  & \v(k,m)=\begin{cases}
\nu_{k}, & \text{if}~1\leq k\leq N+1,\\
 & ~~~\text{and}~m=k-1\\
0, & \text{otherwise for}~k\neq m\,.
\end{cases}
\end{eqnarray*}
 From \prettyref{thm:LS-productform} and \prettyref{cor:LS-lambda-constant}
we conclude that for the ergodic process $(X,Y)$ the steady state
distribution has product form

\begin{equation}
\pi(n,k)=C^{-1}\frac{\lambda^{n}}{\prod_{i=0}^{n-1}\mu(i+1)}\theta(k)\quad(n,k)\in E,\label{eq:LS-constant-alpha-product-form-tandem}
\end{equation}
 with probability distribution $\theta$ on $K$ and normalization
constant 
\begin{eqnarray*}
C & = & \sum_{n=0}^{\infty}\frac{\lambda^{n}}{\prod_{i=0}^{n-1}\mu(i+1)}\,.
\end{eqnarray*}

It remains to determine $\theta$ from \prettyref{cor:LS-lambda-constant},
which is \eqref{eq:LS-constant-lambda-theta-matrix-equation}.

So, the $\tilde{Q}$ matrix is explicitly

\[
\tilde{Q}=\left(\begin{array}{c|cccccc}
 & 0 & 1 & 2 &  & N & N+1\\
\hline 0 & -\lambda & \lambda\\
1 & \nu_{1} & -(\nu_{1}+\lambda) & \lambda\\
2 &  & \nu_{2} & \ddots & \ddots\\
 &  &  & \ddots & \ddots & \ddots\\
N &  &  &  & \ddots & -(\nu_{N}+\lambda) & \lambda\\
N+1 &  &  &  &  & \nu_{N} & -\nu_{N}
\end{array}\right)
\]
 This is exactly the transition rate matrix of an $M/M/1/N$ queue
with Poisson-$\lambda$ arrivals and service rates $\nu_{k}$ and
we have immediately

\[
\theta(k)=G^{-1}\prod_{h=1}^{k}\frac{\lambda}{\nu_{h}}\qquad0\leq k\leq N+1
\]
with normalization constant 
\[
G=\sum_{h=0}^{N+1}(\prod_{h=1}^{k}\frac{\lambda}{\nu_{h}})\,.
\]

\begin{rem}
The result 
\[
\pi(n,k)=C^{-1}\prod_{i=0}^{n-1}\frac{\lambda^{n}}{\mu(i+1)}\cdot G^{-1}\prod_{h=1}^{k}\frac{\lambda}{\nu_{h}}\,,\quad(n,k)\in E,
\]
 is surprising, because it looks like an independence result with
marginal distributions of two queues fed by Poisson-$\lambda$ streams.
Due to the interruptions, neither the arrival process at the first
station nor the departure stream from the first node, which is the
arrival stream to the second, is Poisson-$\lambda$. There seems to
be no intuitive explanation of the results. 
\end{rem}

\subsection{Unreliable $M/M/1/\infty$ queueing system with control of repair
and maintenance}

In her PhD-thesis \cite[Section 3.2]{sauer:06} Cornellia Sauer introduced
degradable networks where failure behaviour was coupled with a service
``counter''. The counter is a special environment variable which
is decreased right after a service and can be reseted by a repair
or a preventive maintenance depending on its current value. Sauer
discussed in Remark 3.2.8 there similarities between degradable network
models with service counter and networks with inventories.

In this section we will analyze a queueing system, which utilizes
a counter to control repair and preventive maintenance and for modeling
of failure behaviour. 

We consider a queueing system, where the server wears down during
service. As a consequence the failure probability can change. We do
not require the failure probability to increase. In some systems it
is observed that whenever the server survives an initial period its
reliability stays constant or even increases.

When the system breaks down it is repaired and thereafter resumes
work as good as new. Furthermore, to prevent break downs, the system
will be maintained after a prescribed (maximal) number $N$ of services
since the most recent repair or maintenance. During repair or maintenance
the system is blocked, i.e., no service is provided and no new job
may join the system. These rejected jobs are \emph{lost} to the system. 

\emph{Subject to optimization} is $N$ - the maximal number of services,
after which the system needs to be maintained.

\begin{center}
\begin{figure}
\begin{tikzpicture}
\begin{scope}[shift={(1,-3.1)}]
  \node()[] at (1.2,1.5) {queueing system ($X(t)$) };
  \draw[rounded corners=1ex,color=gray] (-1,-1) rectangle (5.5,2);
  \node(InputNode)[] at (-2,0) {$\lambda(n)$};
  \node(LossDecision)[draw, shape=circle] at (0,0) {};
  \node(InputLoss)[] at (0,-2) {lost};
  \node(Server)[shape=circle, draw,
  	minimum width=33pt, minimum height=30pt] at (3.5,0){$\mu(n)$};
  \node(Q)[shape=queue, draw, queue head=east, queue ellipsis=true,
	    label=above:queue,
		minimum width=50pt,minimum height=20pt] at (2,0) {};
  \node(OutputNode)[] at (6.5,0) {};
  \draw (InputNode)edge[->,thick] node[]{}(LossDecision);
  \draw[arrows={-latex}, thick]
    (LossDecision) -- (Q);
  \draw[arrows={-latex}, thick]
    (LossDecision) -- (InputLoss);
  \draw[arrows={-latex}, thick]
    (Server) -- (OutputNode);
\end{scope}

\begin{scope}[shift={(4.4,2)}]
  \node()[] at (-1.,2.5) {counter/maintenance/repair ($Y(t)$) };
  \draw[rounded corners=1ex,color=gray] (-4.4,-2.) rectangle (2.1,3);
  \node(ServiceCounter)[rounded corners=3pt, draw,align=center] at (0,0)
    {service\\counter};
  \node(Failure)[rectangle, draw,align=center] at (-3.3,-1)
    {repare};
  \node(Maintainance)[rectangle, draw,align=center] at (-3,1)
    {maintenance};
  \node(AF)[] at (-5.4,-1) {};
  \node(AM)[] at (-5.4,1) {};

  \draw[arrows={-latex}, thick, dashed, near end]
    (ServiceCounter) --node[below]{starts}(Maintainance);
  \draw[arrows={-latex}, thick, dashed, near end]
    (ServiceCounter) --node[above]{on failure}(Failure);
  \draw[arrows={-latex}, thick, dashed]
    (Failure) --(-0.3,-1)--(ServiceCounter);
  \node()[] at (-0.3,-1.2){resets when finished};
  \draw[arrows={-latex}, thick, dashed]
    (Maintainance) --(-0.3,1)--(ServiceCounter);
 \node()[] at (0.1,1.3){resets when finished};
\end{scope}

   \draw[arrows={-latex}, thick, dashed]
     (Server) --(ServiceCounter);
  \node()[] at (5.5,-0.3) {increments};
  \draw[thick, dashed]
    (Failure) --(AF);
  \draw[arrows={-latex}, thick, dashed]
    (Maintainance) --(AM)--(-1,-2)--(0,-2);
  \node()[align=right] at (-3.0,0){interrupts during\\repair or maintanance,\\thereafter resumes};

\end{tikzpicture}\caption{$M/M/1/\infty$ unreliable loss system.}
\end{figure}
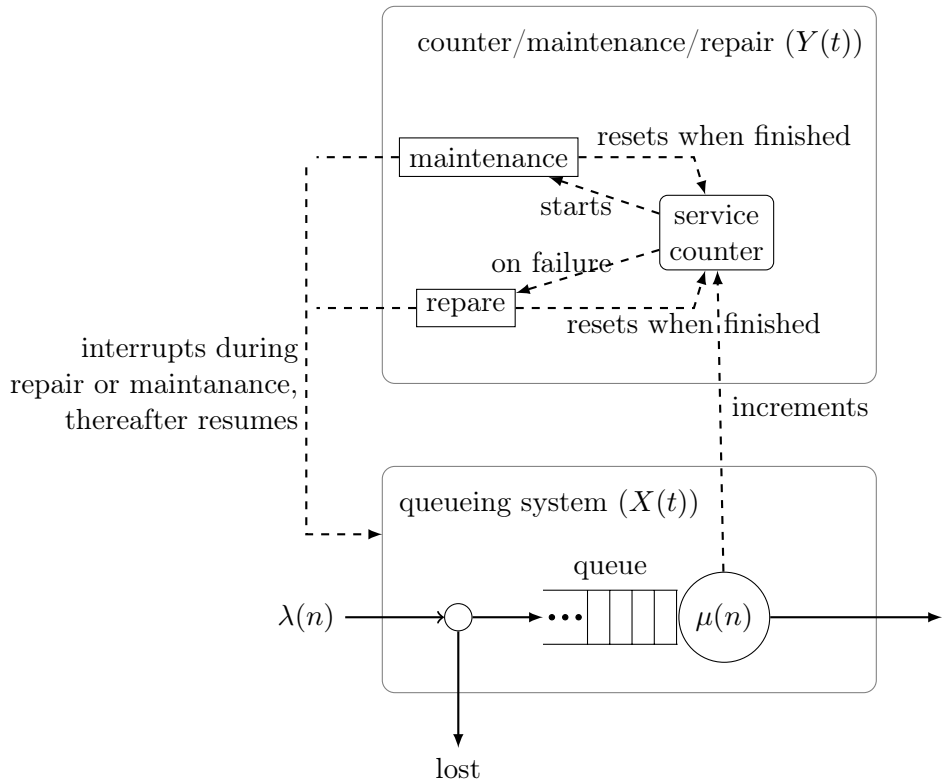

\par\end{center}

\begin{center}
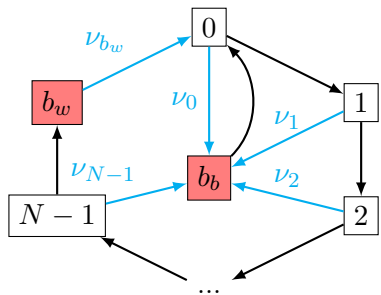
\begin{figure}
\begin{centering}
\begin{tikzpicture}
\begin{scope}
	\path
		(0,0.5) node (Y0)[shape=rectangle,draw] {$0$}
		(2,-0.5) node (Y1)[shape=rectangle,draw] {$1$}
		(2,-2) node (Y2)[shape=rectangle,draw] {$2$}
		(0,-3.0) node (Y3)[] {...}
		(-2,-2) node (YN-1)[shape=rectangle,draw] {$N-1$}
		(-2,-0.5) node (b-w)[shape=rectangle,draw,fill=blocked.bg] {$b_w$}
		(0,-1.5) node (b-b)[shape=rectangle,draw,fill=blocked.bg] {$b_b$};
	\draw (Y0){}
      edge[arrows={-latex},thick] node[auto]{}(Y1);
	\draw (Y1){} edge[arrows={-latex},thick] node[auto]{}(Y2);
	\draw (Y2){} edge[arrows={-latex},thick] node[auto]{}(Y3);
	\draw (Y3){} edge[arrows={-latex},thick] node[auto]{}(YN-1);
	\draw (YN-1){} edge[arrows={-latex},thick] node[auto]{}(b-w);
	\draw (b-w){}
        edge[arrows={-latex},color=vratecolor,thick] node[auto]{$\nu_{b_w}$}(Y0);
	\draw (Y0){}
      edge[arrows={-latex},color=vratecolor,thick] node[auto, anchor=east]{$\nu_0$}(b-b);
	\draw (Y1){} edge[arrows={-latex},color=vratecolor,thick] node[auto, anchor=south]{$\nu_1$}(b-b);
	\draw (Y2){} edge[arrows={-latex},color=vratecolor,thick] node[auto, anchor=south]{$\nu_2$}(b-b);
	\draw (YN-1){}
      edge[arrows={-latex},color=vratecolor,thick] node[auto]{$\nu_{N-1}$}(b-b);
	\draw (b-b){} edge[arrows={-latex},thick,out=45,in=-45] node[auto]{}(Y0);
\end{scope}
\end{tikzpicture}
\par\end{centering}

\caption{{\Etid} for the $M/M/1/\infty$ unreliable system. The environment
describes the service counter, repair state and maintenance state.}
\end{figure}

\par\end{center}

\subsubsection{Model\label{sect:unreliable-math-model}}

We consider a production system which is modeled as an $M/M/1/\infty$
loss system. That is with $\lambda$ Poisson input rate, service rates
$\vec{\mu}:=(\mu(n):n\in\mathbb{N})$ depending on the number of customers
in the system, $FCFS$ service regime, and environment states $K=K_{W}+K_{B}$.

The state space of the system is $E=\mathbb{N}_{0}\times(\{0,1,\ldots,N-1\}\cup\{b_{m},...,b_{r}\})$.
The environment states $K_{W}=\{0,1,...N-1\}$ indicate the number
of services completed since the last repair or maintenance (service
counter). $N$ is the maximal number of services before maintenance
is required. The additional environment states $K_{B}=\{b_{m},b_{r}\}$
indicate when there is an ongoing maintenance ($b_{m}$) or repair
($b_{r}$).

We define a stochastic matrix $R\in[0,1]^{K\times K}$ which determines
the behaviour of the ``service counter''. Transition rates $\Rentry k{k+1}=1$
for $0\leq k\leq N-2$ govern counter increment and transition rate
$\Rentry{N-1}{b_{m}}=1$ enforces mandatory maintenance after $N$
services.

We use infinitesimal generator $V\in\mathbb{R}^{K\times K}$ to control
failure, maintenance, and repair rates: $\v(k,b_{r})=\nu_{k}$ are
failure rates after $k$ complete services, $\v(b_{m},0)=\nu_{m}$
is the maintenance rate and $\v(b_{r},0)=\nu_{r}$ is the repair rate.

We define by $Z=(X,Y)$ the joint queue length and environment process
of this system and make the usual independence assumptions for the
queue and the environment. The $Z$ is Markov process, which we assume
to be ergodic.

The total costs of the system is determined by specific cost constants
per unit of time: maintenance costs $c_{m}$, repair costs $c_{r}$,
costs of non-availability $c_{b}$, and waiting costs in queue and
in service per customer $c_{w}$. Therefore the cost function per
unit of time in the respective states is
\begin{eqnarray*}
f:\mathbb{N}_{0}\times K & \longrightarrow & \mathbb{R}\\
f(n,k) & = & \begin{cases}
c_{w}\cdot n+c_{b}+c_{m}, & \qquad k=b_{m},\\
c_{w}\cdot n+c_{b}+c_{r}, & \qquad k=b_{r},\\
c_{w}\cdot n, & \qquad k\in K\backslash\{b_{m},b_{r}\}.
\end{cases}
\end{eqnarray*}

Our aim is to analyze the long-run system behaviour and to minimize
the long-run average costs.
\begin{prop}
The steady state distribution of the system described above is
\end{prop}
\[
\lim_{t\rightarrow\infty}P(X(t)=n,Y(t)=k)=:\pi(n,k)=\xi(n)\theta(k),\text{ with}
\]

\begin{align*}
\xi(n) & =\prod_{i=1}^{n}\frac{\lambda}{\mu(i)}\xi(0),\text{ and}
\end{align*}

\begin{align*}
\theta(k) & =\prod_{i=1}^{k}\left(\frac{\lambda}{\nu_{i}+\lambda}\right)^{i}\theta(0)\qquad0\leq k\leq N-1\\
\theta(b_{m}) & =\frac{\lambda}{\nu_{m}}\theta(N-1)=\frac{\lambda}{\nu_{m}}\prod_{i=1}^{N-1}\left(\frac{\lambda}{\nu_{i}+\lambda}\right)^{i}\theta(0)\\
\theta(b_{r}) & =\left(\frac{(\nu_{0}+\lambda)}{\nu_{r}}-\frac{\lambda}{\nu_{r}}\prod_{i=1}^{N-1}\left(\frac{\lambda}{\nu_{i}+\lambda}\right)^{i}\right)\theta(0)\\
\theta(0) & =\frac{1}{\left(\frac{(\nu_{0}+\lambda)}{\nu_{r}}+\left(\frac{\lambda}{\nu_{m}}-\frac{\lambda}{\nu_{r}}\right)\prod_{i=1}^{N-1}\left(\frac{\lambda}{\nu_{i}+\lambda}\right)^{i}\right)+\sum_{k=0}^{N-1}\prod_{i=1}^{k}\left(\frac{\lambda}{\nu_{i}+\lambda}\right)^{i}}
\end{align*}

\begin{proof}
The $M/M/1/\infty$ system with unreliable server, maintenance and
repair is a loss system in a random environment with 

\begin{eqnarray*}
 & V=\\
 & \tiny\left(\begin{array}{c|ccccccccc}
 & 0 & 1 & 2 & \ldots & N-3 & N-2 & N-1 & b_{m} & b_{r}\\
\hline 0 & -\nu_{0} & 0 & 0 &  & 0 & 0 & 0 & 0 & \nu_{0}\\
1 & 0 & -\nu_{1} & 0 &  & 0 & 0 & 0 & 0 & \nu_{1}\\
2 & 0 & 0 & -\nu_{2} &  & 0 & 0 & 0 & 0 & \nu_{2}\\
\vdots &  &  &  & \ddots &  &  &  &  & \vdots\\
N-3 & 0 & 0 & 0 & \ldots & -\nu_{N-3} & 0 & 0 & 0 & \nu_{N-3}\\
N-2 & 0 & 0 & 0 & \ldots & 0 & -\nu_{N-2} & 0 & 0 & \nu_{N-2}\\
N-1 & 0 & 0 & 0 & \ldots & 0 & 0 & -\nu_{N-1} &  & \nu_{N-1}\\
b_{m} & \nu_{m} & 0 & 0 & \ldots & 0 & 0 &  & -\nu_{m} & 0\\
b_{r} & \nu_{r} & 0 & 0 & \ldots & 0 & 0 &  & 0 & -\nu_{r}
\end{array}\right)
\end{eqnarray*}
and
\begin{eqnarray*}
 & R=\\
 & \tiny\left(\begin{array}{c|ccccccccc}
 & 0 & 1 & 2 & \ldots & N-3 & N-2 & N-1 & b_{m} & b_{b}\\
\hline 0 & 0 & 1 & 0 & \ldots & 0 & 0 & 0 & 0 & 0\\
1 & 0 & 0 & 1 &  & 0 & 0 & 0 & 0 & 0\\
2 & 0 & 0 & 0 & \ddots & 0 & 0 & 0 & 0 & 0\\
\vdots &  &  &  &  & \ddots &  &  &  & \vdots\\
N-3 & 0 & 0 & 0 & \ldots & 0 & 1 & 0 & 0 & 0\\
N-2 & 0 & 0 & 0 & \ldots & 0 & 0 & 1 & 0\\
N-1 & 0 & 0 & 0 & \ldots & 0 & 0 & 0 & 1 & 0\\
b_{m} & 0 & 0 & 0 & \ldots & 0 & 0 & 0 & 1 & 0\\
b_{r} & 0 & 0 & 0 & \ldots & 0 & 0 & 0 & 0 & 1
\end{array}\right)
\end{eqnarray*}

The matrix $\tilde{Q}$ from \prettyref{cor:LS-lambda-constant} is 

\begin{eqnarray*}
 & \tilde{Q}=\lambda I_{W}(R-I)+V=\\
 & \tiny\left(\begin{array}{c|ccccccccc}
 & 0 & 1 & 2 & \ldots & N-3 & N-2 & N-1 & b_{m} & b_{r}\\
\hline 0 & -\nu_{0}-\lambda & \lambda & 0 &  & 0 & 0 & 0 & 0 & \nu_{0}\\
1 & 0 & -\nu_{1}-\lambda & \lambda &  & 0 & 0 & 0 & 0 & \nu_{1}\\
2 & 0 & 0 & -\nu_{2} & \ddots & 0 & 0 & 0 & 0 & \nu_{2}\\
\vdots &  &  &  & \ddots & \ddots &  &  &  & \vdots\\
N-3 & 0 & 0 & 0 & \ldots & -\nu_{N-3}-\lambda & \lambda & 0 & 0 & \nu_{N-3}\\
N-2 & 0 & 0 & 0 & \ldots & 0 & -\nu_{N-2}-\lambda & \lambda & 0 & \nu_{N-2}\\
N-1 & 0 & 0 & 0 & \ldots & 0 & 0 & -\nu_{N-1}-\lambda & \lambda & \nu_{N-1}\\
b_{m} & \nu_{m} & 0 & 0 & \ldots & 0 & 0 & 0 & -\nu_{m} & 0\\
b_{r} & \nu_{r} & 0 & 0 & \ldots & 0 & 0 & 0 & 0 & -\nu_{r}
\end{array}\right)
\end{eqnarray*}
and the steady state solution of the system has a product form with
marginal distribution $\theta$ solving $\theta\tilde{Q}=0$

We now solve the equation $\theta\tilde{Q}=0$.

For $k\in\{1,2,...,N-1\}$ it follows directly
\begin{eqnarray*}
\lambda\theta(k-1)-(\nu_{k}+\lambda)\theta(k) & = & 0\\
\Longrightarrow\theta(k) & = & \frac{\lambda}{\nu_{k}+\lambda}\theta(k-1)\\
\Longrightarrow\theta(k) & = & \prod_{i=1}^{k}\left(\frac{\lambda}{\nu_{i}+\lambda}\right)^{i}\theta(0)
\end{eqnarray*}

For $k=b_{m}$ we have
\begin{eqnarray*}
\lambda\theta(N-1)-\nu_{m}\theta(b_{m}) & = & 0\\
\theta(b_{m}) & = & \frac{\lambda}{\nu_{m}}\theta(N-1)=\frac{\lambda}{\nu_{m}}\prod_{i=1}^{N-1}\left(\frac{\lambda}{\nu_{i}+\lambda}\right)^{i}\theta(0)
\end{eqnarray*}

Finally, for $k=0$ we obtain

\begin{eqnarray*}
-(\nu_{0}+\lambda)\theta(0)+\nu_{m}\theta(b_{m})+\nu_{r}\theta(b_{r}) & = & 0\\
\Longrightarrow\theta(b_{r}) & = & \frac{(\nu_{0}+\lambda)}{\nu_{r}}\theta(0)-\frac{\nu_{m}}{\nu_{r}}\theta(b_{m})\\
 & = & \left((\nu_{0}+\lambda)-\lambda\prod_{i=1}^{N-1}\left(\frac{\lambda}{\nu_{i}+\lambda}\right)^{i}\right)\frac{1}{\nu_{r}}\theta(0)
\end{eqnarray*}

This leads to 
\[
\theta(0)=\frac{1}{\left(\frac{(\nu_{0}+\lambda)}{\nu_{r}}+\left(\frac{\lambda}{\nu_{m}}-\frac{\lambda}{\nu_{r}}\right)\prod_{i=1}^{N-1}\left(\frac{\lambda}{\nu_{i}+\lambda}\right)^{i}\right)+\sum_{k=0}^{N-1}\prod_{i=1}^{k}\left(\frac{\lambda}{\nu_{i}+\lambda}\right)^{i}}
\]
.
\end{proof}

\subsubsection{Average costs}

We will analyze average long term costs of the system with different
$N$ - the maximal number of services, after which system needs to
be maintained. In order to distinguish the steady state distribution
of the systems with different parameters $N$ we will denote them
$\pi_{N}$ and $\theta_{N}$.
\begin{lem}
\label{lem:LS-unreliable-optimal-consts} The optimal solution for
the problem described in \prettyref{sect:unreliable-math-model} is
\end{lem}
\[
\arg\min(g(N))
\]

with

\[
g(N):=\left(c_{b}+c_{m}\right)\theta_{N}(b_{m})+\left(c_{b}+c_{r}\right)\theta_{N}(b_{r})
\]

\begin{proof}
Due to ergodicity it holds 
\[
\lim_{T\rightarrow\infty}\frac{1}{T}\int_{0}^{T}f\left(X(\omega)_{t},Y_{t}(\omega)\right)dt=\sum_{(n,k)}f(n,k)\pi_{N}(n,k)=:\bar{f}(N),\qquad P.a.s
\]

Using product form properties of the system we get 

\begin{eqnarray*}
\bar{f}(N) & = & \left(c_{b}+c_{m}\right)\theta_{N}(b_{m})+\left(c_{b}+c_{r}\right)\theta_{N}(b_{r})+\underbrace{c_{w}\sum_{n=1}^{\infty}n\xi(n)}_{\text{independent of }N}
\end{eqnarray*}

\[
\Longrightarrow\arg\min(\bar{f}(N))=\arg\min(g(N))
\]
\end{proof}
\begin{example}
\label{ex:LS-failure-system-linear-break-downs}We consider two unreliable
systems with parameters $\lambda=1$, $\mu=1.5$, $\nu_{m}=0.3$,
$\nu_{r}=0.1$, costs $c_{m}=1$, $c_{r}=2$, $c_{b}=1$ and different
linear increasing functions $\nu_{k}$ determining wearout.

In case $\nu_{k}=0.01k$, the optimal number of services after which
maintenance should be performed $N=\arg\min(g(N))=6$. See \prettyref{fig:LS-failure-system-linear-break-downs-0.01k}.

In the case $\nu_{k}=0.001k$ the optimal value is $N=23$. See \prettyref{fig:LS-failure-system-linear-break-downs-0.001k}.

\begin{figure}[H]
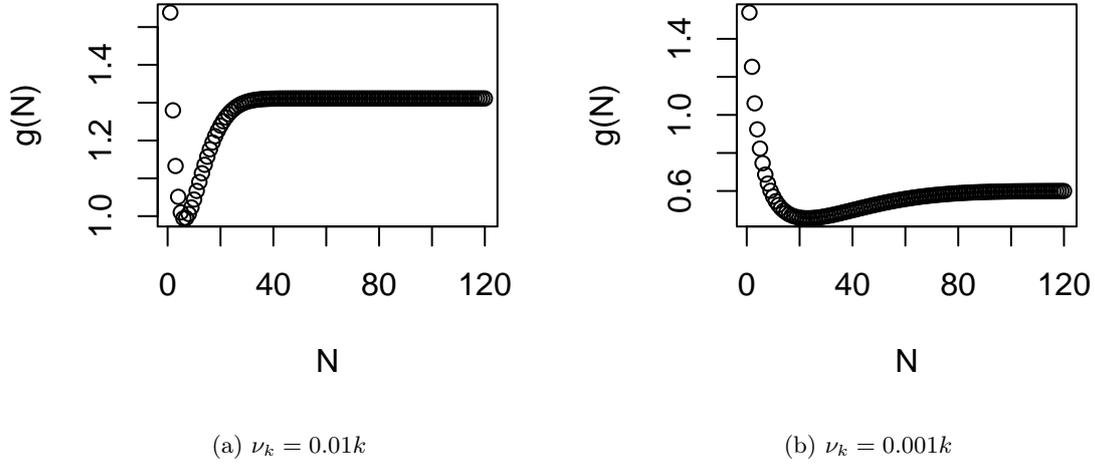

\subfloat[\label{fig:LS-failure-system-linear-break-downs-0.01k}$\nu_{k}=0.01k$]{\includegraphics{g-linear-failure-v0\lyxdot 01k}}\subfloat[\label{fig:LS-failure-system-linear-break-downs-0.001k}$\nu_{k}=0.001k$]{\includegraphics{g-linear-failure-v0\lyxdot 001k}}\caption{Function $g(N)$ for linear break down rates $\nu_{k}$ in Example
\ref{ex:LS-failure-system-linear-break-downs}.}

\end{figure}

\end{example}
The case of constant failure rate $\nu$ (independent of $k$, i.
e., of no aging) is often of particular interest. For example it is
common practice to substitute complex varying system parameters by
average values. In the following corollary and the remark we will
investigate the properties of such a system in detail. 
\begin{cor}
\label{cor:LS-constant-failure-rates}If $\nu$ is constant the function
$g(N)$ is either strictly monotone increasing, or strictly monotone
decreasing, or zero.\end{cor}
\begin{proof}
We analyze the difference 
\[
g(N+1)-g(N)=\left(c_{b}+c_{m}\right)\left(\theta_{N+1}(b_{m})-\theta_{N}(b_{m})\right)+\left(c_{b}+c_{r}\right)\left(\theta_{N}(b_{r})-\theta_{N}(b_{r})\right)
\]

We calculate $\theta_{N}(0)$ for constant $\nu_{i}\equiv\nu$, we
will use $a:=\left(\frac{\lambda}{\nu+\lambda}\right)$ to simplify
the notation

\begin{eqnarray*}
\theta_{N}(0) & = & \frac{1}{\left(\frac{(\nu+\lambda)}{\nu_{r}}-\frac{\lambda}{\nu_{r}}a^{N-1}\right)+\frac{\lambda}{\nu_{m}}a^{N-1}+\sum_{k=0}^{N-1}a^{k}}\\
 & = & \frac{1}{\left(\frac{(\nu+\lambda)}{\nu_{r}}+\left(\frac{\lambda}{\nu_{m}}-\frac{\lambda}{\nu_{r}}\right)a^{N-1}\right)+\frac{\left(1-a^{N}\right)}{(1-a)}}\\
 & = & \frac{(1-a)}{\left(\frac{(v+\lambda)}{\nu_{r}}+\left(\frac{\lambda}{\nu_{m}}-\frac{\lambda}{\nu_{r}}\right)a^{N-1}\right)(1-a)+\left(1-a^{N}\right)}
\end{eqnarray*}

Using the result above we calculate $\theta_{N}(b_{m})$ 
\begin{eqnarray*}
\theta_{N}(b_{m}) & = & \frac{\frac{\lambda}{\nu_{m}}\left(\frac{\lambda}{\nu+\lambda}\right)^{N-1}(1-a)}{\left(\frac{(v+\lambda)}{\nu_{r}}+\left(\frac{\lambda}{\nu_{m}}-\frac{\lambda}{\nu_{r}}\right)a^{N-1}\right)(1-a)+\left(1-a^{N}\right)}\\
 & = & \frac{\overbrace{\left(\frac{\lambda}{\nu_{m}}a^{N-1}\right)(1-a)}^{=:Nom(N,b_{m})}}{\underbrace{\left(\frac{(v_{0}+\lambda)}{\nu_{r}}+\left(\frac{\lambda}{\nu_{m}}-\frac{\lambda}{\nu_{r}}\right)a^{N-1}\right)(1-a)+\left(1-a^{N}\right)}_{=:Den(N)}}\\
 & = & \frac{Nom(N,b_{m})}{Den(N)}
\end{eqnarray*}

\[
\theta_{N+1}(b_{m})-\theta_{N}(b_{m})=\frac{Nom(N+1,b_{m})Den(N)-Nom(N,b_{m})Den(N+1)}{Den(N+1)Den(N)}.
\]
The common denominator of the difference $Den(N+1)Den(N)$ is positive,
therefore we focus on nominator $Nom(N+1,b_{m})Den(N)-Nom(N,b_{m})Den(N+1)=$

\[
=\frac{\lambda(1-a)a^{N-1}}{\nu_{m}\nu_{r}}\left(-a^{2}\nu+2a\nu-\nu+a\nu_{r}-v_{r}-\lambda a^{2}+2\lambda a-\lambda\right)
\]

Similarly we analyze the sign of the difference $\theta_{N+1}(b_{r})-\theta_{N}(b_{r})$:

\begin{eqnarray*}
\theta_{N}(b_{r}) & = & \left(\frac{(\nu+\lambda)}{\nu_{r}}-\frac{\lambda}{\nu_{r}}\left(\frac{\lambda}{\nu+\lambda}\right)^{N-1}\right)\theta_{N}(0)\\
 & = & \frac{\left(\frac{(v+\lambda)}{\nu_{r}}-\frac{\lambda}{\nu_{r}}a^{N-1}\right)}{\left(\frac{(\nu+\lambda)}{\nu_{r}}-\frac{\lambda}{\nu_{r}}a^{N-1}\right)+\frac{\lambda}{\nu_{m}}a^{N-1}+\sum_{k=0}^{N-1}a^{k}}\\
 & = & \frac{\overbrace{\left(\frac{(v+\lambda)}{\nu_{r}}-\frac{\lambda}{\nu_{r}}a^{N-1}\right)(1-a)}^{=:Nom(N,b_{r})}}{\underbrace{\left(\frac{(v+\lambda)}{\nu_{r}}+\left(\frac{\lambda}{\nu_{m}}-\frac{\lambda}{\nu_{r}}\right)a^{N-1}\right)(1-a)+\left(1-a^{N}\right)}_{=:Den(N)}}
\end{eqnarray*}

\[
\theta_{N+1}(b_{r})-\theta_{N}(b_{r})=\frac{Nom(N+1,b_{r})Den(N)-Nom(N,b_{r})Den(N+1)}{Den(N+1)Den(N)}.
\]
with $Nom(N+1,b_{r})Den(N)-Nom(N,b_{r})Den(N+1)=$

\[
\frac{(1-a)a^{N-1}}{\nu_{m}\nu_{r}}\left(a^{2}\nu_{m}\nu-a\nu_{m}\nu+\lambda a^{2}\nu-2\lambda a\nu+\lambda\nu+\lambda a^{2}\nu_{m}-2\lambda a\nu_{m}+\lambda\nu_{m}+\lambda^{2}a^{2}-2\lambda^{2}a+\lambda^{2}\right)
\]
and therefore

\begin{eqnarray*}
 &  & g(N+1)-g(N)\\
 & = & \frac{(1-a)a^{N-1}}{\nu_{m}\nu_{r}Den(N+1)Den(N)}\cdot\\
 &  & \left[(c_{b}+c_{m})\lambda(-a^{2}\nu+2a\nu-\nu+a\nu_{r}-v_{r}-\lambda a^{2}+2\lambda a-\lambda)+(c_{b}+c_{r})\cdot\right.\\
 &  & \left.(a^{2}\nu_{m}\nu-a\nu_{m}\nu+\lambda a^{2}\nu-2\lambda a\nu+\lambda\nu+\lambda a^{2}\nu_{m}-2\lambda a\nu_{m}+\lambda\nu_{m}+\lambda^{2}a^{2}-2\lambda^{2}a+\lambda^{2})\right]
\end{eqnarray*}

The sign of the difference depends only on the expression in the square
brackets which is independent of $N$.

The results can be explained by the memoryless property of the exponential
failure time distribution: If the failure rate is constant, the system
behaviour stays the same, no matter how much time elapsed (or how
many services are completed) since the last maintenance.\end{proof}
\begin{rem}
The most important consequence of \prettyref{cor:LS-constant-failure-rates}
is that the costs function determines only three type of solutions:
\begin{itemize}
\item Maintain immediately (cost function is strictly monotone increasing).
\item Never maintain (cost function is strictly monotone decreasing).
\item Maintain at any time (cost function stays constant).
\end{itemize}
\end{rem}
As a consequence we see that models with constant failure rate are
not well suited for drawing conclusions about, e.g., models with linear
failure rates, see \prettyref{ex:LS-failure-system-linear-break-downs}.

\part{Embedded Markov chains analysis }

\global\long\def\V{V}
\global\long\def\v{v}
\global\long\def\qsep{,}
\global\long\def\intpositive{>0}
\global\long\def\Rentry#1#2{R(#1,#2)}

We continue our investigations described in \prettyref{part:LS-continous-time}.
The systems live in continuous time and we now observe them at departure
instants only, which results in considering an embedded Markov chain.
We find that the behaviour of the embedded Markov chain is often considerably
different from that of the original continuous time Markov process
investigated in \prettyref{part:LS-continous-time}.

Our aim is to identify conditions which guarantee that even for the
embedded Markov chain a product form equilibrium exists in the discrete
observation time points as well.

For exponential queueing systems we show that there is a product form
equilibrium under rather general conditions. For systems with non-exponential
service times more restrictive constraints are needed, which we prove
by a counter example where the environment represents an inventory
attached to an $M/D/1$ queue. Such integrated queueing-inventory
systems are dealt with in the literature previously. Further applications
are, e.g., in modeling unreliable queues.

For investigating $M/G/1/\infty$ queues embedded Markov chains provide
a standard procedure to avoid using supplementary variable technique.
Embedded chain analysis was applied by Vineetha \cite{vineetha:08}
who extended the theory of integrated queueing-inventory models with
exponential service times to systems with service times which are
i.i.d. and follow a general distribution. Our investigations which
are reported in this paper were in part motivated by her investigations.

In Section \ref{sect:LS-EMC-MG1} we revisit some of Vineetha's \cite{vineetha:08}
queueing-inventory systems, using similarly embedded Markov chain
techniques. In the course of these investigations we found that there
arise problems even for purely exponential systems, which we describe
in Section \ref{sect:LS-EMC} and \ref{sect:LS-EMC-steadystate} first,
before describing the $M/G/1/\infty$ queue in a random environment
and its structural properties.\\

To emphasize the problems arising from the interaction of the two
components of integrated systems, we remind the reader, that for ergodic
$M/M/1/\infty$ queues the limiting and stationary distribution of
the continuous time queue length process and the Markov chains embedded
at either arrival instants or departure instants are the same.

Our first finding is, that even in the case of $M/M/1/\infty$ queues
with attached inventory this in general does not hold. This especially
implies, that the product form results obtained in \prettyref{part:LS-continous-time}
do not carry over immediately to the case of loss systems in a random
environment observed at departure times from the queue (downward jumps
of the generalized birth-and-death process).

A striking observation is that for a system which is ergodic in the
continuous time Markovian description the Markov chain embedded at
departure instants may be not ergodic. The reason for this is two-fold.
Firstly, the embedded Markov chain may have inessential states due
to the specified interaction rules. Secondly, even when we delete
all inessential states, the resulting single positive recurrent class
may be periodic.

We study this problem in depth in Section \ref{sect:LS-EMC-steadystate}
for purely exponential systems, and provide a set of examples which
elucidate the problems which one is faced with. Our main result in
this section proves the existence of a product form steady state distribution
(which is not necessary a limiting distribution) for the Markov chain
embedded at departure instants and provides a precise connection between
the steady states of the continuous time process and the embedded
chain (\prettyref{thm:LCS-EMC-product-form-exp-service}).\\

It turns out, that a similar result in the setting with $M/G/1/\infty$
queues is not valid. We are able to give sufficient conditions for
the structure of the environment, which guarantee the existence of
product form equilibria (\prettyref{thm:MG1-PF-result}).

Unfortunately enough, an analogue to \prettyref{thm:LCS-EMC-product-form-exp-service}
is not valid for systems with non exponential service times. We prove
this by constructing a counterexample which is an $M/D/1/\infty$
queue with an attached environment in Section \ref{sect:MD1-1}.\\

Most of our results for systems rely strongly on non-singularity of
a certain matrix which reflects important aspects of the system. For
systems with a finite environment the regularity of that matrix is
proved in an Appendix as a useful lemma which is of interest in its
own. This lemma generalizes the well known theorem of invertibility
of M-matrices which are irreducible to the case where irreducibility
is not required, but only a certain flow condition prevails.\\

\section{M/M/1/$\infty$ queueing system in a random environment\label{sect:LS-EMC-exponential}}

Recall that the paths of $Z$ are cadlag. With $\tau_{0}=\sigma_{0}=\zeta_{0}=0$
and 
\[
\tau_{n+1}:=\inf(t>\tau_{n}:X(t)<X(\tau_{n})),\quad n\in\mathbb{N}\,.
\]
denote the sequence of departure times of customers by $\tau=(\tau_{0},\tau_{1},\tau_{2},\dots)$,
and with 
\[
\sigma_{n+1}:=\inf(t>\sigma_{n}:X(t)>X(\sigma_{n})),\quad n\in\mathbb{N}\,,
\]
denote by $\sigma=(\sigma_{0},\sigma_{1},\sigma_{2},\dots)$ the sequence
of instants when arrivals are admitted to the system (because the
environment is in states of $K_{W}$, i.e., not blocking)\\
 and with 
\[
\zeta_{n+1}:=\inf(t>\zeta_{n}:Z(t)\neq Z(\zeta_{n})),\quad n\in\mathbb{N}\,,
\]
denote by $\zeta=(\zeta_{0},\zeta_{1},\zeta_{2},\dots)$ the sequence
of jump times of $Z$.\\

The following lemmata will be used in the sequel. They refer to the
structure of the continuous time process. We emphasize that the generator
$\V$ is not necessarily irreducible.
\begin{lem}
\label{lem:LS-EMC-non-zero-path} Let $Z$ be ergodic. Then for any
non-empty subset $\tilde{K}_{B}\subset K_{B}$ the overall $\V$-transition
rate from $\tilde{K_{B}}$ to its complement $\tilde{K}_{B}^{c}=K\setminus\tilde{K_{B}}$
is positive, i.e., 
\begin{equation}
\forall~~~\tilde{K}_{B}\subset K_{B},\tilde{K}_{B}\neq\emptyset:\qquad\exists~~~k\in\tilde{K}_{B},m\in\tilde{K}_{B}^{c}:\v(k,m)>0\label{eq:LS-EMC-non-zero-path-lemma}
\end{equation}
\end{lem}
\begin{rem*}
Consider the directed transition graph of $\V$, with vertices $K$
and edges ${\cal E}$ defined by $km\in{\cal E}\Longleftrightarrow\v(k,m)>0\,.$
Then the condition \eqref{eq:LS-EMC-non-zero-path-lemma} guarantees
the existence of a path from any vertex in $K_{B}$ to a some vertex
in $K_{W}$. See \prettyref{rem:LS-EMC-inverse-flow-graph-interpretation}.\end{rem*}
\begin{proof}
(of Lemma \ref{lem:LS-EMC-non-zero-path}) Fix $\tilde{K}_{B}$ and
suppose the system is ergodic and it is started with $Z(0)=(0,k)$,
for some $k\in\tilde{K}_{B}$, i.e., with an empty queue and in an
environment state $k$ which blocks the arrival process. From ergodicity
it follows that for some $m\in K_{W}$ must hold 
\[
P(Z(\sigma_{1})=(1,m)|Z(0)=(0,k))>0,
\]
because there is a positive probability for the first arrival of some
customer admitted into the system.

Because no arrival is possible if $m\in K_{B}$, necessarily $m\in K_{W}$
holds, and because up to $\sigma_{1}-$ no departure or arrival could
happen, the only possibility to enter $m$ is by a sequence of transitions
triggered by $\V$. Because $Z$ is regular this sequence is finite
with probability $1$. The path from $k\in\tilde{K}_{B}$ to $m\in K_{W}$
of the directed transition graph of $\V$ contains an edge $k_{1}k_{2}\in\mathcal{E}$
with $k_{1}\in\tilde{K}_{B}$ and $k_{2}\in\tilde{K}_{B}^{c}$.\end{proof}
\begin{lem}
\label{lem:LS-EMC-diag-inv} For any strictly positive $\eta\in\mathbb{R}^{+}$
the matrix $(-diag(\V)+\eta I_{W})$ is invertible.\end{lem}
\begin{proof}
For any $k\in K_{W}$ the corresponding diagonal element of the matrix
$(-diag(\V)+\eta I_{W})$ is greater than $\eta$ because $-\v(k,k)\geq0$.

If $k\in K_{B}$, we utilize the ergodicity of $Z$ in continuous
time and apply Lemma \ref{lem:LS-EMC-non-zero-path} with $\widetilde{K}_{B}:=\{k\}$.
The lemma implies that there is some $m\neq k$ with $\v(k,m)>0$.
It follows $-\v(k,k)>0$.

We conclude that the diagonal matrix $(-diag(\V)+\eta I_{W})$ has
only strictly positive values on its diagonal and therefore it is
invertible. 
\end{proof}

\subsection{Observing the system at departure instants}

\label{sect:LS-EMC} Recall that the paths of $Z$ are cadlag and
that $\tau=(\tau_{0},\tau_{1},\tau_{2},\dots)$ with $\tau_{0}=0$
denotes the sequence of departure times of customers. Then with $\hat{X}(n):=X(\tau_{n})$
and $\hat{Y}(n):=Y(\tau_{n})$ for $n\in\mathbb{N}_{0}$ it is easy
to see that \marginpar{$\hat{X}$}\index{X hat@$\hat{X}$, embedded Markov chain}\marginpar{$\hat{Y}$}\index{Y hat@$\hat{Y}$, embedded Markov chain}\marginpar{$\hat{Z}$}\index{Z hat@$\hat{Z}=(\hat{X},\hat{Y})$, embedded Markov chain}
\begin{equation}
\hat{Z}=((\hat{X}(n),\hat{Y}(n)):n\in\mathbb{N}_{0})\label{eq:LS-EMC-mg1-1}
\end{equation}
is a homogeneous Markov chain on state space $E=\mathbb{N}_{0}\times K$.
If $\hat{Z}$ has a unique stationary distribution, this will be denoted
by $\hat{\pi}$.
\begin{defn}
If the embedded Markov chain $((\hat{X}(n),\hat{Y}(n):n\in\mathbb{N}_{0})$
of the loss system in a random environment with state space $E:=\mathbb{N}_{0}\times K$
has a unique steady state distribution, we denote this steady state
\marginpar{$\hat{\pi}$}\index{pi-hat@$\hat{\pi}$ - steady state distribution of the embedded Markov
chain} 

\begin{eqnarray*}
\hat{\pi} & := & (\hat{\pi}(n,k):(n,k)\in E:=\mathbb{N}_{0}\times K).
\end{eqnarray*}

and the marginal steady state distributions \marginpar{$\hat{\xi}$}
\index{xi-hat@$\hat{\xi}$ - steady state distribution of the embedded Markov
chain for queueing system} 
\begin{eqnarray*}
\hat{\xi} & := & (\xi(n):n\in\mathbb{N}_{0}),\qquad\hat{\xi}(n):=\sum_{k\in K}\hat{\pi}(n,k)
\end{eqnarray*}

\marginpar{$\hat{\theta}$}\index{theta-hat@ $\hat{\theta}$ - steady state distribution of the embedded
Markov chain for environment}
\begin{eqnarray*}
\hat{\theta} & := & (\hat{\theta}(k):k\in K),\qquad\hat{\theta}(k):=\sum_{n\in\mathbb{N}_{0}}\hat{\pi}(n,k)
\end{eqnarray*}

\end{defn}
It will turn out that this Markov chain exhibits interesting structural
properties of the loss systems in random environments. E.g., with
$\xi$ from \eqref{eq:LS-BaD1} we will prove that

\[
\hat{\pi}(n,k)=\hat{\xi}(n)\cdot\hat{\theta}(k)=\xi(n)\cdot\hat{\theta}(k),
\]
holds, but in general we do not have $\hat{\pi}(n,k)>0$ on the global
state space $E$, because $\hat{\theta}(k)=0$ may occur. Especially,
in general it holds $\theta\neq\hat{\theta}$.

The reason for this seems to be the rather general vice-versa interaction
of the queueing system and the environment. Of special importance
is the fact that we consider the continuous time systems at departure
instants where we have the additional information that right now the
influence of the queueing systems on the change of the environment
is in force (described by the stochastic matrix $R$).\\

The dynamics of $\hat{Z}$ will be described in a way that resembles
the $M/G/1$ type matrix analytical models. Recall that the state
space $E$ carries an order structure which will govern the description
of the transition matrix and, later on, of the steady state vector.
\begin{defn}
\label{def:LS-EMC-P-A-n-B-n}We define the one-step transition matrix
$\mathbf{P}$ of $\hat{Z}$ by \marginpar{$\mathbf{P}$}\index{P transition-matrix@$\mathbf{P}$, transition matrix}
\begin{eqnarray*}
 &  & \left(\mathbf{P}_{(i,k),(j,m)}:(i,k),(j,m)\in E\right)\\
 & := & \left(P(Z(\tau_{1})=(j,m)|Z(0)=(i,k)):(i,k),(j,m)\in E\right),
\end{eqnarray*}
and introducing matrices $A^{(i,n)}\in\mathbb{R}^{K\times K}$ and
$B^{(n)}\in\mathbb{R}^{K\times K}$ \index{A in@$A^{(i,n)}$, loss system}
\marginpar{$A^{(i,n)}$}\index{B n@$B^{(n)}$, loss system}\marginpar{$B^{(n)}$}
by 
\begin{eqnarray}
B_{km}^{(n)} & := & P(Z(\tau_{1})=(n,m)|Z(0)=(0,k))\label{eq:LS-EMC-B-n-definition}\\
A_{km}^{(i,n)} & := & P(Z(\tau_{1})=(i+n-1,m)|Z(0)=(i,k)),\qquad1\leq i\label{eq:LS-EMC-A-n-i-definition}
\end{eqnarray}
for $k,m\in K$, the matrix $\mathbf{P}$ has the form

\begin{equation}
\mathbf{P}=\left(\begin{array}{ccccc}
B^{(0)} & B^{(1)} & B^{(2)} & B^{(3)} & \ldots\\
A^{(1,0)} & A^{(1,1)} & A^{(1,2)} & A^{(1,3)} & \ldots\\
0 & A^{(2,0)} & A^{(2,1)} & A^{(2,2)} & \ldots\\
0 & 0 & A^{(3,0)} & A^{(3,1)} & \ldots\\
\vdots & \vdots & \vdots & \vdots
\end{array}\right),\label{eq:LS-EMC-P-matrix}
\end{equation}
which exploits the structure of the state space as a product of level
variables in $\mathbb{N}_{0}$ and phase variables in $K$. We emphasize
that $K$ is ordered for its own, see \eqref{Eorder1}.
\end{defn}
For the loss system in a random environment we will solve the equation
\begin{equation}
\hat{\pi}\mathbf{P}=\hat{\pi}\label{eq:LS-chain-steady-state-equation}
\end{equation}
for a stochastic solution $\hat{\pi}$ which is a steady state distribution
of the embedded Markov chain $\hat{Z}$. Because $\hat{Z}$ is in
general not irreducible on $E$ there are some subtleties with respect
to the uniqueness of a stochastic solution of the equation.

For further calculations it will be convenient to group $\hat{\pi}$
according to the queue length:
\begin{defn}
We write $\hat{\pi}$ as 
\begin{equation}
\hat{\pi}=(\hat{\pi}^{(0)},\hat{\pi}^{(1)},\hat{\pi}^{(2)},\dots)\label{eq:LS-EMC-pistructure5}
\end{equation}
with 
\begin{equation}
\hat{\pi}^{(n)}=(\hat{\pi}(n,k):k\in K_{W},\hat{\pi}(n,k):k\in K_{B}),\quad n\in\mathbb{N}_{0}\,,\label{eq:LS-EMC-pistructure6}
\end{equation}
where we agree that the representation of $(\hat{\pi}(n,k):k\in K_{W},\hat{\pi}(n,k):k\in K_{B})$
respects the ordering of $K$. Especially we write for $(n,k)\in E$
\[
\hat{\pi}^{(n)}(k):=\hat{\pi}(n,k)
\]

\end{defn}
An immediate consequence of this definition is that the steady state
equation \prettyref{eq:LS-chain-steady-state-equation} can be written
as 
\begin{equation}
\hat{\pi}^{(0)}B^{(n)}+\sum_{i=1}^{n+1}\hat{\pi}^{(i)}A^{(i,n-i+1)}=\hat{\pi}^{(n)},\qquad n\in\mathbb{N}_{0}\,.\label{eq:LS-EMC-steady-state-equation-grouped}
\end{equation}

\subsection{Steady state for the system observed at departure instants}

\label{sect:LS-EMC-steadystate} We start our investigation with a
detailed analysis of the one-step transition matrix \eqref{eq:LS-EMC-P-matrix}
and will express the matrices $B^{(n)}$ and $A^{(i,n)}$ from  \prettyref{def:LS-EMC-P-A-n-B-n}
by means of auxiliary matrices $W$ and $U^{(i,n)}$, which reflect
the dynamics of the system.

It turns out that the matrix $(\lambda I_{W}-\V)$ plays a central
role in this analysis and that we shall need especially its inverse.
We therefore set in force for the rest of the paper the technical\\

\noindent \textbf{Overall Assumption (I)}, that the matrix $(\lambda I_{W}-\V)$
is \textbf{invertible}.\\

This assumption is not restrictive for modeling purposes as the following
proposition reveals. Further examples and a discussion can be found
in the Appendix.
\begin{prop}
\label{prop:LS-EMC-matrix-inversibility}Let $Z$ be ergodic with
finite environment space $K$, and $\V$ be the associated generator
driving the continuous changes of the environment. Then for any $\lambda>0$
the matrix $(\lambda I_{W}-\V)$ is invertible. \end{prop}
\begin{proof}
Follows from \prettyref{lem:LS-EMC-non-zero-path} and \prettyref{lem:LS-EMC-finite-K-M-invertible}
from the Appendix.\end{proof}
\begin{prop}
\label{prop:LS-EMC-matrix-inverse-proof-inf} Let $Z$ be ergodic
with environment space $K$ partitioned according to $K=K_{W}+K_{B}$,
with $K_{W}\neq\emptyset$, and with $|K_{B}|<\infty$, and $\lambda>0$
such that $\lambda I_{W}-\V$ is surjective on $\ell_{\infty}(\mathbb{R}^{K})$.

Let the generator matrix $\V:=(\v(k,m):k,m\in K)\in\mathbb{R}^{K\times K}$
be uniformizable, i.e. it holds $\inf_{k\in K}\v(k,k)>-\infty$.

Then the matrix $\lambda I_{W}-\V$ is invertible. \end{prop}
\begin{proof}
It is immediate, that $\lambda I_{W}-\V$ fulfills the assumptions
\eqref{eq:LS-EMC-inverse-inf-1}, \eqref{eq:LS-EMC-inverse-inf-2},
and \eqref{eq:LS-EMC-inverse-inf-3} of Lemma \ref{lem:LS-EMC-M-invertible-inf}
with $\varepsilon(K_{W})=\lambda$. The \underline{flow condition}
holds in this setting from the ergodicity of the continuous time process
with arguments similar to those in the proof of Lemma \ref{lem:LS-EMC-diag-inv}.
We conclude that $M$ is injective.
\end{proof}
In a first step we analyze the dynamics incorporated in the matrix
$A^{(i,n)}$ and $B^{(n)}$.
\begin{lem}
\label{lem:LS-EMC-ABUW} Recall that $\tau_{1}$ denotes the first
departure instant, that $\sigma_{1}$ denotes the first arrival instant
of a customer, and that $Y(\sigma_{1})\in K_{W}$ holds. 

For $k\in K,m\in K_{W}$, we define\index{U in-km@$U_{km}^{(i,n)}$}\marginpar{$U_{km}^{(i,n)}$}
\begin{equation}
U_{km}^{(i,n)}:=P\left(\left(X(\tau_{1}),Y(\tau_{1}^{-})\right)=(n+i-1,m)|Z(0)=(i,k)\right),\qquad1\leq i,n\in\mathbb{N}_{0}\,,\label{eq:LS-EMC-U-definition-1}
\end{equation}
and for $k\in K$ and $m\in K_{B}$ we prescribe by definition $U_{km}^{(i,n)}=0$.

Similarly, for $k\in K,m\in K_{W}$, we define\index{W km@$W_{km}$}\marginpar{$W_{km}$}
\begin{equation}
W_{km}:=P(Z(\sigma_{1})=(1,m)|Z(0)=(1,k))\,,\label{eq:LS-EMC-W-definition-1}
\end{equation}
and for $k\in K$ and $m\in K_{B}$ prescribe by definition $W_{km}=0$.

Then it holds for $A^{(i,n)}$ and $B^{(n)}$ from Definition \ref{def:LS-EMC-P-A-n-B-n}
\begin{eqnarray}
A^{(i,n)} & = & U^{(i,n)}R\,,\label{eq:LS-EMC-A-UR}\\
B^{(n)} & = & WA^{(1,n)}=WU^{(1,n)}R\,.\label{eq:LS-EMC-B-W-A}
\end{eqnarray}
\end{lem}
\begin{proof}
Using the fact, that the paths of the system in continuous time almost
sure have left limits, we get for $i\intpositive,~n\geq0$ and $k,m\in K$
\begin{eqnarray*}
A_{km}^{(i,n)} & = & P\left(\left(X(\tau_{1}),Y(\tau_{1})\right)=(i+n-1,m)|Z(0)=(i,k)\right)\\
 & = & \sum_{h\in K}P\left(\left(X(\tau_{1}),Y(\tau_{1}^{-})\right)=(i+n-1,h)|Z(0)=(i,k)\right)\Rentry hm\\
 & = & \sum_{h\in K}U_{kh}\Rentry hm\,,
\end{eqnarray*}
which in matrix form is \prettyref{eq:LS-EMC-A-UR}.

For the property \prettyref{eq:LS-EMC-B-W-A} we will use the fact,
that if the system starts with an empty queue, then the first arrival
occurs always before the first departure, $P(\sigma_{1}<\tau_{1})=1$.
We obtain for $n\geq0$ and $k,m\in K$ 
\begin{eqnarray*}
B_{km}^{(n)}: & = & P(\left(X(\tau_{1}),Y(\tau_{1})\right)=(n,m)|Z(0)=(0,k))\\
 & = & \sum_{h\in K}P\left(\left(X(\tau_{1}),Y(\tau_{1})\right)=(n,m)\cap Z(\sigma_{1})=(1,h)|Z(0)=(0,k)\right)\\
 & = & \sum_{h\in K}P\left(\left(X(\tau_{1}),Y(\tau_{1})\right)=(1+n-1,m)|Z(\sigma_{1})=(1,h)\cap Z(0)=(0,k)\right)\\
 &  & \cdot P\left(Z(\sigma_{1})=(1,h)|Z(0)=(0,k)\right)\\
 & \overset{SM}{=} & \sum_{h\in K}\underbrace{P\left(\left(X(\tau_{1}),Y(\tau_{1})\right)=(1+n-1,m)|Z(0)=(1,h)\right)}_{=U_{hm}^{(1,n)}}\\
 &  & \cdot\underbrace{P\left(Z(\sigma_{1})=(1,h)|Z(0)=(0,k)\right)}_{=W_{kh}}\\
 & = & \sum_{h\in K}W_{kh}A_{hm}^{(1,n)}\,,
\end{eqnarray*}
which proves \prettyref{eq:LS-EMC-B-W-A}.
\end{proof}
The proof of Lemma \ref{lem:LS-EMC-ABUW} reveals that the stochastic
matrix $W$ describes the system's development (queue length $\hat{X}$
and environment $\hat{Y}$ process) if it is started empty, until
the next customer enters the system.

The matrix $U^{(i,n)}$ describes the system's development from start
of the ongoing service time of the, say, n-th admitted customer, until
time $\tau_{n}-$; to be more precise, we describe an ongoing service
and the subsequent departure but without the immediately following
jump of the environment triggered by $R$. 

We will use the following properties of the system and its describing
process $Z$: 
\begin{itemize}
\item the strong Markov (\emph{SM}) property of $Z$, 
\item skip free to the left (\emph{SF}) property of the system 
\begin{equation}
P(Z(\zeta_{1})=(n+j,m)|Z(0)=(n,k))=0\qquad\forall j\geq2.\label{eq:LS-EMC-skip-free-property}
\end{equation}

\item cadlag paths; in particular we are interested in the values of $Y(\tau_{1}{-})$,
just before departure instants. 
\end{itemize}
We furthermore have to take into account that matrix multiplication
in general is not commutative. We write therefore $\prod_{j=i}^{n+1}B_{i}$
by definition for $B_{i}B_{i+1}...B_{n+1}$.
\begin{prop}
\label{prop:LCS-W-matrix}For the matrix $W=(W_{km}:k,m\in K)$ from
Lemma \ref{lem:LS-EMC-ABUW} it holds 
\[
W=\lambda(\lambda I_{W}-\V)^{-1}I_{W}
\]
\end{prop}
\begin{proof}
Recall that Let $\sigma_{1}$ denote the arrival time of the first
customer which is admitted to the system, which implies that at time
$\sigma_{1}$ the environment is in a non-blocking state, and $\zeta_{1}$
is the first jump time of the system which can be triggered only by
$\V$ or by an arrival conditioned on $\hat{Y}$ being in $K_{W}$.
It follows for $m\in K_{W}$ 
\begin{eqnarray*}
 &  & W_{km}\\
 & = & P(Z(\sigma_{1})=(1,m)|Z(0)=(0,k))\\
 & = & \sum_{h\in K\backslash\{k\}}P(Z(\sigma_{1})=(1,m)\cap Z(\zeta_{1})=(0,h)|Z(0)=(0,k))\\
 &  & +\delta_{km}P(Z(\zeta_{1})=(1,m)|Z(0)=(0,k))\\
 & = & \sum_{h\in K\backslash\{k\}}P(Z(\sigma_{1})=(1,m)|Z(\zeta_{1})=(0,h),Z(0)=(0,k))P(Z(\zeta_{1})=(0,h)|Z(0)=(0,k))\\
 &  & +\delta_{km}P(Z(\zeta_{1})=(1,m)|Z(0)=(0,k))\\
 & \overset{SM}{=} & \sum_{h\in K\backslash\{k\}}\underbrace{P(Z(\sigma_{1})=(1,m)|Z(0)=(0,h))}_{W_{hm}}\underbrace{P(Z(\zeta_{1})=(0,h)|Z(0)=(0,k))}_{=\frac{\v({k,h})}{-\v({k,k})+\lambda1_{[k\in K_{W}]}}}\\
 &  & +\delta_{km}\underbrace{P(Z(\zeta_{1})=(1,m)|Z(0)=(0,k))}_{=\frac{\lambda}{-\v({k,k})+\lambda1_{[k\in K_{W}]}}1_{[k\in K_{W}]}}
\end{eqnarray*}

The equation above can be written in matrix form 
\begin{eqnarray*}
W & = & (-diag(\V)+\lambda I_{W})^{-1}\left((\V-diag(\V))W+\lambda I_{W}\right)\\
\Longleftrightarrow(-diag(\V)+\lambda I_{W})W & = & (\V-diag(\V))W+\lambda I_{W}\\
\Longleftrightarrow(\lambda I_{W}-\V)W & = & \lambda I_{W}\Longrightarrow W=\lambda(\lambda I_{W}-\V)^{-1}I_{W}
\end{eqnarray*}
\end{proof}
\begin{prop}
\label{prop:LCS-U-matrix}For the matrices $U^{(i,n)}=(U_{km}^{(i,n)}:k,m\in K)$
from Lemma \ref{lem:LS-EMC-ABUW} it holds 
\begin{equation}
U^{(i,0)}=((\lambda+\mu(i))I_{W}-\V)^{-1}\mu(i)I_{W},\quad1\leq i,\label{eq:LS-EMC-U-n-rekursiv-rekursion-begin}
\end{equation}
and for $1\leq i,n\in\mathbb{N}_{0}$, 
\begin{eqnarray}
U^{(i,n+1)} & = & U^{(i,n)}\left(\frac{\lambda}{\mu(n+i)}\right)\mu(n+1+i)\left(\lambda I_{W}+\mu(n+1+i)I_{W}-\V\right)^{-1},\label{eq:LS-EMC-U-n-rekursiv}
\end{eqnarray}
\end{prop}
\begin{proof}
Note that $\tau_{1}$ is the first departure time and$\zeta_{1}$
is the first jump time of the system, and if this jump is triggered
by a departure than $\zeta_{1}=\tau_{1}$.

For $U^{(i,0)}$ with $i\intpositive$ it holds for $k\in K$ and
$m\in K_{W}$:

\begin{eqnarray*}
 &  & U_{km}^{(i,0)}\\
 & = & P\left(\left(X(\tau_{1}),Y(\tau_{1}^{-})\right)=(i-1,m)|Z(0)=(i,k)\right)\\
 & = & \sum_{h\in K\backslash\{k\}}P\left(\left(X(\tau_{1}),Y(\tau_{1}^{-})\right)=(i-1,m)\cap Z(\zeta_{1})=(i,h)|Z(0)=(i,k)\right)\\
 &  & +\delta_{km}P\left(\left(X(\tau_{1}),Y(\tau_{1}^{-})\right)=(i-1,k)|Z(0)=(i,k)\right)\\
 & = & \sum_{h\in K\backslash\{k\}}P\left(\left(X(\tau_{1}),Y(\tau_{1}^{-})\right)=(i-1,m)|Z(\zeta_{1})=(i,h),Z(0)=(i,k)\right)\\
 &  & \cdot P\left(Z(\zeta_{1})=(i,h)|Z(0)=(i,k)\right)\\
 &  & +\delta_{km}P\left(\left(X(\tau_{1}),Y(\tau_{1}^{-})\right)=(i-1,k)|Z(0)=(i,k)\right)\\
 & = & \sum_{h\in K\backslash\{k\}}P\left(\left(X(\tau_{1}),Y(\tau_{1}^{-})\right)=(i-1,m)|Z(0)=(i,h)\right)\frac{\v({k,h})}{-\v({k,k})+(\lambda+\mu(i))1_{[k\in K_{W}]}}\\
 &  & +\delta_{km}\frac{\mu(i)}{-\v({k,k})+(\lambda+\mu(i))1_{[k\in K_{W}]}}\\
 & = & \frac{1}{-\v({k,k})+(\lambda+\mu(i))1_{[k\in K_{W}]}}\left(\sum_{h\in K\backslash\{k\}}\v({k,h})U_{hk}^{(i,0)}+\delta_{km}\mu(i)1_{[k\in K_{W}]}\right)
\end{eqnarray*}

We write the equation above in a matrix form

\begin{eqnarray*}
U^{(i,0)}~~=~~\left(-diag(\V)+(\lambda+\mu(i))I_{W}\right)^{-1} & \cdot & \left((\V-diag(\V))U^{(i,0)}+\mu(i)I_{W}\right)\\
\Longleftrightarrow(-diag(\V)+(\lambda+\mu(i))I_{W})U^{(i,0)} & = & ((\V-diag(\V))U^{(i,0)}+\mu(i)I_{W})\\
\Longleftrightarrow((\lambda+\mu(i))I_{W}-\V)U^{(i,0)} & = & \mu(i)I_{W}\\
\Longrightarrow U^{(i,0)} & = & ((\lambda+\mu(i))I_{W}-\V)^{-1}\mu(i)I_{W}
\end{eqnarray*}

Next we calculate for $n\geq0$ and $1\leq i$ the elements of the
matrix $U_{km}^{(i,n+1)}$

\begin{eqnarray*}
 &  & U_{km}^{(i,n+1)}\\
 & = & P\left(\left(X(\tau_{1}),Y(\tau_{1}^{-})\right)=(n+1+i-1,m)|Z(0)=(i,k)\right)\\
 & = & \sum_{j=0}^{n+1}\sum_{h\in K}1_{[(j,h)\neq(i,k)]}P\left(\left(X(\tau_{1}),Y(\tau_{1}^{-})\right)=(n+i,m)\cap Z(\zeta_{1})=(j,h)|Z(0)=(i,k)\right)\\
 & \overset{SF}{=} & \sum_{h\in K\backslash\{h\}}P\left(\left(X(\tau_{1}),Y(\tau_{1}^{-})\right)=(n+i,m)\cap Z(\zeta_{1})=(i,h)|Z(0)=(i,k)\right)\qquad\\
 &  & +P\left(\left(X(\tau_{1}),Y(\tau_{1}^{-})\right)=(n+i,m)\cap Z(\zeta_{1})=(i+1,k)|Z(0)=(i,k)\right)\\
 & = & \sum_{h\in K\backslash\{h\}}P\left(\left(X(\tau_{1}),Y(\tau_{1}^{-})\right)=(n+i,m)|Z(\zeta_{1})=(i,h),Z(0)=(i,k)\right)\\
 &  & \cdot P\left(Z(\zeta_{1})=(i,h)|Z(0)=(i,k)\right)\\
 &  & +P\left(\left(X(\tau_{1}),Y(\tau_{1}^{-})\right)=(n+i,m)|Z(\zeta_{1})=(i+1,k),Z(0)=(i,k)\right)\\
 &  & \cdot P\left(\left(X(\tau_{1}),Y(\tau_{1}^{-})\right)=(i+1,k)|Z(0)=(i,k)\right)\\
 & \overset{SM}{=} & \sum_{h\in K\backslash\{h\}}P\left(\left(X(\tau_{1}),Y(\tau_{1}^{-})\right)=(n+i,m)|Z(0)=(i,h)\right)P(\left(Z(\zeta_{1})=(i,h)|Z(0)=(i,k)\right))\\
 &  & +P\left(\left(X(\tau_{1}),Y(\tau_{1}^{-})\right)=(n+i,m)|Z(\zeta_{1})=(i+1,k)\right)P\left(Z(\zeta_{1})=(i+1,k)|Z(0)=(i,k)\right)\\
 & = & \sum_{h\in K\backslash\{h\}}P(\left(X(\tau_{1}),Y(\tau_{1}^{-})\right)=(n+i,m)|Z(0)=(i,h))\\
 &  & \cdot\frac{\v({k,h})}{-\v({k,k})+(\mu(i)+\lambda)1_{[k\in K_{W}]}}\\
 &  & +P(\left(X(\tau_{1}),Y(\tau_{1}^{-})\right)=(n+i+1-1,m)|Z(0)=(i+1,k))\\
 &  & \cdot1_{[k\in K_{W}]}\frac{\lambda}{-\v({k,k})+(\mu(i)+\lambda)1_{[k\in K_{W}]}}
\end{eqnarray*}

The last equation can be written in matrix form as

\begin{eqnarray*}
U^{(i,n+1)}=\left(-diag(\V)+(\lambda+\mu(i))I_{W}\right)^{-1} & \cdot & \left((\V-diag(\V)U^{(i,n+1)}+\lambda I_{W}U^{(i+1,n)}\right)\\
\Longleftrightarrow\left(-diag(\V)+(\lambda+\mu(i))I_{W}\right)U^{(i,n+1)} & = & \left(\V-diag(\V)\right)U^{(i,n+1)}+\lambda I_{W}U^{(i+1,n)}\\
\Longleftrightarrow((\lambda+\mu(i))I_{W}-\V)U^{(i,n+1)} & = & \lambda I_{W}U^{(i+1,n)}\\
\Longrightarrow U^{(i,n+1)} & = & ((\lambda+\mu(i))I_{W}-\V)^{-1}\lambda I_{W}U^{(i+1,n)}
\end{eqnarray*}

Iterating the last equation $n$-times and then applying \eqref{eq:LS-EMC-U-n-rekursiv-rekursion-begin}
leads with $I_{W}^{2}=I_{W}$ to

\begin{eqnarray}
U^{(i,n+1)} & = & \prod_{j=i}^{n+1+i}\left[\lambda\left((\lambda+\mu(j))I_{W}-\V\right)^{-1}I_{W}\right]\frac{\mu(n+1+i)}{\lambda}\Longrightarrow\nonumber \\
U^{(i,n+1)} & = & U^{(i,n)}\frac{\lambda}{\mu(n+i)}\mu(n+1+i)\left((\lambda+\mu(n+1+i))I_{W}-\V\right)^{-1}I_{W}\label{eq:LS-EMC-U-n-rekursiv-2}
\end{eqnarray}

Finally we verify that the recursion \prettyref{eq:LS-EMC-U-n-rekursiv-2}
is compatible with \eqref{eq:LS-EMC-U-n-rekursiv-rekursion-begin}:

\begin{eqnarray*}
U^{(i,1)} & = & \underbrace{\lambda\left((\lambda+\mu(i))I_{W}-\V\right)^{-1}}_{U^{(i,0)}\frac{\lambda}{\mu(i)}}I_{W}\lambda\left((\lambda+\mu(i+1))I_{W}-\V\right)^{-1}I_{W}\frac{\mu(1+i)}{\lambda}\\
 & = & U^{(i,0)}\frac{\lambda}{\mu(i)}\mu(1+i)\left((\lambda+\mu(1+i))I_{W}-\V\right)^{-1}I_{W}
\end{eqnarray*}

\end{proof}
We are now ready to evaluate the steady state equations \prettyref{eq:LS-EMC-steady-state-equation-grouped}
of $\hat{Z}$. Because we have a Poisson-$\lambda$ arrival stream,
the marginal steady state \eqref{eq:LS-BaD1} of the continuous time
queue length process $X$ is 
\[
\xi=\left(\xi(n):=C^{-1}\prod_{i=1}^{n}\frac{\lambda}{\mu(i)}:n\in\mathbb{N}_{0}\right)
\]
Recall \eqref{eq:LS-EMC-steady-state-equation-grouped} 
\[
\hat{\pi}^{(0)}B^{(n)}+\sum_{i=1}^{n+1}\hat{\pi}^{(i)}A^{(i,n-i+1)}=\hat{\pi}^{(n)},\qquad n\in\mathbb{N}_{0}\,,
\]
and the decomposition from Lemma \ref{lem:LS-EMC-ABUW}: 
\[
A^{(i,n)}=U^{(i,n)}R,~~~\text{and}~~~B^{(n)}=WU^{(1,n)}R\,.
\]

The \emph{conjectured product form steady state} will eventually be
realized as 
\[
\hat{\pi}(n,k)=\xi(n)\cdot\hat{\theta}(k),~\text{for}~~(n,k)\in E,~\text{and}~~\hat{\pi}^{n}=\xi(n)\cdot\hat{\theta},~\text{for}~~n\in\mathbb{N}_{0},
\]
with $\hat{\theta}(k)=0$ for some $k\in K.$

The \emph{idea of the proof }is: The steady state equation is transformed
into 
\begin{equation}
\xi(n)\cdot\hat{\theta}=\xi(0)\cdot\hat{\theta}\cdot W\cdot U^{(1,n)}\cdot R+\sum_{i=1}^{n+1}\xi(n)\cdot\hat{\theta}\cdot U^{(i,n-i+1)}\cdot R,~~~n\in\mathbb{N}_{0}.\label{steadyst1}
\end{equation}
We insert $\xi(n)$, cancel $C^{-1}$, and obtain the ''\emph{environment
equations}'' 
\begin{eqnarray}
\hat{\theta} & = & \hat{\theta}\cdot W\cdot U^{(1,0)}\cdot R+\left(\frac{\lambda}{\mu(1)}\right)\cdot\hat{\theta}\cdot U^{(i,n-i+1)}\cdot R\,,\label{steadyst2}\\
 &  & \text{and for}~~n\intpositive\nonumber \\
\left(\prod_{i=1}^{n}\frac{\lambda}{\mu(i)}\right)\cdot\hat{\theta} & = & \hat{\theta}\cdot W\cdot U^{(1,n)}\cdot R+\sum_{i=1}^{n+1}\left(\prod_{j=1}^{i}\frac{\lambda}{\mu(i)}\right)\cdot\hat{\theta}\cdot U^{(i,n-i+1)}\cdot R,,\label{steadyst2a}
\end{eqnarray}
which we may consider as a sequence of equations with vector of unknowns
$\hat{\theta}$. The obvious problem with this system, namely, having
an infinite sequence of equations for the same solution, is resolved
by the following lemma.
\begin{lem}
\label{lem:LS-EMC-Mn}For $n\in\mathbb{N}_{0}$ denote 
\[
M^{(n)}:=WU^{(1,n)}+\sum_{i=1}^{n+1}\left(\prod_{j=1}^{i}\frac{\lambda}{\mu(i)}\right)\cdot U^{(i,n-i+1)}\,.
\]
Then it holds 
\begin{eqnarray}
M{}^{(0)} & = & \lambda(\lambda I_{W}-\V)^{-1}I_{W}\label{eq:LS-EMC-steadyst4}\\
M^{(n)} & = & \left(\prod_{i=1}^{n}\frac{\lambda}{\mu(i)}\right)\lambda(\lambda I_{W}-\V)^{-1}I_{W},\quad n\intpositive\,,\label{eq:LS-EMC-M-n-explicite}
\end{eqnarray}
and consequently 
\begin{equation}
M^{(n)}=\left(\prod_{i=1}^{n}\frac{\lambda}{\mu(i)}\right)M^{(0)},\qquad n\intpositive\,.\label{eq:LS-EMC-M-n-explicite-with-xi}
\end{equation}
\end{lem}
\begin{proof}
We show that \eqref{eq:LS-EMC-steadyst4} holds and compute directly

\begin{eqnarray*}
M{}^{(0)} & = & WU^{(1,0)}+\frac{\lambda}{\mu(1)}U^{(1,0)}=\left(W+\frac{\lambda}{\mu(1)}I\right)U^{(1,0)}\\
 & = & \left(\lambda(\lambda I_{W}-\V)^{-1}I_{W}+\frac{\lambda}{\mu(1)}I\right)\left(-\V+(\mu(1)+\lambda)I_{W}\right)^{-1}\mu(1)I_{W}\\
 & = & \frac{\lambda}{\mu(1)}(\lambda I_{W}-\V)^{-1}\left(\mu(1)I_{W}-(\lambda I_{W}-\V)\right)(\V-(\mu(1)+\lambda)I_{W})^{-1}\mu(1)I_{W}\\
 & = & \lambda(\lambda I_{W}-\V)^{-1}I_{W}
\end{eqnarray*}
Assume now, that for $n\geq0$ \prettyref{eq:LS-EMC-M-n-explicite}
holds for $M^{(n)}$. Then 
\begin{eqnarray*}
 &  & M^{(n+1)}\\
 & = & WU^{(n+1,1)}+\sum_{i=1}^{n+2}\left(\prod_{j=1}^{i}\frac{\lambda}{\mu(j)}\right)U^{(i,n-i+2)}\\
 & = & \underbrace{\left(WU^{(1,n)}+\sum_{i=1}^{n+1}\left(\prod_{j=1}^{i}\frac{\lambda}{\mu(j)}\right)^{n}U^{(i,n-i+1)}\right)\frac{\mu(n+2)}{\mu(n+1)}\lambda}_{=M^{(n)}\frac{\mu(n+2)}{\mu(n+1)}\lambda}\left(\lambda I_{W}+\mu(n+2)I_{W}-\V\right)^{-1}I_{W}\\
 &  & +\left(\prod_{i=1}^{n+2}\frac{\lambda}{\mu(i)}\right)\mu(n+2)(\lambda I_{W}+\mu(n+2)I_{W}-\V)^{-1}I_{W}\\
 & = & \left(\frac{\lambda}{\mu(n+1)}\mu(n+2)M^{(n)}+\left(\prod_{i=1}^{n+2}\frac{\lambda}{\mu(i)}\right)\mu(n+2)I\right)\left(\lambda I_{W}+\mu(n+2)I_{W}-\V\right)^{-1}I_{W}\\
 & = & \left(\prod_{i=1}^{n+1}\frac{\lambda}{\mu(i)}\right)\left(\lambda(\lambda I_{W}-\V)^{-1}I_{W}\mu(n+2)+\lambda I\right)\left(\lambda I_{W}+\mu(n+2)I_{W}-\V\right)^{-1}I_{W}\\
 & = & \lambda\left(\prod_{i=1}^{n+1}\frac{\lambda}{\mu(i)}\right)(\lambda I_{W}-\V)^{-1}\left(I_{W}\mu(n+2)+(\lambda I_{W}-\V)\right)\left(\lambda I_{W}+\mu(n+2)I_{W}-\V\right)^{-1}I_{W}\\
 & = & \lambda\left(\prod_{i=1}^{n+1}\frac{\lambda}{\mu(i)}\right)(\lambda I_{W}-\V)^{-1}\left(\mu(n+2)I_{W}+(\lambda I_{W}-\V)\right)\left(\mu(n+2)I_{W}+(\lambda I_{W}-\V)\right)^{-1}I_{W}\\
 & = & \lambda\left(\prod_{i=1}^{n+1}\frac{\lambda}{\mu(i)}\right)(\lambda I_{W}-\V)^{-1}I_{W}
\end{eqnarray*}

\end{proof}
Note, that with the definitions in Lemma \ref{lem:LS-EMC-Mn} the
sequence of \emph{environment equations} \eqref{steadyst2} and \eqref{steadyst2a}
reduces to

\[
\hat{\theta}=\hat{\theta}\cdot M^{(0)}\cdot R\,,\quad\text{and for}~~n\intpositive:~~~\left(\prod_{i=1}^{n}\frac{\lambda}{\mu(i)}\right)\cdot\hat{\theta}=\hat{\theta}\cdot M^{(n)}\cdot R\,,
\]
and the result in \eqref{eq:LS-EMC-M-n-explicite-with-xi} says, that
all these equations are compatible, in fact, they are the same. Therefore
the next lemma will open the path to our main theorem by providing
a common solution to all the environment equations. 
\begin{lem}
\label{lem:LS-EMC-theta-theta-hat-releationship} \textbf{(a)} The
matrix $M^{(0)}$ is stochastic, i.e., it holds 
\begin{equation}
M^{(0)}\mathbf{e}=\mathbf{e}\quad\text{and}\quad M_{km}^{(0)}\geq0,\qquad\forall k,m\in K\,.\label{eq:LCS-M0-row-sums}
\end{equation}

\textbf{(b)} If the continuous time process $Z$ is ergodic with product
form steady state $\pi$ with 
\begin{equation}
{\pi}(n,k)=\xi(n){\theta}(k),\quad(n,k)\in E\,,\label{eq:LCS-mm1inf-pi-hat-product-form-11}
\end{equation}
then, with the marginal stationary distribution $\theta$ of $Y$
in continuous time, 
\begin{equation}
\hat{\theta}=(\theta I_{W}\mathbf{e})^{-1}\cdot\theta I_{W}R\label{eq:LCS-theta-vs-theta-check}
\end{equation}
is a stochastic solution of the equation 
\begin{equation}
\hat{\theta}\underbrace{\lambda(\lambda I_{W}-\V)^{-1}I_{W}}_{=M^{(0)}}R=\hat{\theta}\,.\label{eq:LCS-existence-theta-hat-M0-R-eq-theta-hat}
\end{equation}
If $\left(I_{W}-\frac{1}{\lambda}V\right)^{-1}$ is injective, the
stochastic solution $\hat{\theta}$ of the equation \eqref{eq:LCS-existence-theta-hat-M0-R-eq-theta-hat}
is unique.

\textbf{(c)} Let $L:=\{k\in K:\exists~~m\in K_{W}:R_{m,k}>0\}$ the
set of states of the environment which can be reached from $K_{W}$
by a one-step jump governed by $R$. Then it holds 
\begin{equation}
\hat{\theta}(k)=0,\quad\forall k\in K\setminus L\,.\label{eq:stochsolution2}
\end{equation}
\textbf{(d)} If $\hat{\theta}$ is a stochastic solution of \eqref{eq:LCS-existence-theta-hat-M0-R-eq-theta-hat}
then $x$ defined as

\begin{eqnarray}
x & := & \hat{\theta}\left(I_{W}-\frac{1}{\lambda}\V\right)^{-1}\label{eq:LCS-x-theta-hat}
\end{eqnarray}

is a solution of the steady state equation \eqref{eq:LS-constant-lambda-theta-matrix-equation}
of the continuous time process $(X,Y)$. Therefore the uniquely determined
stationary distribution of \eqref{eq:LS-constant-lambda-theta-matrix-equation}
is

\begin{eqnarray}
\theta & = & \left(\hat{\theta}\left(I_{W}-\frac{1}{\lambda}\V\right)^{-1}\mathbf{e}\right)^{-1}\hat{\theta}\left(I_{W}-\frac{1}{\lambda}\V\right)^{-1}\label{eq:LCS-finit-K-unique-theta-hat}
\end{eqnarray}
\end{lem}
\begin{proof}
\textbf{(a)} Recall, that the matrix $\V$ is a generator, so is $\V\mathbf{e}=0$
and 
\begin{equation}
(\lambda I_{W}-\V)\mathbf{e}=\lambda I_{W}\mathbf{e}\label{eq:LCS-exitstenz-lambda-I-w}
\end{equation}
Applying \prettyref{eq:LCS-exitstenz-lambda-I-w} to 
\[
M^{(0)}\mathbf{e}=(\lambda I_{W}-\V)^{-1}\lambda I_{W}\mathbf{e}=(\lambda I_{W}-\V)^{-1}(\lambda I_{W}-\V)\mathbf{e}
\]
proves the first part of the statement \prettyref{eq:LCS-M0-row-sums}.

In the \prettyref{lem:LS-EMC-Mn} we defined the matrix $M{}^{(0)}$
as $WU^{(0,1)}+\frac{\lambda}{\mu(1)}U^{(0,1)}$, where the entries
of $W$ and $U^{(0,1)}$ are probabilities. Because $\frac{\lambda}{\mu(1)}$
is positive, the matrix $M^{(0)}$ is non negative, and $M^{(0)}\cdot R$
as well. 

\textbf{(b)} Due to ergodicity of $Z$ with product form steady state,
$\theta$ is the unique stochastic solution of (see \eqref{eq:LS-constant-lambda-theta-matrix-equation}
in \prettyref{cor:LS-lambda-constant}) 
\begin{equation}
\theta\left(\lambda\left(I_{W}R-I_{W}\right)+\V\right)=0\,.\label{steadystatecontinuous12}
\end{equation}
To prove the existence of a stochastic solution of \eqref{eq:LCS-existence-theta-hat-M0-R-eq-theta-hat}
we rewrite \eqref{steadystatecontinuous12} as 
\[
\theta(\lambda(I_{W}R-I_{W})+\V)=0~~~\Longleftrightarrow~~~\lambda\theta I_{W}R=\theta(\lambda I_{W}-\V)
\]
Multiplying both sides of the last equation with $(\lambda I_{W}-\V)^{-1}I_{W}R$
leads to 
\begin{eqnarray*}
\lambda\theta I_{W}R(\lambda I_{W}-\V)^{-1}I_{W}R & = & \theta(\lambda I_{W}-\V)(\lambda I_{W}-\V)^{-1}I_{W}R\\
\Longrightarrow({\theta I_{W}R})\lambda(\lambda I_{W}-\V)^{-1}I_{W}R & = & ({\theta I_{W}R})
\end{eqnarray*}
One can see that $\theta I_{W}R$ solves the steady state equation
\eqref{eq:LCS-existence-theta-hat-M0-R-eq-theta-hat} and is therefore
after normalization a stationary distribution of $M^{(0)}\cdot R$.
The normalization constant is 
\[
\theta I_{W}\mathbf{e}=\theta I_{W}\underbrace{R\mathbf{e}}_{=\mathbf{e}}\,.
\]
To prove uniqueness of $\hat{\theta}$ we assume that $\hat{\theta}_{1}$
and $\hat{\theta}_{2}$ are different non-zero solutions of the equation
\eqref{eq:LCS-existence-theta-hat-M0-R-eq-theta-hat} and define 
\begin{eqnarray}
x_{1} & := & \hat{\theta}_{1}\left(I_{W}-\frac{1}{\lambda}\V\right)^{-1}\\
x_{2} & := & \hat{\theta}_{2}\left(I_{W}-\frac{1}{\lambda}\V\right)^{-1}
\end{eqnarray}

Both $x_{1}$ and $x_{2}$ are solutions of the continuous time steady
state equation \eqref{eq:LS-constant-lambda-theta-matrix-equation}.
Due to ergodicity of the process $(X,Y)$ and its product form stationary
distribution there exits some constant $c$ such that $x_{1}=cx_{2}$
holds.

\begin{eqnarray}
x_{1}-cx_{2} & = & 0\\
\Longleftrightarrow(\hat{\theta}_{1}-c\hat{\theta}_{2})\left(I_{W}-\frac{1}{\lambda}\V\right)^{-1} & = & 0
\end{eqnarray}

Because $\left(I_{W}-\frac{1}{\lambda}\V\right)^{-1}$ is injective
it follows $\hat{\theta}_{1}-c\hat{\theta}_{2}=0$ and thus the uniqueness
of the stochastic solution $\hat{\theta}$ of \eqref{eq:LCS-existence-theta-hat-M0-R-eq-theta-hat}.

\textbf{(c)} Denote $\hat{\phi}:=\hat{\theta}{\lambda(\lambda I_{W}-\V)^{-1}I_{W}}.$
Because $I_{W}$ has zero $K_{B}$-columns, the matrix ${\lambda(\lambda I_{W}-\V)^{-1}I_{W}}$
has the same property and therefore $\hat{\phi}(k)=0$ for all $k\in K_{B}$.
It follows for all $k\in K$ 
\[
\hat{\theta}(k)=\sum_{m\in K}\hat{\phi}(m)\Rentry mk=\sum_{m\in K_{W}}\hat{\phi}(m)\Rentry mk\,,
\]
which is by definition not zero only if $k\in L$.

\textbf{(d)} We show that $x$ defined in \eqref{eq:LCS-x-theta-hat}
is a solution of the continuous time steady state equation \eqref{eq:LS-constant-lambda-theta-matrix-equation}.

We write \eqref{eq:LCS-existence-theta-hat-M0-R-eq-theta-hat} in
the following form and assume that $\hat{\theta}$ is any stochastic
solution 
\begin{eqnarray}
\hat{\theta}\left(I_{W}-\frac{1}{\lambda}V\right)^{-1}I_{W}R & = & \hat{\theta}\label{eq:LCS-uniqueness-theta-hat-equation-2}
\end{eqnarray}

We multiply at the right-hand side of the equation the identity matrix,
and obtain

\[
\hat{\theta}\left(I_{W}-\frac{1}{\lambda}\V\right)^{-1}I_{W}R=\hat{\theta}\underbrace{\left(I_{W}-\frac{1}{\lambda}\V\right)^{-1}\left(I_{W}-\frac{1}{\lambda}\V\right)}_{=I}
\]

and rewrite it as 
\begin{equation}
xI_{W}R=x\left(I_{W}-\frac{1}{\lambda}\V\right)\label{eq:LCS-x-equation}
\end{equation}
with 
\[
x:=\hat{\theta}\left(I_{W}-\frac{1}{\lambda}\V\right)^{-1}
\]

The equation \eqref{eq:LCS-x-equation} can be transformed directly
into the continuous steady state equation \eqref{eq:LS-constant-lambda-theta-matrix-equation}.

\[
\Longleftrightarrow x\left(I_{W}R-I_{W}+\frac{1}{\lambda}\V\right)=0\Longleftrightarrow x\left(\lambda(R_{W}-I_{W})+\V\right)=0
\]
\end{proof}
\begin{thm}
\label{thm:LCS-EMC-product-form-exp-service}Consider the ergodic
Markov process $Z=(Z(t):t\geq0)$ which describes the $M/M/1/\infty$
loss system in a random environment.

\textbf{(a)} The Markov chain $\hat{Z}=(\hat{Z}(n):n\in\mathbb{N}_{0})$
embedded at departure instants of $Z$ has a stationary distribution
$\hat{\pi}$ of product form 
\begin{equation}
\hat{\pi}(n,k)=\xi(n)\hat{\theta}(k),\quad(n,k)\in E\,.\label{eq:LCS-mm1inf-pi-hat-product-form}
\end{equation}
Here $\xi=(\xi(n):n\in\mathbb{N}_{0})$ is the probability distribution
\begin{eqnarray}
\xi(n) & := & C^{-1}\left(\prod_{i=1}^{n}\frac{\lambda}{\mu(i)}\right),~~n\in\mathbb{N}_{0},\label{eq:LCS-mm1inf-xi-n-equations}
\end{eqnarray}
with normalization constant $C^{-1}$ and $\hat{\theta}$ is the stochastic
solution \eqref{eq:LCS-theta-vs-theta-check} of the equation 
\begin{equation}
\hat{\theta}{\lambda}({\lambda}I_{W}-\V)^{-1}I_{W}R=\hat{\theta}\,,\label{eq:LCS-mm1inf-theta-hat-equations}
\end{equation}
which is independent of the values of $\mu(n)$.

\textbf{(b)} Let $L:=\{k\in K:\exists~~m\in K_{W}:R_{m,k}>0\}$ the
set of states of the environment which can be reached from $K_{W}$
by a one-step jump governed by $R$. Then the states in $\mathbb{N}_{0}\times(K\setminus L)$
are inessential for $\hat{Z}$ and consequently for all $n\in\mathbb{N}_{0}$
\begin{equation}
\hat{\pi}(n,k)=0,\quad\forall k\in(K\setminus L)\label{inessential1}
\end{equation}
\end{thm}
\begin{proof}
We show that the product form distribution \prettyref{eq:LCS-mm1inf-pi-hat-product-form}
with marginal distributions \prettyref{eq:LCS-mm1inf-xi-n-equations}
and the solution $\hat{\theta}$ of \prettyref{eq:LCS-mm1inf-theta-hat-equations}
solves the steady state equations \prettyref{eq:LS-EMC-steady-state-equation-grouped}
for $n=0$.

\[
\hat{\pi}^{(0)}B^{(0)}+\hat{\pi}^{(1)}A^{(0,1)}=\hat{\pi}^{(0)}~~~\Longleftrightarrow~~~\hat{\theta}\big(\xi(0)B^{(0)}+\xi(1)A^{(0,1)}\big)=\xi(0)\hat{\theta}
\]

With matrices $W$, $U^{(0,1)}$, $R$ and $M^{(0)}$ this equation
can be written as 
\begin{equation}
\xi(0)\hat{\theta}WU^{(0,1)}R+\xi(1)\hat{\theta}U^{(0,1)}R=\xi(0)\hat{\theta}~~~\Longleftrightarrow~~~\hat{\theta}M^{(0)}R=\hat{\theta}\,,\label{eq:LCS-chain-theta-m0-r-solution}
\end{equation}
which has a stochastic solution $\hat{\theta}$ according to Lemma
\ref{lem:LS-EMC-theta-theta-hat-releationship}~\textbf{(a)}.

We finally show that $\hat{\pi}^{(n)}=\xi(n)\hat{\theta}$ solves
all remaining equations for $n\intpositive$:

\begin{eqnarray*}
\hat{\pi}^{(0)}B^{(1,n)}+\sum_{i=1}^{n+1}\hat{\pi}^{(n)}A^{(i,n-i+1)} & = & \hat{\pi}^{(n)}\\
\Longleftrightarrow\xi(0)WU^{(1,n)}R+\sum_{i=1}^{n+1}\xi(n)\hat{\theta}U^{(i,n-i+1)}R & = & \xi(n)\hat{\theta}\\
\Longleftrightarrow\xi(0)\hat{\theta}(M^{(n)})R & = & \xi(n)\hat{\theta}
\end{eqnarray*}

Using the property \prettyref{eq:LS-EMC-M-n-explicite-with-xi} of
the matrix $M^{(n)}$ the last equation becomes 
\[
\xi(0)\hat{\theta}\left(\prod_{i=1}^{n}\frac{\lambda}{\mu(i)}\right)M^{(0)}R=\xi(n)\hat{\theta}~~~\Longleftrightarrow~~~\hat{\theta}M^{(0)}R=\hat{\theta}
\]
which is again \prettyref{eq:LCS-chain-theta-m0-r-solution}. Substituting
for $M^{(0)}$ the expression \eqref{eq:LS-EMC-steadyst4} finally
proves the rest of part \textbf{(a)}.

For part \textbf{(b)} of the theorem we realize from the dynamics
of the system that $\mathbb{N}_{0}\times L$ are the only states that
can be entered just after a departure instant. So, if $\hat{Z}$ is
started in some state $\mathbb{N}_{0}\times(K\setminus L)$ these
states states will never be visited again by $\hat{Z}$ and are therefore
inessential, which is in accordance with \eqref{eq:stochsolution2}. 
\end{proof}
Part \textbf{(b)} of \prettyref{thm:LCS-EMC-product-form-exp-service}
shows that in case of $K\setminus L\neq\emptyset$ $\hat{Z}$ is not
irreducible on $E$, hence not ergodic, although $Z$ is ergodic on
$E$. Furthermore, in general $\hat{Z}$ is even on the reduced state
space $\hat{E}:=\mathbb{N}_{0}\times L$ not ergodic. The reason is,
that $\hat{Z}$ may have periodic classes as the following example
shows.
\begin{example}
\label{ex:0S-1}\cite{schwarz;sauer;daduna;kulik;szekli:06}(See also
\prettyref{sect:LS-inventories}.) We consider an M/M/1/$\infty$-system
with attached inventory, i.e. a single server with infinite waiting
room under FCFS regime and an attached inventory under $(r,S)$-policy,
which is set in this example to $r=0$.

There is a Poisson-$\lambda$-arrival stream, $\lambda>0$. Customers
request for an amount of service time which is exponentially distributed
with mean $\mu>0$.

The server needs for each customer exactly one item from the inventory.
The on-hand inventory decreases by one at the moment of service completion.
If the inventory is decreased to the reorder point $r=0$ after the
service of a customer is completed, a replenishment order is instantaneously
triggered. The replenishment lead times are i.i.d. exponentially distributed
with parameter $\nu>0$. The replenishment fills the inventory up
to maximal inventory size $S>0$.

During the time the inventory is depleted and the server waits for
a replenishment order to arrive, no customers are admitted to join
the queue (\textquotedbl{}lost sales\textquotedbl{}).\\
 All service, interarrival and lead times are assumed to be independent.\\
 $X(t)$ is the number of customers present at the server at time
$t\geq0$, and $Y(t)$ is the on-hand inventory at time $t\geq0$.

The state space of $(X,Y)$ is $E=\{(n,k):n\in\mathbb{N}_{0},k\in K\},$
with $K=\{S,S-1,\dots,1,0\}$. where $S<\infty$ is the maximal size
of the inventory at hand.

The inventory management process under $(0,S)$-policy fits into the
definition of the environment process by setting 
\begin{eqnarray*}
K=\{S,S-1,\dots,1,0\},\qquad &  & K_{B}=\{0\},\\
\Rentry 00=1,~~~\Rentry k{k-1}=1,\quad1\leq k\leq S\,,\qquad &  & \v(k,m)=\begin{cases}
\nu, & \text{if}~~k=0,m=S\\
0, & \text{otherwise for}~~k\neq m\,.
\end{cases}
\end{eqnarray*}
The queueing-inventory process $Z=(X,Y)$ in continuous time is ergodic
iff ${\lambda}<{\mu}$. The steady state distribution $\pi=(\pi(n,k):(n,k)\in E)$
of $(X,Y)$ has product form 
\[
\pi(n,k)=\left(1-\frac{\lambda}{\mu}\right)\left(\frac{\lambda}{\nu}\right)^{n}\theta(k),
\]
where $\theta=(\theta(k):k\in K)$ with normalization constant $C$
is 
\begin{eqnarray}
\theta(k) & =\begin{cases}
C^{-1}(\frac{\lambda}{\nu}) & \qquad k=0,\\
C^{-1}(\frac{\lambda+\nu}{\lambda})^{k-1} & \qquad k=1,...,r,\\
C^{-1}(\frac{\lambda+\nu}{\lambda})^{r} & \qquad k=r+1,...,S.
\end{cases}\label{eq:mm1-r-S-theta}
\end{eqnarray}
For the Markov chain $\hat{Z}$ embedded in $Z$ at departure instants
we have $L=\{0,1,\dots,S-1\}$ and therefore the states $\mathbb{N}_{0}\times\{S\}$
are inessential.

From the dynamics of the system determined by the inventory management
follows directly that $\hat{Z}$ is periodic with period $S$ and
that $\mathbb{N}_{0}\times L$ is an irreducible closed set (the single
essential class), which is positive recurrent iff ${\lambda}<{\mu}$
holds. $\mathbb{N}_{0}\times L$ is partitioned into $S$ subclasses
$\mathbb{N}_{0}\times\{k\}$ which are periodically visited 
\[
\ldots\to\mathbb{N}_{0}\times\{S-1\}\to\mathbb{N}_{0}\times\{S-2\}\to\ldots\to\mathbb{N}_{0}\times\{0\}\to\mathbb{N}_{0}\times\{S-1\}\ldots
\]

\end{example}
The following corollary and examples demonstrate the versatility of
the class of models under consideration and consequences for the interplay
of $\theta$ for the continuous time setting and $\hat{\theta}$ for
the embedded Markov chain due to special settings of the environment.
\begin{cor}
\label{cor:RW1} Consider an ergodic $M/M/1/\infty$ loss system in
a random environment with any $\lambda$, $\mu(n)$, $\V$, and $R$
as defined in \prettyref{sect:LS-MM1Inf-model}.\\

\begin{itemize}
\item [(a)] If $R=I$, then the conditional distribution $\hat{\theta}$
of $\theta$ conditioned on $L$, 
\[
\hat{\theta}(k)=\begin{cases}
\frac{\theta(k)}{\theta(L)} & \quad\text{if}~~k\in L,\\
0 & \quad\text{if}~~k\in K\setminus L,
\end{cases}\quad\text{with}~~~\theta(L):=\sum_{m\in L}\theta(m)\,,
\]
solves \eqref{eq:LCS-mm1inf-theta-hat-equations}, which shows that
the embedded chain in this case reveals only the behaviour of the
environment on $L$, i.e. we loose information incorporated in the
continuous time description of the process.
\item [(b)] If $\V=0$ then the set $K_{B}$ of blocking states is empty,
and therefore $I_{W}=I$ holds. Furthermore, $R$ is irreducible and
positive recurrent.\\
The marginal steady state distribution $\theta$ of $Y$ in continuous
time is the stationary distribution of $R$, i.e., the solution of
$\theta R=\theta$.\\
And finally it holds $\theta=\hat{\theta}$, i.e., $\theta$ solves
on $K$ \eqref{eq:LCS-mm1inf-theta-hat-equations}, which shows that
the embedded chain exploits in this case the full information about
the possible environment of the system. 
\end{itemize}
\end{cor}
\begin{proof}
\textbf{(a)} is a direct consequence of \eqref{eq:LCS-theta-vs-theta-check}
because of $R=I$, and therefore $L=K_{W}$.

\textbf{(b)} If $\V=0$ and $K_{B}\neq\emptyset$ , then from ergodicity
the environment process $Y$ must enter $K_{B}$ in finite time, but
once the system entered a blocking state $k$ it can never leave this
because of $\v(k,m)=0$ for all $m\in K$. Furthermore, from ergodicity
of $Z$ with a similar argument, $R$ must be irreducible and positive
recurrent.

\eqref{eq:LS-constant-lambda-theta-matrix-equation} then reduces
to $\theta(\lambda(R-I))=0$ which is the steady state equation for
$R$.

We substitute $I_{W}=I$ and $\V=0$ into into the left side of equation
\prettyref{eq:LCS-mm1inf-theta-hat-equations} and obtain 
\[
\hat{\theta}\lambda(\lambda{I_{W}}-{\V})^{-1}{I_{W}}R=\hat{\theta}R\,,
\]
which reduces equation \prettyref{eq:LCS-mm1inf-theta-hat-equations}
to $\hat{\theta}R=\hat{\theta}$, which from irreducibility and positive
recurrence has a unique stochastic solution $\theta$. \end{proof}
\begin{example}
\label{ex:information-los-from-time-to-chain}We consider an $M/M/1/\infty$
system with arrival rate $\lambda$, service rate $\mu$, with $\lambda<\mu$,
in a random environment. The following examples will address the interrelations
between $\theta$ and $\hat{\theta}$. \\

\textbf{(a)} The first example provides a continuum of different environments,
which in continuous time have different marginal stationary distributions
$\theta$, but all of them have the same $\hat{\theta}$.

The environment is $K=\{1,2\}$ with $K_{B}=\{2\}$, and 
\begin{eqnarray*}
\V & = & \left(\begin{array}{c|cc}
 & 1 & 2\\
\hline 1 & 0 & 0\\
2 & \nu & -\nu
\end{array}\right)\qquad\quad R=\left(\begin{array}{c|cc}
 & 1 & 2\\
\hline 1 & 0 & 1\\
2 & 0 & 1
\end{array}\right)
\end{eqnarray*}

According to \eqref{eq:LS-constant-lambda-theta-matrix-equation}
the marginal steady state $\theta$ of the environment in continuous
time is the solution of the equation $\theta(\lambda(R-I)+\V)=0$,
which is 
\begin{eqnarray*}
\theta\left(\begin{array}{c|cc}
 & 1 & 2\\
\hline 1 & -\lambda & \lambda\\
2 & \nu & -\nu
\end{array}\right) & = & 0
\end{eqnarray*}

It follows that $\theta=(\frac{\nu}{\lambda+\nu},\frac{\lambda}{\lambda+\nu})$,
which depends on both, $\lambda$ and $\nu$.

On the other side we have $L=\{2\}$, and therefore the marginal steady
state distribution for the environment in the embedded chain is $\hat{\theta}=(0,1)$\emph{
for all} $\lambda$ and $\nu$.\\

\textbf{(b)} The second example provides two different environments
with the same environment space $K=\{1,2\}$. The point of interest
is that in continuous time both have the same marginal stationary
distribution of the environment ($\theta$), but the embedded chains
have different marginal stationary distributions of the environment
($\hat{\theta}$). For both systems holds $\nu=\lambda$.

\textbf{(b1)} The first system is a special case of $(a)$ with $K=\{1,2\}$,
blocking set $K_{1,B}=\{2\}$, and $\nu=\lambda$, i.e., 
\begin{eqnarray*}
\V_{1} & = & \left(\begin{array}{c|cc}
 & 1 & 2\\
\hline 1 & 0 & 0\\
2 & \lambda & -\lambda
\end{array}\right)\qquad\quad R_{1}=\left(\begin{array}{c|cc}
 & 1 & 2\\
\hline 1 & 0 & 1\\
2 & 0 & 1
\end{array}\right)
\end{eqnarray*}
Using the results from example \textbf{(a)}, we immediately get $\theta_{1}=\left(\frac{1}{2},\frac{1}{2}\right)$
and $\hat{\theta}_{1}=(0,1)$.

\textbf{(b2)} The second system has environment space $K=\{1,2\}$,
blocking set $K_{2,B}=\{1\}$, and $\nu=\lambda$, i.e., 
\begin{eqnarray*}
\V_{2} & = & \left(\begin{array}{c|cc}
 & 1 & 2\\
\hline 1 & \lambda & -\lambda\\
2 & 0 & 0
\end{array}\right)\,,\qquad\qquad R_{2}=\left(\begin{array}{c|cc}
 & 1 & 2\\
\hline 1 & 1 & 0\\
2 & 1 & 0
\end{array}\right)\,.
\end{eqnarray*}

The steady state equation $\theta_{2}(\lambda(R_{2}-I)+\V_{2})=0$
for $\theta$ of the second system is 
\begin{eqnarray*}
\theta_{2}\left(\begin{array}{c|cc}
 & 1 & 2\\
\hline 1 & -\lambda & \lambda\\
2 & \lambda & -\lambda
\end{array}\right) & = & 0
\end{eqnarray*}
which is solved by $\theta_{2}=\left(\frac{1}{2},\frac{1}{2}\right)$.
Using the same argumentation as in example (a) and the fact that $L_{2}=\{1\}$,
the marginal steady state distribution of the embedded Markov chain
is $\hat{\theta}_{2}=(1,0)$.
\end{example}
\begin{figure}[h]
\centering{} \subfloat[\label{fig:information-lost-dynamic-to-markov-chain}Environment states
in Example \prettyref{ex:information-los-from-time-to-chain} a.]{\begin{tikzpicture}
  \path
    (90:1) node (Y1)[shape=rectangle,draw] {$1$}
    (-90:1) node (Y2)[shape=rectangle,draw] {$2$}
    (180:2) node (DummyNode)[] {} 
    (0:2) node (DummyNode2)[] {} 
  ;
  \draw (Y1){} edge[arrows={-latex}, thick,out=0,in=0]
    node[auto]{}(Y2);
  \draw (Y2){} edge[arrows={-latex}, thick,color=vratecolor,out=180,in=180]
    node[auto]{$\nu$}(Y1);
\end{tikzpicture}

}$\qquad$\subfloat[\label{fig:information-lost-markov-chain-dynamic}Environment states
in Examples \prettyref{ex:information-los-from-time-to-chain} b1,b2.]{\begin{tikzpicture}
  \path
    (90:1) node (Y1)[shape=rectangle,draw] {$1$}
    (-90:1) node (Y2)[shape=rectangle,draw] {$2$}
  ;
  \draw (Y1){} edge[arrows={-latex}, thick,out=0,in=0]
     node[auto]{}(Y2);
  \draw (Y2){} edge[arrows={-latex}, thick,color=vratecolor,out=180,in=180]
     node[auto]{$\lambda$}(Y1);
\end{tikzpicture}

\begin{tikzpicture}
  \path 	    (90:1) node (Y1)[shape=rectangle,draw] {$2$}
    (-90:1) node (Y2)[shape=rectangle,draw] {$1$};
  \draw (Y1){} edge[arrows={-latex}, thick,out=0,in=0]
      node[auto]{}(Y2);
  \draw (Y2){} edge[arrows={-latex}, thick,color=vratecolor,out=180,in=180]
      node[auto]{$\lambda$}(Y1);
\end{tikzpicture}

}

\caption{{\Etid} from \prettyref{ex:information-los-from-time-to-chain}. }
\end{figure}

\begin{example}
We consider an $M/M/1/\infty$ system in a random environment with
arrival rate $\lambda$, service rate $\mu$, with $\lambda<\mu$,
in a random environment. The system is ergodic in continuous time
and the Markov chain observed at departure instants is ergodic as
well. There are no blocking states and therefore no loss of customers
occurs, i.e. the stream of admitted customers is Poissonian.

The environment is constructed in a way that the stationary distributions
of the of the environment of the continuous time process and of the
embedded Markov chains are distinct: $\hat{\theta}\neq\theta$.

We set $K=\{1,2\}$, $K_{W}=K$ and $K_{B}=\emptyset$, and with $\nu_{1},\nu_{2}>0$,
$\nu_{1}\neq\nu_{2}$ the matrices which govern the environment are
$\V=\left(\begin{array}{cc}
-\nu_{1} & \nu_{1}\\
\nu_{2} & -\nu_{2}
\end{array}\right)$ and $R=\left(\begin{array}{cc}
0 & 1\\
1 & 0
\end{array}\right)$.

Then it holds for the generator $\tilde{Q}(n)=:\tilde{Q},$ (which
is independent of $n$), see \eqref{eq:LS-q-tilde-matrix-representation},

\[
\tilde{Q}=(\lambda(R-I)+\V)=\left(\begin{array}{cc}
-\lambda-\nu_{1} & \lambda+\nu_{1}\\
\lambda+\nu_{2} & -\lambda-\nu_{2}
\end{array}\right)
\]

We calculate $\theta$, which solves $\theta\tilde{Q}=0$ (see \eqref{eq:LS-constant-lambda-theta-matrix-equation})
and obtain 
\[
\theta(1)=\frac{\lambda+\nu_{2}}{(\lambda+\nu_{1})+(\lambda+\nu_{2})}\,,\qquad\theta(2)=\frac{\lambda+\nu_{1}}{(\lambda+\nu_{1})+(\lambda+\nu_{2})}\quad\overset{\nu_{1}\neq\nu_{2}}{\Longrightarrow}\theta(1)\neq\theta(2)
\]

The system in this example is ergodic and $|K_{B}|=0<\infty$, therefore
according to \ref{prop:LS-EMC-matrix-inverse-proof-inf} $\left(\lambda I_{W}-V\right)$
it is invertible. The matrix $\left(\lambda I_{W}-V\right)^{-1}$
is also injective since $K$ is finite and with \prettyref{lem:LS-EMC-theta-theta-hat-releationship}
it follows existence of a unique solution with

\[
\hat{\theta}=(\theta I_{W}\mathbf{e})^{-1}\cdot\theta I_{W}R=\theta R=(\theta(2),\theta(1))\neq\theta
\]

because of $\theta(1)\neq\theta(2)$.
\end{example}

\section{$M/G/1/\infty$ queueing system in a random environment\label{sect:LS-EMC-MG1} }

Vineetha \cite{vineetha:08} extended the theory of integrated queueing-inventory
models with exponential service times to systems with i.i.d. service
times which follow a general distribution. The lead time is exponential
and during stock-out periods lost sales occur. Her approach was classical
in that she considered the continuous time Markovian state process
at departure instants of customers.

In this section we revisit some of Vineetha's \cite{vineetha:08}
models. We prove some of our results for queues with general environments
from the previous sections on $M/M/1/\infty$ systems in the $M/G/1/\infty$
framework, which includes an extension of Vineetha's queueing-inventory
systems to queues with state dependent service speeds and with non-exponential
service times.\\

Our framework is as in \prettyref{sect:LS-EMC-exponential}: Consider
the system at departure instants and utilize Markov chain analysis.

Our main aim is to identify conditions which enforce the systems to
stabilize in a way that the queue and the environment decouple in
the sense that the stationary queue length and environment of the
embedded Markov Chain behave independently, i.e., a product form equilibrium
exists.

It will come out that this is not always possible, but we are able
to provide sufficient conditions for the existence of product form
equilibria.

\subsection{$M/G/1/\infty$ queueing systems with state dependent service intensities}

\label{sect:MG1-1} We first describe a pure queueing model in continuous
time which is of $M/G/1/\infty$ type, under FCFS regime, where the
single server works with different queue length dependent speeds (''service
intensities''), and the customers' service requests are queue length
dependent as well.

A review of $M/G/1/\infty$ queueing systems with state dependent
arrival and service intensities, which are related to the model described
here, and their asymptotic and equilibrium behaviour is provided in
the survey of Dshalalow \cite{dshalalow:97}.

\label{pageMGs1}The arrival stream is Poisson-$\lambda$. When a
customer enters the single server seeing $n-1\geq0$ customers behind
him, i.e., the queue length is $n$, his amount of requested service
time is drawn according to a distribution function $B_{n}:[0,\infty)\to[0,1]$
with $B_{n}(0)=0$. The set of all interarrival times and service
time requests is an independent collection of variables.

The server works with queue length dependent service speeds $c(n)>0$,
i.e., when at time $t\geq0$ there are $X(t)=n>0$ customers in the
system ($n$ including the one in service), and if the residual service
request of the customer in service at time $t$ is $R(t)=r>0$, then
at time $t+\varepsilon$ his residual service request is 
\[
R(t+\varepsilon)=r-\varepsilon\cdot c(n),\quad\text{if this is}>0\,,
\]
otherwise at time $t+\varepsilon$ his service expired and he has
already departed from the system.

It is a standard observation that the process 
\[
(X,R)=((X(t),R(t)):t\geq0)
\]
is a homogeneous strong Markov process on state space $\mathbb{N}_{0}\times\mathbb{R}_{0}^{+}$
(with cadlag paths).

With $\tau_{0}=0$ we will denote as in the previous sections by $\tau=(\tau_{0},\tau_{1},\dots)$
the sequence of departure times of customers. It is a similar standard
observation that the process 
\[
\hat{X}=(\hat{X}(n):=(X(\tau_{n}),R(\tau_{n}-)):n\in\mathbb{N}_{0})
\]
is a homogeneous Markov chain on state space $\mathbb{N}_{0}\times\{0\}$.
Because of $R(\tau_{n}-)=0~\forall n$, we prefer to use for this
Markov chain on state space $\mathbb{N}_{0}$ the description 
\[
\hat{X}=(\hat{X}(n):=X(\tau_{n}):n\in\mathbb{N}_{0})\,.
\]

A little reflection shows that the one-step transition matrix of $\hat{X}$
is a matrix which has the usual skip-fee to the left property, i.e.,
with \index{p-tilde(i,n)@$\tilde{p}(i,n)$} \marginpar{$\tilde{p}^{(i,n)}$}
$\tilde{p}(i,n)$ defined as 
\begin{equation}
\tilde{p}(i,n):=P\left(X(\tau_{1})=i+n-1|X(0)=i\right)\,,\label{eq:p-n-i-definition}
\end{equation}

it is of the form (empty entries are zero) 
\begin{eqnarray}
\tilde{P} & := & \left(\begin{array}{ccccc}
\tilde{p}(1,0) & \tilde{p}(1,1) & \tilde{p}(1,2) & \tilde{p}(1,3) & \ldots\\
\tilde{p}(1,0) & \tilde{p}(1,1) & \tilde{p}(1,2) & \tilde{p}(1,3) & \ldots\\
 & \tilde{p}(2,0) & \tilde{p}(2,1) & \tilde{p}(2,2) & \ldots\\
 &  & \tilde{p}(3,0) & \tilde{p}(3,1) & \ldots\\
\\
\end{array}\right)\,,\label{MGs1-1}
\end{eqnarray}
which is an upper Hessenberg matrix. A similar one-step transition
matrix arises in \cite{dshalalow:97}{[}p.68{]} where the service
requests are state dependent, but no speeds are incorporated.

So for $\tilde{P}$ the row index $i$ indicates the number of customers
in system when a service commences (and the service request is drawn
according to $B_{n}$), and the (varying in row number) column index
$n$ indicates the number of customers who arrived during the ongoing
service.

Note, that although we have used an intuitive notation for the non
zero entries of $\tilde{P}$, the matrix is a fairly general upper
Hessenberg matrix: The only restrictions are strict positivity of
the $\tilde{p}(i,n)$ and row sum $1$.

We will not go into the details of computing $\tilde{P}$, but recall
the classical result for state independent service speeds ($=1$)
in the following subsection.

\subsubsection{$M/D/1/\infty$ queueing systems}

\label{sect:MD1-1} The classical situation is as follows (See \cite[177+]{kleinrock:75}). 
\begin{prop}
For the $M/G/1/\infty$ queuing system with service time distribution
$B:[0,\infty)\rightarrow[0,1]$, the transition probabilities $\tilde{p}(i,n)$
are independent of $i$ and the transition matrix $\tilde{P}$ has
the form 
\begin{equation}
\tilde{P}:=\left(\begin{array}{ccccc}
\tilde{p}(0) & \tilde{p}(1) & \tilde{p}(2) & \tilde{p}(3) & \ldots\\
\tilde{p}(0) & \tilde{p}(1) & \tilde{p}(2) & \tilde{p}(3) & \ldots\\
 & \tilde{p}(0) & \tilde{p}(1) & \tilde{p}(2) & \ldots\\
 &  & \tilde{p}(0) & \tilde{p}(1) & \ldots\\
\\
\end{array}\right)\label{MG1-P}
\end{equation}
with \index{p-tilde(n)@$\tilde{p}(n)$} \marginpar{$\tilde{p}(n)$}
\[
p(n):=\int_{0}^{\infty}e^{-\lambda t}\frac{(\lambda t)^{n}}{n!}dB(t)
\]
With $\mu^{-1}<\infty$ we denote the mean service time. Then under
$\lambda\mu^{-1}<1$ the continuous time process and the chain embedded
at departure instants are ergodic. 
\end{prop}
We now recall well known results for standard $M/D/1/\infty$ queues
where the service time is deterministic of length ${\frac{1}{\mu}}$,
i.e., the distribution function is $B=\delta_{\frac{1}{\mu}}$ (Dirac
measure). We assume $\rho:=\lambda/\mu<1$. Then the queue length
process $\hat{X}=(\hat{X}(n)):n\in\mathbb{N}_{0})$ at departure times
is an ergodic Markov chain with one-step transition matrix \eqref{MG1-P}
with 
\begin{equation}
\tilde{p}(n):=\int_{0}^{\infty}e^{-\lambda t}\frac{(\lambda t)^{n}}{n!}d\delta_{\frac{1}{\mu}}(t)=e^{-\frac{\lambda}{\mu}}\frac{(\frac{\lambda}{\mu})^{n}}{n!}\label{ptilde}
\end{equation}

We denote as usual the stationary distribution of of $\hat{X}$ by
$\hat{\xi}$, which is the unique stochastic solution of the equation
\begin{equation}
\hat{\xi}\tilde{P}=\hat{\xi}\,.\label{MD1steadystateeq}
\end{equation}
We will utilize later on some special values of $\hat{\xi}$ (see
\cite[p. 241]{gross;harris:74}) 
\begin{equation}
\hat{\xi}(0)=(1-\rho)\qquad\hat{\xi}(1)=\left(1-\rho\right)(e^{\rho}-1)\qquad\hat{\xi}(2)=\left(1-\rho\right)e^{\rho}(e^{\rho}-\rho-1)\label{MG1steadystate012}
\end{equation}

\subsection{$M/D/1/\infty$ system with inventory under lost sales}

\label{sect:MD1}

We analyze an $M/D/1/\infty$ queueing system with an attached inventory
under $(r,S)$-policy with lost sales, which is similar to Example
\ref{ex:0S-1}, but with deterministic service times. We summarize
the system's parameters:\\
 Poisson-$\lambda$ input, deterministic-${\frac{1}{\mu}}$ service
times, $\rho:=\lambda/\mu<1$. Lead times are exponential-$\nu$.
All service, interarrival, and lead times constitute an independent
family.\\
 Order policy is $(r,S)$ with $r=1$ and $S=2$. When the inventory
is depleted no service is provided and new arrivals are rejected (lost
sales).\\

The Markovian state process of the integrated queueing-inventory system
relies on the description of the $M/D/1/\infty$ queueing system,
given at the beginning of Section \ref{sect:MD1-1}.

For the system's description in continuous time we use the supplemented
queue length process $(X,R)$, where the $R$ process on $[0,\mu^{-1}]$
denotes the residual service time of the ongoing service as the supplementary
variable. We enlarge this process by adding the inventory size $Y$.

The joint queueing-inventory process with supplementary variable $R$
will be denoted by $Z=(X,R,Y)$, and lives on state space $\mathbb{N}_{0}\times[0,\mu^{-1}]\times\{2,1,0\}$.
We consider the system at departure instants, which leads to a one-step
transition matrix similar to \eqref{MG1-P}.

The dynamics of of the Markov chain $\hat{Z}$ embedded into $Z$
at departure instants will be described in a way that resembles the
$M/G/1$ type matrix analytical models.

From the structure of the embedding, we know, that $R(\tau_{n}-)=0$
and whenever $X(\tau_{n})=0$ we see $R(\tau_{n})=0$, resp. whenever
$X(\tau_{n})>0$ we see $R(\tau_{n})=1/\mu$. We therefore can, without
loss of information, delete the $R$-component of the process, to
obtain a Markov chain embedded at departure times 
\[
\hat{Z}=(\hat{X},\hat{Y})=((\hat{X}(n),\hat{Y}(n)):n\in\mathbb{N}_{0})\,,~~\text{with}~~\hat{Z}(n):=(\hat{X}(n),\hat{Y}(n)):=(X(\tau_{n}),Y(\tau_{n}))\,.
\]
The state space of $\hat{Z}$ is $E=\mathbb{N}_{0}\times\{2,1,0\}$
where $K=\{2,1,0\}$ is partitioned into $K=K_{W}+K_{B}$ with $K_{B}=\{0\}$
and carries the reversed natural order structure.\\

We proceed with nomenclature similar to Definition \ref{def:LS-EMC-P-A-n-B-n}
with the obvious modifications, which stem from the observation, that
for $i\intpositive$ the probabilities $P(Z(\tau_{1})=(i+n-1,m)|Z(0)=(i,k))$
do not depend on $i$, because service is provided with an intensity
which is independent of the queue length. We reuse several of the
previous notations but there will be no danger of misinterpretation
in this section. Recall, that $(\tau_{n}:n\in\mathbb{N}_{0})$ is
the sequence of departure instants
\begin{defn}
\label{def:P-A-n-B-n-D1}We introduce matrices $A^{(n)}\in\mathbb{R}^{K\times K}$
\index{$A^{(n)}$} \marginpar{$A^{(n)}$} by 
\begin{eqnarray}
A_{km}^{(n)} & := & P(Z(\tau_{1})=(i+n-1,m)|Z(0)=(i,k)),\qquad1\leq i\label{eq:A-n-i-definitionD1}
\end{eqnarray}
for $k,m\in K$, then one-step transition matrix $\mathbf{P}$ defined
according to \prettyref{def:LS-EMC-P-A-n-B-n} has the form 

\begin{equation}
\mathbf{P}=\left(\begin{array}{ccccc}
B^{(0)} & B^{(1)} & B^{(2)} & B^{(3)} & \ldots\\
A^{(0)} & A^{(1)} & A^{(2)} & A^{(3)} & \ldots\\
0 & A^{(0)} & A^{(1)} & A^{(2)} & \ldots\\
0 & 0 & A^{(0)} & A^{(1)} & \ldots\\
\vdots & \vdots & \vdots & \vdots
\end{array}\right)\,.\label{eq:LCS-EMC-P-matrix-level-independent}
\end{equation}

\end{defn}
We will clarify the structure of the solution of the equation $\hat{\pi}\mathbf{P}=\hat{\pi}$.
So, $\hat{\pi}$ is the steady state distribution of the embedded
Markov chain $\hat{Z}$. It will become clear that $\hat{Z}$ is in
general not irreducible on $E$.

It will be convenient to group $\hat{\pi}$ as 
\begin{equation}
\hat{\pi}=(\hat{\pi}^{(0)},\hat{\pi}^{(1)},\hat{\pi}^{(2)},\dots)\label{eq:LS-EMC-pistructure5-1}
\end{equation}
with 
\begin{equation}
\hat{\pi}^{(n)}=(\hat{\pi}(n,2),\hat{\pi}(n,1),\hat{\pi}(n,0)),\qquad n\in\mathbb{N}_{0}\,.\label{eq:LS-EMC-pistructure6-1}
\end{equation}

An immediate consequence is that the steady state equation can be
written as 
\begin{equation}
\hat{\pi}^{(0)}B^{(n)}+\sum_{i=1}^{n+1}\hat{\pi}^{(i)}A^{(n-i+1)}=\hat{\pi}^{(n)},\qquad n\in\mathbb{N}_{0}\,.\label{eq:LCS-EMC-steady-state-equation-grouped-1}
\end{equation}

We determine $A^{(n)},~B^{(n)}$ explicitly, distinguishing cases
by the initial states $\hat{Z}(0)$.\\

\noindent $\bullet$ $\hat{Z}(0)=(i,0),~i\intpositive$: The server
waits for replenishment of inventory. The queue length stays at $i$
until the ordered replenishment arrives. Then the inventory is restocked
to $S=2$ and the server resumes his work, stochastically identical
to a standard $M/D/1/\infty$-system until the service expires. When
the served customer leaves the system, the inventory contains one
item. 
\[
A_{(0,1)}^{(n)}=P(Z(\tau_{1})=(i+n-1,1)|Z(0)=(i,0))=\int_{0}^{\infty}e^{-\lambda t}\frac{(\lambda t)^{n}}{n!}d\delta_{\frac{1}{\mu}}(t)=e^{-\frac{\nu}{\mu}}\frac{(\frac{\lambda}{\mu})^{n}}{n!}=\tilde{p}(n)\,.
\]
Obviously, from the inventory management regime 
\[
A_{(0,0)}^{(n)}=A_{(0,2)}^{(n)}=0\,.
\]

\noindent $\bullet$ $\hat{Z}(0)=(i,1),~i\intpositive$: A lead time
is ongoing and the server is active serving the first customer in
the queue. In this case there are two possible target states for the
inventory when the customer currently in service leaves the system.

\noindent $\circ$ \emph{Target state} $0$: The ongoing service expires
before the lead time does. The resulting inventory state after service
is finished is $0$. 
\begin{eqnarray*}
A_{(1,0)}^{(n)} & = & P(Z(\tau_{1})=(i+n-1,0)|Z(0)=(i,1))\\
 & = & \int_{0}^{\infty}e^{-\lambda t}\frac{(\lambda t)^{n}}{n!}e^{-\nu t}d\delta_{\frac{1}{\mu}}(t)=e^{-\frac{\lambda+\nu}{\mu}}\frac{(\frac{\lambda}{\mu})^{n}}{n!}=e^{-\frac{\nu}{\mu}}\tilde{p}(n)
\end{eqnarray*}

\noindent $\circ$ \emph{Target state} $1$: The ongoing lead expires
before the service time does, and the inventory is filled up to $S=2$
during the ongoing service. The resulting inventory state when service
expired is $1$. (Additionally, an order is placed, but this does
not change the state.) 
\begin{eqnarray*}
A_{(1,1)}^{(n)} & = & P(Z(\tau_{1})=(i+n-1,1)|Z(0)=(i,1))\\
 & = & \int_{0}^{\infty}e^{-\lambda t}\frac{(\lambda t)^{n}}{n!}(1-e^{-\nu t})d\delta_{\frac{1}{\mu}}(t)=(1-e^{-\frac{\nu}{\mu}})e^{-\frac{\lambda}{\mu}}\frac{(\frac{\lambda}{\mu})^{n}}{n!}=(1-e^{-\frac{\nu}{\mu}})\tilde{p}(n)
\end{eqnarray*}
Obviously, from the inventory management regime 
\begin{eqnarray*}
A_{(1,2)}^{(n)} & = & 0
\end{eqnarray*}

\noindent $\bullet$ $\hat{Z}(0)=(i,2),~i\intpositive$: There are
$S=2$ items on stock, no order is placed and the service is provided
just as in a standard $M/D/1/\infty$ system. The resulting inventory
state when service expired is $1$. (Additionally, an order is placed,
but this does not change the state.) 
\[
A_{(2,1)}^{(n)}=P(Z(\tau_{1})=(i+n-1,1)|Z(0)=(i,2))=\int_{0}^{\infty}e^{-\lambda t}\frac{(\lambda t)^{n}}{n!}d\delta_{\frac{1}{\mu}}(t)=\tilde{p}(n)
\]
Obviously, from the inventory management regime 
\[
A_{(2,0)}^{(n)}=A_{(2,2)}^{(n)}=0
\]

\noindent $\bullet$ $\hat{Z}(0)=(0,0)$: The queue is empty, an order
is placed. No customers are admitted until replenishment of inventory.
When the ongoing lead time expires, inventory is restocked to $S=2$.
Thereafter new customers are admitted, and service starts immediately
after the first arrival. When this customer is served, the stock size
is $1$. 
\[
B_{(0,1)}^{(n)}=P(Z(\tau_{1})=(n,m)|Z(0)=(0,k))=\int_{0}^{\infty}e^{-\lambda t}\frac{(\lambda t)^{n}}{n!}e^{-\nu t}d\delta_{\frac{1}{\mu}}(t)=\tilde{p}(n)
\]
Obviously, from the inventory management regime 
\[
B_{(0,0)}^{(n)}=B_{(0,2)}^{(n)}=0
\]

\noindent $\bullet$ $\hat{Z}(0)=(0,1)$: The queue is empty, there
is 1 item on stock, and an order is placed. In this case there are
two possible target states for the inventory when the first customer
who arrives will be served and leaves the system.

\noindent $\circ$ \emph{Target state} $0$: The ongoing inter-arrival
time expires before the lead time does. The arriving customer's service
starts immediately and is finished before the replenishment arrives.
The resulting inventory state after service is finished is $0$.

\begin{eqnarray*}
B_{(1,0)}^{(n)} & = & P(Z(\tau_{1})=(n,0)|Z(0)=(0,1))\\
 & = & \frac{\lambda}{\nu+\lambda}\int_{0}^{\infty}e^{-\lambda t}\frac{(\lambda t)^{n}}{n!}e^{-\nu t}d\delta_{\frac{1}{\mu}}(t)=\frac{\lambda}{\nu+\lambda}e^{-\frac{\nu}{\mu}}\tilde{p}(n)
\end{eqnarray*}

\noindent $\circ$ \emph{Target state} $1$:\\
 ${\bf (1)}$ The ongoing lead expires before the inter-arrival time
does, and the inventory is filled up to $S=2$ during the ongoing
inter-arrival time. Then, until the first departure, the system acts
like a standard $M/D/1/\infty$ queue. When the first departure happens,
inventory size decreases to $1$. (Additionally, an order is placed,
but this does not change the state.)\\
 ${\bf (2)}$ The ongoing inter-arrival time expires before the lead
time does. The arriving customer's service starts immediately and
the replenishment arrives before the service is finished and by the
replenishment the stock size increases to $2$. The resulting inventory
state after service is finished is $1$. (Additionally, an order is
placed, but this does not change the state.) 
\begin{eqnarray*}
B_{(1,1)}^{(n)} & = & P(Z(\tau_{1})=(n,1)|Z(0)=(0,1))\\
 & = & \frac{\nu}{\nu+\lambda}\int_{0}^{\infty}e^{-\lambda t}\frac{(\lambda t)^{n}}{n!}d\delta_{\frac{1}{\mu}}(t)+\frac{\lambda}{\nu+\lambda}\int_{0}^{\infty}e^{-\lambda t}\frac{(\lambda t)^{n}}{n!}(1-e^{-\nu t})d\delta_{\frac{1}{\mu}}(t)\\
 & = & \frac{\nu}{\nu+\lambda}\tilde{p}(n)+\frac{\lambda}{\nu+\lambda}(1-e^{-\frac{\nu}{\mu}})\tilde{p}(n)=\left(1-\frac{\lambda}{\nu+\lambda}e^{-\frac{\nu}{\mu}}\right)\tilde{p}(n)
\end{eqnarray*}
Obviously, from the inventory management regime 
\[
B_{(1,2)}^{(n)}=0
\]

\noindent $\bullet$ $\hat{Z}(0)=(0,2)$: The queue is empty, there
are $2$ items on stock, and an inter-arrival time is ongoing. Until
the first departure the system develops like a standard $M/D/1/\infty$
queue. After that departure the inventory size is $1$. (Additionally,
an order is placed, but this does not change the state.)

\[
B_{(2,1)}^{(n)}=P(Z(\tau_{1})=(n,1)|Z(0)=(0,2))=\int_{0}^{\infty}e^{-\lambda t}\frac{(\lambda t)^{n}}{n!}d\delta_{\frac{1}{\mu}}(t)=\tilde{p}(n)
\]
Obviously, from the inventory management regime 
\[
B_{(2,0)}^{(n)}=B_{(2,2)}^{(n)}=0
\]

Summarizing the results we have (note, that we ordered the environment
in line: $2,1,0$)

\[
A^{(n)}=\tilde{p}(n)\left(\begin{array}{ccc}
0 & 1 & 0\\
0 & 1-e^{-\frac{\nu}{\mu}} & e^{-\frac{\nu}{\mu}}\\
0 & 1 & 0
\end{array}\right)\,,~~\text{and}~~B^{(n)}=\tilde{p}(n)\left(\begin{array}{ccc}
0 & 1 & 0\\
0 & 1-\frac{\lambda}{\nu+\lambda}e^{-\frac{\nu}{\mu}} & \frac{\lambda}{\nu+\lambda}e^{-\frac{\nu}{\mu}}\\
0 & 1 & 0
\end{array}\right)\,.
\]

We first prove that the steady state (marginal) queue length distribution
of $\hat{X}$ is the steady state distribution $\hat{\xi}$ of the
standard $M/D/1/\infty$ queue.

The row sums of $B^{(n)}$ and $A^{(n)}$ are $\tilde{p}(n)$, that
is 
\[
B^{(n)}\mathbf{e}=A^{(n)}\mathbf{e}=\tilde{p}(n)\mathbf{e},\quad n\in\mathbb{N}_{0}.
\]

Multiplying the steady state equations \eqref{eq:LCS-EMC-steady-state-equation-grouped-1}
for $\hat{Z}$ with $\mathbf{e}$ leads to 
\[
\hat{\pi}^{(0)}B^{(n)}\mathbf{e}+\sum_{i=1}^{n+1}\hat{\pi}^{(i)}A^{(i,n-i+1)}\mathbf{e}=\hat{\pi}^{(n)}\mathbf{e}\Longrightarrow\hat{\pi}^{(0)}\mathbf{e}\tilde{p}(0)+\sum_{i=1}^{n+1}\hat{\pi}^{(i)}\mathbf{e}\tilde{p}(n+1-i)=\hat{\pi}^{(n)}\mathbf{e}\,,
\]

which is \eqref{MD1steadystateeq}, which has a unique stochastic
solution. Now, $\hat{\pi}^{(i)}\mathbf{e}$ is the steady state (marginal)
queue length distribution of $\hat{X}$ and solves \eqref{MD1steadystateeq},
so we have shown $\hat{\pi}^{(i)}\mathbf{e}=\hat{\xi}(i)$ for all
$i\in\mathbb{N}_{0}$.\\

Now we are prepared to show that assuming a product form steady state
distribution $(\pi(n,k)=\xi(n)\hat{\theta}(k),(n,k)\in E)$ inserted
\eqref{eq:LS-chain-steady-state-equation} leads to a contradiction.

Inserting this product form $\pi(n,k)=\xi(n)\theta(k)$ into the equation
for the level $n=0$ and phase $k=0$, the steady state equation \eqref{eq:LS-EMC-steady-state-equation-grouped}
is transformed into

\begin{eqnarray}
\hat{\pi}(0,1)B_{(1,0)}^{(0)}+\hat{\pi}(1,1)A_{(1,0)}^{(0)} & = & \hat{\pi}(0,0)\nonumber \\
\Longleftrightarrow\frac{\lambda}{\nu+\lambda}e^{-\frac{\lambda+\nu}{\mu}}\hat{\xi}(0)\hat{\theta}(1)+e^{-\frac{\lambda+\nu}{\mu}}\hat{\xi}(1)\hat{\theta}(1) & = & \hat{\xi}(0)\hat{\theta}(0)\nonumber \\
\Longleftrightarrow e^{-\frac{\lambda+\nu}{\mu}}\left(\frac{\lambda}{\nu+\lambda}+\frac{\hat{\xi}(1)}{\hat{\xi}(0)}\right)\hat{\theta}(1) & = & \hat{\theta}(0)\nonumber \\
\Longleftrightarrow e^{-\frac{\lambda+\nu}{\mu}}\left(\frac{\lambda}{\nu+\lambda}+e^{\rho}-1\right)\hat{\theta}(1) & = & \hat{\theta}(0)\,,\label{eq:md1inf-failed-theta-0-1}
\end{eqnarray}

and the equation for level $n=1$ and phase $k=0$ under this product
form assumption is transformed into

\begin{eqnarray}
\hat{\pi}(0,1)B_{(1,0)}^{(1)}+\hat{\pi}(1,1)A_{(1,0)}^{(1)}+\hat{\pi}(2,1)A_{(1,0)}^{(0)} & = & \hat{\pi}(1,0)\\
\hat{\xi}(0)\hat{\theta}(1)B_{(1,0)}^{(1)}+\hat{\xi}(1)\hat{\theta}(1)A_{(1,0)}^{(1)}+\hat{\xi}(2)\hat{\theta}(1)A_{(1,0)}^{(0)} & = & \hat{\xi}(1)\hat{\theta}(0)\nonumber \\
\Longleftrightarrow\left(\frac{\lambda}{\nu+\lambda}e^{-\frac{\lambda+\nu}{\mu}}\frac{\lambda}{\mu}\hat{\xi}(0)+e^{-\frac{\lambda+\nu}{\mu}}\frac{\lambda}{\mu}\hat{\xi}(1)+e^{-\frac{\lambda+\nu}{\mu}}\hat{\xi}(2)\right)\hat{\theta}(1) & = & \hat{\xi}(1)\hat{\theta}(0)\nonumber \\
\Longleftrightarrow e^{-\frac{\lambda+\nu}{\mu}}\left(\frac{\lambda}{\nu+\lambda}\frac{\lambda}{\mu}\frac{\hat{\xi}(0)}{\hat{\xi}(1)}+\frac{\lambda}{\mu}+\frac{\hat{\xi}(2)}{\hat{\xi}(1)}\right)\hat{\theta}(1) & = & \hat{\theta}(0)\nonumber \\
\Longleftrightarrow e^{-\frac{\lambda+\nu}{\mu}}\left(\frac{\lambda}{\nu+\lambda}\frac{\lambda}{\mu}\frac{1}{e^{\rho}-1}+\frac{\lambda}{\mu}+\frac{e^{\rho}(e^{\rho}-\rho-1)}{(e^{\rho}-1)}\right)\hat{\theta}(1) & = & \hat{\theta}(0)\label{eq:md1inf-failed-theta-0-2}
\end{eqnarray}

One can see that the expressions \eqref{eq:md1inf-failed-theta-0-1}
and \eqref{eq:md1inf-failed-theta-0-2} are in general not equal.
For example, with the parameters $\lambda=1$, $\mu=2$ and $\nu=3$
the $\hat{\theta}(0)$ from \eqref{eq:md1inf-failed-theta-0-1} is
approximately $0.122\cdot\hat{\theta}(1)$ and the $\hat{\theta}(0)$
from the expression \eqref{eq:md1inf-failed-theta-0-2} is approximately
$0.145\cdot\hat{\theta}(1)$.

\subsection{$M/G/1/\infty$ queueing systems with state dependent service intensities
and product form steady state}

\label{sect:MG1-PF} In the previous section we have shown by a counterexample,
that in general the steady state distribution of an $M/G/1/\infty$
system with $(r,S)$ policy and lost sales does not have a product
form. Nevertheless, there are cases where loss systems in a random
environment have product form steady states. These systems belong
to a class of generalized $M/G/1/\infty$ loss systems, which will
be discussed in this subsection. We point out, that the results apply
to general birth-and-death processes in a random environment as well.
\begin{defn}
\label{def:LCS-MG1Inf-interference-less} We consider an $M/G/1/\infty$
queueing system in continuous time with state dependent service intensities
(speeds) as described at the beginning of Section \ref{sect:MG1-1}
(p.\pageref{pageMGs1}) and use the notation introduced there.

The supplemented queue length process $(X,R)$ (queue length, residual
service request) is not Markov because we additionally assume that
this queueing system is coupled with a \textbf{finite} environment
$K=K_{W}+K_{B}$ with $K_{W}\neq\emptyset$, driven again by a generator
$\V$ and a stochastic jump matrix $R$, as described at the beginning
of \prettyref{sect:LS-MM1Inf-model}. The state of the environment
process will be denoted by $Y$ again.

We prescribe that the interaction of $(X,R)$ with the environment
process $Y$ is via the following principles and restrictions:

${\bf (1)}$ If the environment process is in a \emph{non-blocking}
state $k$, i.e. $k\in K_{W}$, the queueing system develops in the
same way as an $M/G/1/\infty$ queuing system in isolation, governed
by $\tilde{P}$ from \eqref{MGs1-1}, without any change of the environment
until the next departure happens. Holding the environment invariant
during this period is guaranteed by $\v(k,m)=0$ for all $k\in K_{W},m\in K$.

${\bf (2)}$ If at time $t$ a customer departs from the system, the
environment state changes according to the stochastic jump matrix
$R$, independent of the history of the system given $Y(t)$.

${\bf (3)}$ Whenever the environment process is in a \emph{blocking}
state $k\in K_{B}$, it may change its state with rates governed by
the matrix $\V$, independent of the queue length and the residual
service request. 
\end{defn}
From these assumptions it is immediate, that $Z=(X,R,Y)$ is a continuous
time strong Markov process. We introduce sequences of stopping times
for the process $Z=(X,R,Y)$ as before: With $\tau_{0}=\sigma_{0}=\zeta_{0}=0$
we will denote by

$\tau=(\tau_{0},\tau_{1},\dots)$ the sequence of departure times
of customers,

$\sigma=(\sigma_{0},\sigma_{1},\dots)$ the sequence of arrival times
of customers admitted to the system,

$\zeta=(\zeta_{0},\zeta_{1},\dots)$ the sequence of jump times of
the continuous time process $Z$.\\

By standard arguments it is seen that the sequence 
\[
(X(\tau_{n}),R(\tau_{n}-),Y(\tau_{n})):n\in\mathbb{N}_{0})
\]
is a homogeneous Markov chain on state space $\mathbb{N}_{0}\times\{0\}\times K$.
Because for all $n\in\mathbb{N}_{0}$ holds $R(\tau_{n}-)=0$ we omit
the $R$-component and consider henceforth the homogeneous Markov
chain 
\[
\hat{Z}=((\hat{X}(n),\hat{Y}(n)):=(X(\tau_{n}),Y(\tau_{n})):n\in\mathbb{N}_{0})
\]
on state space $\mathbb{N}_{0}\times K$. The following formulae follow
directly from the description.

${\bf (1)}\Longrightarrow$ for $k\in K_{W},~m\in K$ 
\[
P\left(\left(X(\tau_{1}),Y(\tau_{1}{-})\right)=(n+i-1,m)|Z(0)=(i,k)\right)=\delta_{km}\tilde{p}(i,n)\,.
\]

${\bf (2)}\Longrightarrow$ for $k\in K_{W},~m\in K$ 
\begin{eqnarray*}
 &  & P\left(\left(X(\tau_{1}),Y(\tau_{1})\right)=(n+i-1,m)|Z(0)=(i,k)\right)\\
 &  & =\sum_{h\in K}P\left(\left(X(\tau_{1}),Y(\tau_{1}{-})\right)=(n+i-1,h)|Z(0)=(i,k)\right)\cdot\Rentry hm\,.
\end{eqnarray*}

${\bf (3)}\Longrightarrow$ for $k\in K_{B},~m\in K$ 
\[
P\left(\left(X(\zeta_{1}),Y(\zeta_{1})\right)=(j,m)|Z(0)=(i,k)\right)=\delta_{ij}\frac{\v(k,m)}{-\v(k,k)}\,.
\]

Note, that in the last expression $k\in K_{B}$ implies that the queueing
system is frozen, and therefore in the denominator of the right side
a summand $+1_{[k\in K_{W}]}(\lambda+\mu1_{[i>0]})$, which one might
have expected, does \textbf{not} appear.\\

Although we have imposed constraints on the behaviour of the environment
the model still is a very versatile one. The class of models from
Definition \ref{def:LCS-MG1Inf-interference-less} encompasses (e.g.)
many vacation models. These are models describing a server working
on primary and secondary customers, a situation which arises in many
computer, communication, and production systems and networks. If one
is mainly interested in the service process of primary customers,
then working on secondary customers means from the viewpoint of the
primary customers, that the server is not available or is interrupted.
For more details see e.g. the survey of Doshi \cite{doshi:90}. In
the classification given there \cite{doshi:90}{[}p. 221, 222{]} the
above model is a single server queue with \emph{general nonexhaustive
service, with nonpreemptive vacations,} and \emph{general vacation
rule}. Our system fits into these classification because whenever
a service expires the server decides (governed by $R$) whether to
perform another service or to wait for the next arriving customer
(a state in $K_{W}$ is selected), or to change its activity for secondary
customers (a state in $K_{B}$ is selected). The sojourn time in this
status is completely general distributed by construction, in fact
these sojourns in general are neither identical distributed nor independent.\\

The proposed product form property of $\hat{Z}$ originates from the
specific structure of the one-step transition matrix $\mathbf{P}$
of $\hat{Z}$. With some stochastic matrix $H\in\mathbb{R}^{K\times K}$,
which we present in all details below, 
\begin{equation}
\mathbf{P}=\left(\begin{array}{ccccc}
\tilde{p}(1,0)H & \tilde{p}(1,1)H & \tilde{p}(1,2)H & \tilde{p}(1,3)H & \ldots\\
\tilde{p}(1,0)H & \tilde{p}(1,1)H & \tilde{p}(1,2)H & \tilde{p}(1,3)H & \ldots\\
0 & \tilde{p}(2,2)H & \tilde{p}(2,1)H & \tilde{p}(2,2)H & \ldots\\
0 & 0 & \tilde{p}(3,0)H & \tilde{p}(3,1)H & \ldots\\
\vdots & \vdots & \vdots & \vdots
\end{array}\right)\,.\label{eq:LCS-interference-less-P}
\end{equation}
We will use an evaluation procedure similar to that used for the $M/M/1/\infty$
in a random environment, by decomposing the matrices $B^{(n)}=WU^{(n,0)}R$
and $A^{(i,n)}=U^{(i,n)}R$.

The next lemma guarantees that the expression $\frac{1}{-\v(k,k)+1_{[k\in K_{W}]}}$
in \prettyref{lem:LS-EMC-mg1-interference-less-W} and \prettyref{lem:LCS-mg1-interference-less-U}
is always well defined.
\begin{lem}
\label{lem:LS-EMC-mgs1-interference-less-Vkk-aso-well-defined}For
the system defined in Definition \prettyref{def:LCS-MG1Inf-interference-less}
it holds 
\begin{equation}
|\v(k,k)|>0,\qquad\forall k\in K_{B}\label{eq:LCS-mg1-interference-less-non-negative-Vkk}
\end{equation}
Therefore the expression $\frac{1}{-\v(k,k)+1_{[k\in K_{W}]}}$ is
well defined for any $k\in K$\end{lem}
\begin{proof}
The proof uses the same idea as that of Lemma \ref{lem:LS-EMC-diag-inv}.
Because $Z$ is ergodic there must be a positive rate $\v(k,m)>0$
to leave any blocking state $k\in K_{B}$. The generator property
$|\v(k,k)|=\sum_{h\neq k}\v{(k,h)}$ of the matrix $\V$ proves the
inequality \prettyref{eq:LCS-mg1-interference-less-non-negative-Vkk}. 
\end{proof}
We now define similar to \prettyref{eq:LS-EMC-W-definition-1} in
Lemma \ref{lem:LS-EMC-ABUW} a matrix $W$ and determine an explicit
representation. 
\begin{lem}
\label{lem:LS-EMC-mg1-interference-less-W} For the system from Definition
\ref{def:LCS-MG1Inf-interference-less} we set for $k,m\in K$ 
\begin{equation}
W_{km}:={P\left(Z(\sigma_{1})=(1,m)|Z(0)=(0,k)\right)}\,,\label{eq:LS-EMC-Wkm-1}
\end{equation}
and remark that $W_{km}=0$ for all $m\in K_{B}$. Then it holds 
\begin{equation}
W=(I_{W}-\V)^{-1}I_{W}\label{eq:LS-EMC-mg1-inf-interference-less-W}
\end{equation}
\end{lem}
\begin{proof}
Basically, the matrix $W$ has the same structure as $W$ in \prettyref{prop:LCS-W-matrix},
but we will derive a new representation, which is more suitable in
the subsequent proofs. Using the same transformation as in \prettyref{prop:LCS-W-matrix}
we get by a first entrance argument 
\begin{eqnarray*}
W_{km} & = & \sum_{h\in K\backslash\{k\}}\underbrace{P\left(Z(\sigma_{1})=(1,m)|Z(0)=(0,h)\right)}_{=W_{hm}}P\left(Z(\zeta_{1})=(0,h)|Z(0)=(0,k)\right)\\
 &  & +\delta_{km}{P\left(Z(\zeta_{1})=(1,m)|Z(0)=(0,k)\right)}
\end{eqnarray*}
The last term simplifies (with $\v(k,k)=0$ for $k\in K_{W}$) to
\begin{eqnarray*}
 &  & \delta_{km}P\left(Z(\zeta_{1})=(1,m)|Z(0)=(0,k)\right)=\delta_{km}\frac{\lambda1_{[k\in K_{W}]}}{\lambda1_{[k\in K_{W}]}-\v(k,k)}=\delta_{km}1_{[k\in K_{W}]}\,.
\end{eqnarray*}

If $k\neq h$ the expression $P\left(Z(\zeta_{1})=(0,h)|Z(0)=(0,k)\right)$
is $\frac{\v(k,h)}{-\v(k,k)}$ for $k\in K_{B}$ and $0$ for $k\in K_{W}$.
In both cases we will use the expression $\frac{\v(k,h)}{-\v(k,k)+\lambda1_{[k\in K_{W}]}}$,
which is defined for any $k\in K$ (see \prettyref{lem:LS-EMC-mgs1-interference-less-Vkk-aso-well-defined});
it follows:

\begin{eqnarray*}
W_{km} & = & \sum_{h\in K\backslash\{k\}}W_{hm}\frac{\v(k,h)}{-\v(k,k)+\lambda1_{[k\in K_{W}]}}+\delta_{km}{\frac{\lambda}{-\v(k,k)+\lambda1_{[k\in K_{W}]}}}1_{[k\in K_{W}]}\\
 & = & \sum_{h\in K\backslash\{k\}}W_{hm}\frac{\v(k,h)}{-\v(k,k)+1_{[k\in K_{W}]}}+\delta_{km}\underbrace{\frac{1}{-\v(k,k)+1_{[k\in K_{W}]}}}_{1\text{ for }k\in K_{W}}1_{[k\in K_{W}]}
\end{eqnarray*}
This equation reads in matrix form 
\[
W=(-diag(\V)+I_{W})^{-1}\left((\V-diag(\V))W+I_{W}\right)
\]
and can finally be transformed into the lemma's statement \prettyref{eq:LS-EMC-mg1-inf-interference-less-W}:
\begin{eqnarray*}
(-diag(\V)+I_{W})W & = & (\V-diag(\V))W+I_{W}\\
\Longleftrightarrow(I_{W}-\V)W & = & I_{W}~~\Longrightarrow W=(I_{W}-\V)^{-1}I_{W}
\end{eqnarray*}

\end{proof}
We now determine in a similar way the matrices $U^{(i,n)}$, see the
definition \eqref{eq:LS-EMC-U-definition-1} in \prettyref{lem:LS-EMC-ABUW}
for the exponential case, and determine an explicit representation. 
\begin{lem}
\label{lem:LCS-mg1-interference-less-U}In the system from \prettyref{def:LCS-MG1Inf-interference-less}
we define for $n\geq0$ and $i\intpositive$ 
\[
U_{km}^{(i,n)}:=P\left(\left(X(\tau_{1}),Y(\tau_{1}^{-})\right)=(n+i-1,m)|Z(0)=(i,k)\right)\,.
\]
Then for the transition probability matrix $U$ it holds 
\[
U^{(i,n)}=\tilde{p}(i,n)(I_{W}-\V)^{-1}I_{W}
\]
\end{lem}
\begin{proof}
For $U^{(i,n)}$ with any $n\geq0$ and $i\intpositive$ it holds:

\begin{eqnarray*}
U_{km}^{(i,n)} & = & P\left(\left(X(\tau_{1}),Y(\tau_{1}^{-})\right)=(n+i-1,m)|Z(0)=(i,k)\right)\\
 & = & \sum_{h\in K\backslash\{k\}}P\left(\left(X(\tau_{1}),Y(\tau_{1}^{-})\right)=(n+i-1,m)\cap Z(\zeta_{1})=(i,h)|Z(0)=(i,k)\right)\\
 &  & +\delta_{km}P\left(\left(X(\tau_{1}),Y(\tau_{1}^{-})\right)=(n+i-1,k)|Z(0)=(i,k)\right)\\
 & = & \sum_{h\in K\backslash\{k\}}P\left(\left(X(\tau_{1}),Y(\tau_{1}^{-})\right)=(n+i-1,m)|Z(\zeta_{1})=(i,h),Z(0)=(i,k)\right)\\
 &  & \cdot P\left(Z(\zeta_{1})=(i,h)|Z(0)=(i,k)\right)\\
 &  & +\delta_{km}P\left(\left(X(\tau_{1}),Y(\tau_{1}^{-})\right)=(n+i-1,k)|Z(0)=(i,k)\right)\\
 & = & \sum_{h\in K\backslash\{k\}}\underbrace{P\left(\left(X(\tau_{1}),Y(\tau_{1}^{-})\right)=(n+i-1,m)|Z(0)=(i,h)\right)}_{=U_{hm}^{(i,n)}}\\
 &  & \cdot P\left(Z(\zeta_{1})=(i,h)|Z(0)=(i,k)\right)\\
 &  & +\delta_{km}P\left(\left(X(\tau_{1}),Y(\tau_{1}^{-})\right)=(n+i-1,k)|Z(0)=(i,k)\right)
\end{eqnarray*}

We analyze the expression $P\left(Z(\zeta_{1})=(i,h)|Z(0)=(i,k)\right)$:\\
 it is $\frac{\v(k,h)}{-\v(k,k)}$ for $k\in K_{B}$ and $0$ for
$k\in K_{W}$. As in the proof of the \prettyref{lem:LS-EMC-mg1-interference-less-W},
we use the combined expression $\frac{\v(k,h)}{-\v(k,k)+\lambda1_{[k\in K_{W}]}}=\frac{\v(k,h)}{-\v(k,k)+1_{[k\in K_{W}]}}$
which is valid for any $k\in K$. It follows

\begin{eqnarray*}
U_{km}^{(i,n)} & = & \sum_{h\in K\backslash\{k\}}U_{hm}^{(i,n)}\frac{\v(k,h)}{-\v(k,k)+\lambda1_{[k\in K_{W}]}}+{\frac{\lambda}{-\v(k,k)+\lambda1_{[k\in K_{W}]}}}\delta_{km}\tilde{p}(i,n)1_{[k\in K_{W}]}\\
 & = & \sum_{h\in K\backslash\{k\}}U_{hm}^{(i,n)}\frac{\v(k,h)}{-\v(k,k)+1_{[k\in K_{W}]}}+\underbrace{\frac{1}{-\v(k,k)+1_{[k\in K_{W}]}}}_{=1~~\text{for }k\in K_{W}}\delta_{km}\tilde{p}(i,n)1_{[k\in K_{W}]}
\end{eqnarray*}

The equation above, written in matrix form, reads 
\begin{eqnarray*}
U^{(i,n)} & = & (-diag(\V)+I_{W})^{-1}((\V-diag(\V))U^{(i,n)}+\tilde{p}(i,n)I_{W})\\
\Longleftrightarrow(-diag(\V)+I_{W})U^{(i,n)} & = & ((\V-diag(\V))U^{(i,n)}+\tilde{p}(i,n)I_{W})\\
\Longleftrightarrow(I_{W}-\V)U^{(i,n)} & = & \tilde{p}(i,n)I_{W}\Longleftrightarrow U^{(i,n)}=\tilde{p}(i,n)(I_{W}-\V)^{-1}I_{W}.
\end{eqnarray*}

\end{proof}
We are now prepared to evaluate the transition matrix of the $M/G/1/\infty$
system in a random environment from Definition \prettyref{def:LCS-MG1Inf-interference-less}.
It turns out that it has precisely the form \prettyref{eq:LCS-interference-less-P}.
\begin{lem}
\label{lem:LCS-mg1-interference-less-B-A} Consider the continuous
time Markov state process of the system described in Definition \ref{def:LCS-MG1Inf-interference-less},
and the Markov chain $\hat{Z}$, embedded at departure instants of
customers.

The one-step transition matrix $\mathbf{P}$ of $\hat{Z}$ 
\[
\mathbf{P}=\left(\begin{array}{ccccc}
B^{(0)} & B^{(1)} & B^{(2)} & B^{(3)} & \ldots\\
A^{(0,1)} & A^{(1,1)} & A^{(2,1)} & A^{(3,1)} & \ldots\\
0 & A^{(0,2)} & A^{(1,2)} & A^{(2,2)} & \ldots\\
0 & 0 & A^{(0,3)} & A^{(1,3)} & \ldots\\
\vdots & \vdots & \vdots & \vdots
\end{array}\right),
\]
is build up by the following block matrices: 
\[
B^{(n)}=A^{(1,n)}=\tilde{p}(1,n)H~~\text{and}~~A^{(i,n)}=\tilde{p}(i,n)H
\]
with\index{$H$}\marginpar{$H$} 
\begin{equation}
H:=(I_{W}-\V)^{-1}I_{W}R\label{eq:LCS-mg1-interference-less-M}
\end{equation}
\end{lem}
\begin{proof}
We analyze the block structure of the matrix $(I_{W}-\V)^{-1}I_{W}$:\\
 
\begin{eqnarray*}
(I_{W}-\V) & = & \left(\begin{array}{c|cc}
 & K_{W} & K_{B}\\
\hline K_{W} & I_{W} & 0\\
K_{B} & -\V|_{K_{B}\times K_{W}} & -\V|_{K_{B}\times K_{B}}
\end{array}\right)\\
\Longrightarrow(I_{W}-\V)^{-1} & = & \left(\begin{array}{c|cc}
 & K_{W} & K_{B}\\
\hline K_{W} & I_{W} & 0\\
K_{B} & (I_{W}-\V)^{-1}|_{K_{B}\times K_{W}} & (I_{W}-\V)^{-1}|_{K_{B}\times K_{B}}
\end{array}\right)\\
\Longrightarrow(I_{W}-\V)^{-1}I_{W} & = & \left(\begin{array}{c|cc}
 & K_{W} & K_{B}\\
\hline K_{W} & I_{W} & 0\\
K_{B} & (I_{W}-\V)^{-1}|_{K_{B}\times K_{W}} & 0
\end{array}\right)
\end{eqnarray*}

This leads to the useful property 
\begin{equation}
(I_{W}-\V)^{-1}I_{W}(I_{W}-\V)^{-1}I_{W}=(I_{W}-\V)^{-1}I_{W}\,.\label{eq:LCS-Iw-V-invariant-1}
\end{equation}
In a completely similar way as in Lemma \ref{lem:LS-EMC-ABUW} we
can show the following representations 
\[
A^{(i,n)}=U^{(i,n)}R\qquad\text{and}\qquad B^{(n)}=WU^{(1,n)}R.
\]

Inserting the results from Lemma \ref{lem:LCS-mg1-interference-less-U}
and Lemma \ref{lem:LS-EMC-mg1-interference-less-W} we obtain directly

\[
A^{(i,n)}=U^{(i,n)}R=\tilde{p}(i,n)(I_{W}-\V)^{-1}I_{W}R\,,
\]
\begin{eqnarray*}
B^{(n)} & = & WU^{(1,n)}R=\tilde{p}(1,n)\left((I_{W}-\V)^{-1}I_{W}\right)^{2}R\\
 & \overset{\eqref{eq:LCS-Iw-V-invariant-1}}{=} & \tilde{p}(1,n)(I_{W}-\V)^{-1}I_{W}R=A^{(1,n)}\,,
\end{eqnarray*}
which is the proposed result. 
\end{proof}
The next step is similar to that in case of the purely exponential
system. 
\begin{lem}
\label{lem:LCS-mg1-interference-stochastic-M}The matrix $H=(I_{W}-\V)^{-1}I_{W}R$
defined in \prettyref{eq:LCS-mg1-interference-less-M} is a stochastic
matrix and there exists a stochastic solution $\hat{\theta}$ of the
steady state equation 
\begin{equation}
\hat{\theta}H=\hat{\theta}\label{eq:LCS-mg1-interference-less-theta-check-existence}
\end{equation}
\end{lem}
\begin{proof}
The generator property of $\V$ leads to 
\begin{equation}
(I_{W}-\V)\mathbf{e}=I_{W}\mathbf{e}+\underbrace{\V\mathbf{e}}_{=0}=I_{W}\mathbf{e}\label{eq:LCS-mg1-interference-less-M-stochastic-Iw-e}
\end{equation}
and the stochasticity of $R$ yields $R\mathbf{e}=\mathbf{e}$. Inserting
this into the definition of $H$ leads to 
\[
H\mathbf{e}=(I_{W}-\V)^{-1}I_{W}R\mathbf{e}=(I_{W}-\V)^{-1}I_{W}\mathbf{e}\overset{\eqref{eq:LCS-mg1-interference-less-M-stochastic-Iw-e}}{=}(I_{W}-\V)^{-1}(I_{W}-\V)\mathbf{e}=\mathbf{e}
\]
Since the matrix $\tilde{p}(i,n)(I_{W}-\V)^{-1}I_{W}$ describes transition
probabilities, all its entries are non-negative, therefore the matrix
$H$ is stochastic.

Finally, finiteness of $K$ guarantees the existence of a stochastic
solution of \prettyref{eq:LCS-mg1-interference-less-theta-check-existence}.\end{proof}
\begin{thm}
\label{thm:MG1-PF-result} Consider the $M/G/1/\infty$ in a random
environment from Definition \ref{def:LCS-MG1Inf-interference-less}
with state dependent service speeds and state dependent selection
of requested service times. The describing Markov process $(X,R,Y)$
in continuous time is assumed to be ergodic. For the Markov chain
$(\hat{X},\hat{Y})$ embedded at departure points of customers denote
the (existing) stationary distribution by $\hat{\pi}$.

Then $\hat{\pi}$ has product form according to 
\[
\hat{\pi}(n,k)=\hat{\xi}(n)\hat{\theta}(k)\,,\quad(n,k)\in\mathbb{N}_{0}\times K.
\]

Here $\hat{\xi}$ is the steady state distribution of the Markov chain
with one-step transition matrix \eqref{MGs1-1} derived for the queue
length process at departure points in a system with the same parameter
as under consideration but without environment, that is a solution
of 
\begin{equation}
\hat{\xi}\tilde{P}=\hat{\xi}\,,\label{eq:LS-xi-check-steady-state-equation}
\end{equation}
and $\hat{\theta}$ is a stochastic solution of the equation 
\begin{equation}
\hat{\theta}H=\hat{\theta}(I_{W}-\V)^{-1}I_{W}R=\hat{\theta}\label{eq:LS-EMC-mgs1-theta-check-equation}
\end{equation}
\end{thm}
\begin{proof}
According to \prettyref{lem:LCS-mg1-interference-less-B-A} the transition
matrix $\mathbf{P}$ of the system has block form \prettyref{eq:LCS-interference-less-P},
which is the tensor product of $\tilde{P}$ from \eqref{MGs1-1} and
$H$: 
\[
\mathbf{P}=\tilde{P}\otimes H\,.
\]

Let $\hat{\xi}$ be the steady state solution of \eqref{eq:LS-xi-check-steady-state-equation},
i.e., of the pure queuing system without environment.

Let $\hat{\theta}$ be the stochastic solution of the equation $\hat{\theta}H=\hat{\theta}$,
which exists according to \prettyref{lem:LCS-mg1-interference-stochastic-M}.
Then from tensor calculus of matrices {\cite[(2.2.1.9) on p. 53]{neuts:81}}
$\hat{\pi}(n,k)=\hat{\xi}(n)\hat{\theta}(k)$ solves the steady state
equation 
\[
\hat{\pi}\mathbf{P}=(\hat{\xi}\otimes\hat{\theta})\tilde{P}\otimes H=(\hat{\xi}\tilde{P})\otimes(\hat{\theta}H)=\hat{\xi}\otimes\hat{\theta}=\hat{\pi}.
\]

\end{proof}

\section{Applications\label{sect:LS-EMC-applications}}

We apply the results from Sections \ref{sect:LS-EMC}, \ref{sect:LS-EMC-steadystate},
and \ref{sect:LS-EMC-MG1} to queueing-inventory systems which are
dealt with in literature recently, see the review in \cite{krishnamoorthy;lakshmy;manikandan:11}.

Note that to obtain the steady state distribution $\hat{\pi}$ from
$\pi$ using \eqref{eq:LCS-theta-vs-theta-check} could be easier
than using \eqref{eq:LCS-mm1inf-theta-hat-equations}, since starting
from \eqref{eq:LCS-theta-vs-theta-check} does not require to calculate
the inverse of $(\lambda I_{W}-\V)$. Nevertheless, we will use \eqref{eq:LCS-mm1inf-theta-hat-equations}
and calculate $\lambda(\lambda I_{W}-\V)^{-1}I_{W}R$ explicitly to
gain more insight into the mathematical structure of the problem.

In any of the following applications the queueing system represents
a production facility where raw material arrives and to assemble a
final product from a piece of raw material exactly one item from the
stock is needed. This item will formally be taken from the stock when
the production of the final product is finished.

In \prettyref{prop:LS-EMC-mm1inf-r-S} and \prettyref{prop:LS-EMC-mm1inf-r-Q}
we slightly extend the lost sales problems from \prettyref{ex:LS-original-r-S}
and \prettyref{ex:LS-original-r-Q} by incorporating stock size dependent
delivery rate $\nu_{k}$. The main reason of this modification is
besides of having more versatile models to demonstrate how each entry
of the matrices $\V$ and $R$ influences the transition probabilities
$\lambda(\lambda I_{W}-\V)^{-1}I_{W}R$ and the steady state distribution
$\hat{\theta}$.

\subsection{Systems with exponential service requests\label{sect:appl-exp}}

\begin{center}
\begin{figure}
\centering{}\subfloat[\label{fig:LS-EMC-r-S-example-diagram}$M/M/1/\infty$ system with
${(r=2,S=5)}$-policy and lost sales.]{\begin{tikzpicture}
  \path
	(60:2) node (Y5)[shape=rectangle,draw] {$5$}
	(0:2) node (Y4)[shape=rectangle,draw] {$4$}
	(-60:2) node (Y3)[shape=rectangle,draw] {$3$}
	(-120:2) node (Y2)[shape=rectangle,draw] {$2$}
	(-180:2) node (Y1)[shape=rectangle,draw] {$1$}
	(-240:2) node (Y0)[shape=rectangle,draw,fill=blocked.bg] {$0$};
  \draw[arrows={-latex}, thick]
    (Y5) -- (Y4);
  \draw[arrows={-latex}, thick]
    (Y4) -- (Y3);
  \draw[arrows={-latex}, thick]
    (Y3) -- (Y2);
  \draw[arrows={-latex}, thick]
    (Y2) -- (Y1);
  \draw[arrows={-latex}, thick]
    (Y1) -- (Y0);
  \draw[arrows={-latex}, thick, color=vratecolor]
    (Y0) -- node[auto]{$\nu_0$}(Y5);
  \draw[arrows={-latex}, thick, color=vratecolor]
    (Y1) -- node[auto]{$\nu_1$}(Y5);
  \draw[arrows={-latex}, thick, color=vratecolor]
    (Y2) -- node[auto]{$\nu_2$}(Y5);
\end{tikzpicture}}~~~~~~~\subfloat[\label{fig:LS-EMC-r-Q-example-diagram}$M/M/1/\infty$ system with
${(r=2,Q=3)}$-policy and lost sales.]{\begin{tikzpicture}
  \path 	(60:2) node (Y5)[shape=rectangle,draw] {$5$}
	(0:2) node (Y4)[shape=rectangle,draw] {$4$}
	(-60:2) node (Y3)[shape=rectangle,draw] {$3$}
	(-120:2) node (Y2)[shape=rectangle,draw] {$2$}
	(-180:2) node (Y1)[shape=rectangle,draw] {$1$}
	(-240:2) node (Y0)[shape=rectangle,draw,fill=blocked.bg] {$0$};
  \draw[arrows={-latex}, thick]
    (Y5) -- (Y4);
  \draw[arrows={-latex}, thick]
    (Y4) -- (Y3);
  \draw[arrows={-latex}, thick]
    (Y3) -- (Y2);
  \draw[arrows={-latex}, thick]
    (Y2) -- (Y1);
  \draw[arrows={-latex}, thick]
    (Y1) -- (Y0);
  \draw[arrows={-latex},near start, thick, color=vratecolor]
    (Y0.east) -- node[auto]{$\nu_0$}(Y3);
  \draw[arrows={-latex},near start, thick, color=vratecolor]
    (Y1) -- node[auto]{$\nu_1$}(Y4);
  \draw[arrows={-latex},near start, thick, color=vratecolor]
    (Y2) -- node[auto]{$\nu_2$}(Y5);
\end{tikzpicture}}\caption{{\Etid} for the environment of the lost sales of Propositions \ref{prop:LS-EMC-mm1inf-r-S}
and \ref{prop:LS-EMC-mm1inf-r-Q}. The environment counts the stock
size of the inventory.}
\end{figure}

\par\end{center}
\begin{prop}
\label{prop:LS-EMC-mm1inf-r-S} We consider an exponential single
server queue with state dependent service rates, environment dependent
replenishment rates, and an attached inventory under $(r,S)$ policy
(with $0\leq r<S\in\mathbb{N}$), and lost sales when the inventory
is depleted.

Using the definitions of Section \ref{sect:LS-EMC} we set the environment
state space $K:=\{0,...,S\}$ with $K_{B}=\{0\}$, $X(t)$ the queue
length at time $t$, and $Y(t)=k$ indicates that at time $t$ the
stock contains exactly $k$ items. The strictly positive transitions
intensities are 
\begin{eqnarray*}
q((n,k)\qsep(n+1,k)) & = & \lambda\qquad k>0\\
q((n,k)\qsep(n,S)) & = & \nu_{k}\qquad0\leq k\leq r\\
q((n,k)\qsep(n-1,k-1)) & = & \mu(n)\qquad n>0,1\leq k\leq S\\
q((n,k)\qsep(l,m)) & = & 0,\qquad otherwise
\end{eqnarray*}

The steady state $\hat{\pi}$ of the Markov chain $(\hat{X},\hat{Y})$
embedded at departure times has product form 

\begin{equation}
\hat{\pi}(n,k)=\xi(n)\hat{\theta}(k)\,,\quad(n,k)\in\mathbb{N}_{0}\times K\,,\label{eq:LS-EMC-mm1-r-S-product-form-assumption}
\end{equation}

with

\[
\xi(n):=C^{-1}\left(\prod_{i=1}^{n}\frac{\lambda}{\mu(i)}\right),~~n\in\mathbb{N}_{0},\qquad\text{with normalization constant}\, C^{-1}
\]

and 
\begin{equation}
\hat{\theta}(k)=\begin{cases}
C^{-1}\cdot\prod_{i=1}^{k}\left(\frac{\lambda+\nu_{i}}{\lambda}\right)^{i} & ,\qquad0\leq k\leq r\\
C^{-1}\cdot\prod_{i=1}^{r}\left(\frac{\lambda+\nu_{i}}{\lambda}\right)^{i} & ,\qquad r+1\leq k\leq S-1\\
0 & \qquad k=S
\end{cases}\label{eq:LS-EMC-mm1-chain-r-S-theta-hat}
\end{equation}

with

\[
C=\sum_{k=0}^{r-1}\prod_{i=1}^{k}\left(\frac{\lambda+\nu_{i}}{\lambda}\right)^{i}+(S-r)\prod_{i=1}^{r}\left(\frac{\lambda+\nu_{i}}{\lambda}\right)^{i}
\]

\end{prop}
Note that even for the constant values $\nu_{k}=\nu$ the marginal
distribution $\hat{\theta}$ in \eqref{eq:LS-EMC-mm1-chain-r-S-theta-hat}
differs from the marginal stationary distribution $P(Y(t)=k)$ of
the continuous time process in \prettyref{ex:LS-original-r-S}.
\begin{proof}
According to \prettyref{thm:LCS-EMC-product-form-exp-service} the
marginal distribution $\hat{\theta}$ is the stochastic solution of
the equation 
\begin{equation}
\hat{\theta}\lambda(\lambda I_{W}-\V)^{-1}I_{W}R=\hat{\theta}\label{eq:LS-EMC-theta-hat-equation-in-r-S}
\end{equation}

We calculate the matrix $\lambda(\lambda I_{W}-V)^{-1}I_{W}R$ explicitly.
The cautious reader will realize, that we write down the matrices
here with indices in an order which inverts the usual one which we
prescribed for $\preccurlyeq$ on $K$ in the first part of the paper.
This will make reading easier in this special case.

\begin{eqnarray*}
 & (\lambda I_{W}-\V)=\\
 & \tiny\left(\begin{array}{c|cccccccccc}
 & 0 & 1 & 2 & \ldots & r-1 & r & r+1 & \ldots & S-1 & S\\
\hline 0 & \nu_{0} & 0 & 0 &  & 0 & 0 & 0 &  & 0 & -\nu_{0}\\
1 & 0 & (\nu_{1}+\lambda) & 0 &  & 0 & 0 & 0 &  & 0 & -\nu_{1}\\
2 & 0 & 0 & (\nu_{2}+\lambda) &  & 0 & 0 & 0 &  & 0 & -\nu_{2}\\
\vdots & \vdots & \vdots & \ddots &  &  &  &  &  &  & \vdots\\
r & 0 & 0 & 0 & \ldots & 0 & (\nu_{r}+\lambda) & 0 & \ldots & 0 & -\nu_{r}\\
r+1 & 0 & 0 & 0 &  & 0 & 0 & \lambda & \ldots & 0 & 0\\
\vdots & \vdots & \vdots & \vdots & \vdots & \vdots & \vdots & \ddots & \ddots & \vdots & \vdots\\
S-1 & 0 & 0 & 0 & \ldots & 0 & 0 &  &  & \lambda & 0\\
S & 0 & 0 & 0 & \ldots & 0 & 0 & 0 &  & 0 & \lambda
\end{array}\right)
\end{eqnarray*}

\begin{eqnarray*}
 & \lambda(\lambda I_{W}-\V)^{-1}I_{W}=\\
 & \tiny\left(\begin{array}{c|cccccccccc}
 & 0 & 1 & 2 & \ldots & r-1 & r & r+1 & \ldots & S-1 & S\\
\hline 0 & 0 & 0 & 0 &  & 0 & 0 & 0 &  & 0 & 1\\
1 & 0 & \frac{\lambda}{\nu_{1}+\lambda} & 0 &  & 0 & 0 & 0 &  & 0 & \frac{\nu_{1}}{\nu_{1}+\lambda}\\
2 & 0 & 0 & \frac{\lambda}{\nu_{2}+\lambda} &  & 0 & 0 & 0 &  & 0 & \frac{\nu_{2}}{\nu_{2}+\lambda}\\
\vdots & \vdots & \vdots & \ddots &  &  &  &  &  &  & \vdots\\
r & 0 & 0 & 0 & \ldots & 0 & \frac{\lambda}{\nu_{r}+\lambda} & 0 & \ldots & 0 & \frac{\nu_{r}}{\nu_{r}+\lambda}\\
r+1 & 0 & 0 & 0 &  & 0 & 0 & 1 & \ldots & 0 & 0\\
\vdots & \vdots & \vdots & \vdots & \vdots & \vdots & \vdots & \ddots & \ddots & \vdots & \vdots\\
S-1 & 0 & 0 & 0 & \ldots & 0 & 0 &  &  & 1 & 0\\
S & 0 & 0 & 0 & \ldots & 0 & 0 & 0 &  & 0 & 1
\end{array}\right)
\end{eqnarray*}

\begin{eqnarray*}
 & R=\\
 & \tiny\left(\begin{array}{c|cccccc}
 & 0 & 1 & 2 & \ldots & S-1 & S\\
\hline 0 & 1 & 0 & 0 & \ldots & 0 & 0\\
1 & 1 & 0 & 0 & \ldots & 0 & 0\\
2 & 0 & 1 & 0 &  & 0 & 0\\
\vdots & \vdots & \vdots & \ddots & \ddots & \vdots & \vdots\\
S-1 & 0 & 0 & 0 & \ldots & 0 & 0\\
S & 0 & 0 & 0 & \ldots & 1 & 0
\end{array}\right)
\end{eqnarray*}

\begin{eqnarray*}
 & \lambda(\lambda I_{W}-\V)^{-1}I_{W}R=\\
 & \tiny\left(\begin{array}{c|cccccccccc}
 & 0 & 1 & 2 & \ldots & r-1 & r & r+1 & \ldots & S-1 & S\\
\hline 0 & 0 & 0 & 0 &  & 0 & 0 & 0 &  & 1 & 0\\
1 & \frac{\lambda}{\nu_{1}+\lambda} & 0 & 0 &  & 0 & 0 & 0 &  & \frac{\nu_{1}}{\nu_{1}+\lambda} & 0\\
2 & 0 & \frac{\lambda}{\nu_{2}+\lambda}0 & 0 &  & 0 & 0 & 0 &  & \frac{\nu_{2}}{\nu_{2}+\lambda} & 0\\
\vdots & \vdots & \vdots & \ddots &  &  &  &  &  & \vdots & 0\\
r & 0 & 0 & 0 & \ldots & \frac{\lambda}{\nu_{r}+\lambda} & 0 & 0 &  & \frac{\nu_{r}}{\nu_{r}+\lambda} & 0\\
r+1 & 0 & 0 & 0 &  & 0 & 1 & 0 &  & 0 & 0\\
\vdots & \vdots & \vdots & \vdots & \vdots & \vdots & \vdots & \ddots & \ddots & \vdots & 0\\
S-1 & 0 & 0 & 0 & \ldots & 0 & 0 & 0 & \ddots & 0 & 0\\
S & 0 & 0 & 0 & \ldots & 0 & 0 & 0 &  & 1 & 0
\end{array}\right)
\end{eqnarray*}

Inserting this and \prettyref{eq:LS-EMC-mm1-chain-r-S-theta-hat}
into \prettyref{eq:LS-EMC-theta-hat-equation-in-r-S} finishes the
proof.\end{proof}
\begin{prop}
\label{prop:LS-EMC-mm1inf-r-Q} We consider an exponential single
server queue with state dependent service rates, environment dependent
replenishment rates, and an attached inventory under $(r,Q)$ policy
(with $0\leq r<Q\in\mathbb{N}$), and lost sales when the inventory
is depleted.

Using the definitions of Section \ref{sect:LS-EMC} we set the environment
state space $K:=\{0,...,S\}$ with $K_{B}=\{0\}$, $X(t)$ the queue
length at time $t$, and $Y(t)=k$ indicates that at time $t$ the
stock contains exactly $k$ items. The strictly positive transition
intensities are 
\begin{eqnarray*}
q((n,k)\qsep(n+1,k)) & = & \lambda\qquad k>0\\
q((n,k)\qsep(n,k+Q)) & = & \nu_{k}\qquad0\leq k\leq r\\
q((n,k)\qsep(n-1,k-1)) & = & \mu(n)\qquad n>0,1\leq k\leq r+Q\\
q((n,k)\qsep(l,m)) & = & 0,\qquad otherwise
\end{eqnarray*}

The steady state $\hat{\pi}$ has product form 

\begin{equation}
\hat{\pi}(n,k)=\xi(n)\hat{\theta}(k)\,,\quad(n,k)\in\mathbb{N}_{0}\times K\,,\label{eq:LS-EMC-mm1-r-Q-product-form-assumption}
\end{equation}

with

\[
\xi(n):=C^{-1}\left(\prod_{i=1}^{n}\frac{\lambda}{\mu(i)}\right),~~n\in\mathbb{N}_{0},\qquad\text{with normalization constant}\, C^{-1}
\]

and 
\begin{equation}
\hat{\theta}(k)=\begin{cases}
C^{-1}\cdot\prod_{i=1}^{k}\left(\frac{\lambda+\nu_{i}}{\lambda}\right)^{i} & ,\qquad0\leq k\leq r\\
C^{-1}\cdot\prod_{i=1}^{r}\left(\frac{\lambda+\nu_{i}}{\lambda}\right)^{i} & ,\qquad r+1\leq k\leq Q-1\\
C^{-1}\prod_{i=1}^{r}\left(\frac{\lambda+\nu_{i}}{\lambda}\right)^{i}-\prod_{i=1}^{k-Q}\left(\frac{\lambda+\nu_{i}}{\lambda}\right)^{i} & ,\qquad Q\leq k\leq r+Q-1\\
0 & ,\qquad k=r+Q
\end{cases}\label{eq:LS-EMC-mm1-chain-r-Q-theta-hat}
\end{equation}

with normalization constant

\[
C=(Q-r)\prod_{i=1}^{k}\left(\frac{\lambda+\nu_{i}}{\lambda}\right)^{i}
\]
\end{prop}
\begin{proof}
According to \prettyref{thm:LCS-EMC-product-form-exp-service} the
marginal distribution $\hat{\theta}$ is a solution of the equation
\begin{equation}
\hat{\theta}\lambda(\lambda I_{W}-\V)^{-1}I_{W}R=\hat{\theta}\label{eq:LS-EMC-theta-hat-equation-in-r-Q}
\end{equation}
We calculate the matrix $\lambda(\lambda I_{W}-V)^{-1}I_{W}R$ explicitly
(the remark from \prettyref{prop:LS-EMC-mm1inf-r-S} on indexing the
matrices applies here as well).

\begin{eqnarray*}
 & (\lambda I_{W}-\V)=\\
 & \tiny\left(\begin{array}{c|ccccccccccccc}
 & 0 & 1 & 2 & \ldots & r-1 & r & r+1 & \ldots & Q & 1+Q & 2+Q & \ldots & r+Q\\
\hline 0 & \nu_{0} & 0 & 0 &  & 0 & 0 & 0 &  & -\nu_{0} &  &  &  & 0\\
1 & 0 & (\nu_{1}+\lambda) & 0 &  & 0 & 0 & 0 &  & 0 & -\nu_{1} &  &  & 0\\
2 & 0 & 0 & (\nu_{2}+\lambda) &  & 0 & 0 & 0 &  & 0 &  & -\nu_{2} &  & 0\\
\vdots & \vdots & \vdots & \ddots &  &  &  &  &  &  &  &  & \ddots & \vdots\\
r & 0 & 0 & 0 & \ldots & 0 & (\nu_{r}+\lambda) & 0 & \ldots & 0 &  &  &  & -\nu_{r}\\
r+1 & 0 & 0 & 0 &  & 0 & 0 & \lambda & \ldots & 0 &  &  &  & 0\\
\vdots & \vdots & \vdots & \vdots & \vdots & \vdots & \vdots & \ddots & \ddots & \vdots &  &  &  & \vdots\\
Q & 0 & 0 & 0 & \ldots & 0 & 0 &  &  & \lambda &  &  &  & 0\\
1+Q & 0 & 0 & 0 &  & 0 & 0\\
2+Q\\
\vdots\\
r+Q & 0 & 0 & 0 & \ldots & 0 & 0 & 0 &  & 0 &  &  &  & \lambda
\end{array}\right)
\end{eqnarray*}

\begin{eqnarray*}
 & \lambda(\lambda I_{W}-\V)^{-1}I_{W}=\\
 & \tiny\left(\begin{array}{c|ccccccccccccc}
 & 0 & 1 & 2 & \ldots & r-1 & r & r+1 & \ldots & Q & 1+Q & 2+Q & \ldots & r+Q\\
\hline 0 & 0 & 0 & 0 &  & 0 & 0 & 0 &  & 1 & 0 & 0 &  & 0\\
1 & 0 & \frac{\lambda}{\nu_{1}+\lambda} & 0 &  & 0 & 0 & 0 &  & 0 & \frac{\nu_{1}}{\nu_{1}+\lambda}\\
2 & 0 & 0 & \frac{\lambda}{\nu_{2}+\lambda} &  & 0 & 0 & 0 &  & 0 &  & \frac{\nu_{2}}{\nu_{2}+\lambda}\\
\vdots & \vdots & \vdots & \ddots &  &  &  &  &  &  &  &  & \ddots\\
r & 0 & 0 & 0 & \ldots & 0 & \frac{\lambda}{\nu_{r}+\lambda} & 0 & \ldots & 0 &  &  &  & \frac{\nu_{r}}{\nu_{r}+\lambda}\\
r+1 & 0 & 0 & 0 &  & 0 & 0 & 1 & \ldots & 0 &  &  &  & 0\\
\vdots & \vdots & \vdots & \vdots & \vdots & \vdots & \vdots & \ddots & \ddots & \vdots &  &  &  & \vdots\\
Q & 0 & 0 & 0 & \ldots & 0 & 0 & 0 &  & 1 &  &  &  & 0\\
1+Q & 0 & 0 & 0 &  & 0 & 0 & 0 &  &  & 1\\
2+Q & 0 & 0 & 0 &  & 0 & 0 & 0 &  &  &  & 1\\
 & \vdots & \vdots & \vdots &  & \vdots & \vdots & \vdots &  &  &  &  & \ddots\\
S & 0 & 0 & 0 & \ldots & 0 & 0 & 0 &  & 0 & 0 & 0 &  & 1
\end{array}\right)
\end{eqnarray*}

\begin{eqnarray*}
 & R=\\
 & \tiny\left(\begin{array}{c|cccccc}
 & 0 & 1 & 2 & \ldots & r-1+Q & r+Q\\
\hline 0 & 1 & 0 & 0 &  & 0 & 0\\
1 & 1 & 0 & 0 &  & 0 & 0\\
2 & 0 & 1 & 0 &  & 0 & 0\\
\vdots & \vdots & \vdots & \ddots &  & \vdots & \vdots\\
r-1+Q & 0 & 0 & 0 & \ldots & 0 & 0\\
r+Q & 0 & 0 & 0 & \ldots & 1 & 0
\end{array}\right)
\end{eqnarray*}

\begin{eqnarray*}
 & \lambda(\lambda I_{W}-\V)^{-1}I_{W}R=\\
 & \tiny\left(\begin{array}{c|cccccccccccccc}
 & 0 & 1 & 2 & \ldots & r-1 & r & r+1 & \ldots & Q-1 & Q & 1+Q & 2+Q & \ldots & r+Q\\
\hline 0 & 0 & 0 & 0 &  & 0 & 0 & 0 &  & 1 & 0 & 0 & 0 &  & 0\\
1 & \frac{\lambda}{\nu_{1}+\lambda} & 0 & 0 &  & 0 & 0 & 0 &  &  & \frac{\nu_{1}}{\nu_{1}+\lambda}\\
2 & 0 & \frac{\lambda}{\nu_{2}+\lambda} &  &  & 0 & 0 & 0 &  &  &  & \frac{\nu_{2}}{\nu_{2}+\lambda}\\
\vdots & \vdots & \vdots & \ddots &  &  &  &  &  &  &  &  & \ddots\\
r & 0 & 0 & 0 & \ldots & \frac{\lambda}{\nu_{r}+\lambda} & 0 & 0 &  &  &  &  &  & \frac{\nu_{r}}{\nu_{r}+\lambda}\\
r+1 & 0 & 0 & 0 &  & 0 & 1 &  &  &  &  &  &  &  & 0\\
\vdots & \vdots & \vdots & \vdots & \vdots & \vdots & \vdots & \ddots &  &  &  &  &  &  & \vdots\\
Q-1 & 0 & 0 & 0 & 0 & 0 & 0 &  & \ddots\\
Q & 0 & 0 & 0 & \ldots & 0 & 0 & 0 &  & 1 &  &  &  &  & 0\\
1+Q & 0 & 0 & 0 &  & 0 & 0 & 0 &  &  & 1\\
2+Q & 0 & 0 & 0 &  & 0 & 0 & 0 &  &  &  & 1\\
 & \vdots & \vdots & \vdots &  & \vdots & \vdots & \vdots &  &  &  &  & \ddots\\
r+Q & 0 & 0 & 0 & \ldots & 0 & 0 & 0 &  &  & 0 & 0 & 0 & 1 & 0
\end{array}\right)
\end{eqnarray*}
 Note that even for constant values $\nu_{k}=\nu$ the marginal distribution
$\hat{\theta}$ under $(r,Q)$ policy differs from the marginal steady
state distribution $P(Y(t)=k)$ in continuous time from \prettyref{ex:LS-original-r-Q}.

Inserting this and \prettyref{eq:LS-EMC-mm1-chain-r-Q-theta-hat}
into \prettyref{eq:LS-EMC-theta-hat-equation-in-r-Q} finishes the
proof.
\end{proof}

\subsection{Systems with non-exponential service requests\label{sect:LS-EMC-appl-non-exp} }
\begin{prop}
We consider a single server queue of $M/G/1/\infty$-type, with state
dependent service speeds, state dependent selection of requested service
times, exponential-$\nu$ replenishment times, and an attached inventory
under $(r=0,S)$ policy (with $0<S\in\mathbb{N}$), and lost sales
when the inventory is depleted (see Definition \ref{def:LCS-MG1Inf-interference-less}).

We have $K=\{S,S-1,\dots,1,0\}$ with $K_{B}=\{0\}$.

The stochastic jump matrix $R$ represents the downward jumps of the
inventory 
\[
R=\left(\begin{array}{c|cc}
 & 0\ldots S-1 & S\\
\hline 0 & \left(1,0,\ldots,0\right) & 0\\
\begin{array}{c}
1\\
\vdots\\
S
\end{array} & \left(\begin{array}{ccc}
1\\
 & \ddots\\
 &  & 1
\end{array}\right) & 0
\end{array}\right)\,,
\]
and because the environment moves independently only if there is stockout,
the environment generator $\V$ has only non zero entries $\v(0,S)=\nu,~\v(0,0)=-\nu$.
So with $K_{B}=\{0\}$ the requirement of \prettyref{thm:MG1-PF-result}
is fulfilled. 
\[
\V=\left(\begin{array}{c|cc}
 & 0 & 1\ldots S\\
\hline 0 & -\nu & \left(0,\ldots,0,\nu\right)\\
\begin{array}{c}
1\\
\vdots\\
S
\end{array} &  & 0
\end{array}\right)
\]

From \prettyref{thm:MG1-PF-result} we conclude that the Markov chain
$(\hat{X},\hat{Y})$, embedded at departure instants of customers
has a stationary distribution $\hat{\pi}$ of product form 
\[
\hat{\pi}(n,k)=\hat{\xi}(n)\hat{\theta}(k)\,,\quad(n,k)\in\mathbb{N}_{0}\times K.
\]

Here $\hat{\xi}$ is the steady state distribution of the Markov chain
with one-step transition matrix \eqref{MGs1-1} derived for the queue
length process at departure points in a system with the same parameters
as under consideration but without environment, i.e, a solution of
$\hat{\xi}\tilde{P}=\hat{\xi}$, and $\hat{\theta}$ is for $k\in\{0,1,\dots,S\}$
\begin{equation}
\hat{\theta}(k)=\frac{1}{S}\quad k\neq S,\qquad\hat{\theta}(S)=0\label{eq:LS-EMC-mgs1-r0-S-theta}
\end{equation}

According to \prettyref{thm:MG1-PF-result}, $\hat{\theta}$ is a
stochastic solution of the equation $\hat{\theta}(I_{W}-\V)^{-1}I_{W}R=\hat{\theta}$
We calculate the matrix $H=(I_{W}-\V)^{-1}I_{W}R$ explicitly.

\[
(I_{W}-\V)=\left(\begin{array}{c|cc}
 & 0 & 1\ldots S\\
\hline 0 & \nu & \left(0,\ldots,0,-\nu\right)\\
\begin{array}{c}
1\\
\vdots\\
S
\end{array} & \begin{array}{c}
0\\
\vdots\\
0
\end{array} & \left(\begin{array}{ccc}
1\\
 & \ddots\\
 &  & 1
\end{array}\right)
\end{array}\right)
\]

\[
(I_{W}-\V)^{-1}=\left(\begin{array}{c|cc}
 & 0 & 1\ldots S\\
\hline 0 & \frac{1}{\nu} & \left(0,\ldots,0,1\right)\\
\begin{array}{c}
1\\
\vdots\\
S
\end{array} & \begin{array}{c}
0\\
\vdots\\
0
\end{array} & \left(\begin{array}{ccc}
1\\
 & \ddots\\
 &  & 1
\end{array}\right)
\end{array}\right)
\]

\[
(I_{W}-\V)^{-1}I_{W}=\left(\begin{array}{c|cc}
 & 0 & 1\ldots S\\
\hline 0 & 0 & \left(0,0,\ldots,1\right)\\
\begin{array}{c}
1\\
\vdots\\
S
\end{array} & \begin{array}{c}
0\\
\vdots\\
0
\end{array} & \left(\begin{array}{ccc}
1\\
 & \ddots\\
 &  & 1
\end{array}\right)
\end{array}\right)
\]

\[
H=(I_{W}-\V)^{-1}I_{W}R=\left(\begin{array}{c|cc}
 & 0\ldots S-1 & S\\
\hline 0 & \left(0,0,\ldots,1\right) & 0\\
\begin{array}{c}
1\\
\vdots\\
S
\end{array} & \left(\begin{array}{ccc}
1\\
 & \ddots\\
 &  & 1
\end{array}\right) & \begin{array}{c}
0\\
\vdots\\
0
\end{array}
\end{array}\right)
\]

$\hat{\theta}$ defined in \prettyref{eq:LS-EMC-mgs1-r0-S-theta}
is the unique solution of the equation \prettyref{eq:LS-EMC-mgs1-theta-check-equation}.
\end{prop}
\begin{prop}
We consider a single server queue of $M/G/1/\infty$-type, with state
dependent service speeds, state dependent selection of requested service
times, inventory management policy $(r,Q)$ or $(r,S)$, and zero
lead times (see Definition \ref{def:LCS-MG1Inf-interference-less},
and note that lost sales do not occur because of zero lead time).

In the case of $(r,S)$ policy the inventory size after the first
delivery will stay on between $r+1$ and $S$, therefore for long
term behaviour of the system we take in account only environment states
$K=\{r+1,r+2,...,S\}$. The zero lead time means $\V=0$, $K_{B}=\emptyset$,
and the corresponding $R$ matrix has the form

\[
R=\left(\begin{array}{c|cc}
 & r+1\ldots S-1 & S\\
\hline r+1 & \left(0,0,\ldots,0\right) & 1\\
\begin{array}{c}
r+2\\
\vdots\\
S
\end{array} & \left(\begin{array}{ccc}
1\\
 & \ddots\\
 &  & 1
\end{array}\right) & 0
\end{array}\right)
\]

The steady state distribution has a product form 
\[
\hat{\pi}(n,k)=\hat{\xi}(n)\hat{\theta}(k)\,,\quad(n,k)\in\mathbb{N}_{0}\times K\,,
\]

with

\begin{equation}
\hat{\theta}(k)=\frac{1}{S-r}\,,\quad k\in K\,.\label{eq:LS-EMC-mgs1-zero-lead-theta}
\end{equation}
\end{prop}
\begin{proof}
According to \prettyref{thm:MG1-PF-result} $\hat{\theta}$ is a stochastic
solution of the equation $\hat{\theta}(I_{W}-\V)^{-1}I_{W}R=\hat{\theta}$
We calculate the matrix $H=(I_{W}-\V)^{-1}I_{W}R$, which in the case
of the model equivalent to

\[
\hat{\theta}\underbrace{(I_{W}-V)^{-1}}_{=I}\underbrace{I_{W}}_{=I}R=\hat{\theta}\Longleftrightarrow\hat{\theta}R=\hat{\theta}
\]

with a unique stochastic solution \eqref{eq:LS-EMC-mgs1-zero-lead-theta}.

For system under $(r,Q)$ policy with zero lead times \eqref{eq:LS-EMC-mgs1-zero-lead-theta}
holds as well, the proof is analogous, we just set $S=r+Q$. \end{proof}
\begin{rem*}
Similar results for the steady state of queueing-inventory systems
with zero lead times (without speeds) were obtained by Vineetha in
\cite[Theorem 5.2.1]{vineetha:08} for the case of i.i.d service times. 
\end{rem*}

\section{Useful lemmata\label{sect:LS-EMC-useful-lemma}}

In our proofs we require the matrix $(\lambda I_{W}-\V)$ to be invertible,
the following lemma is the key to this property in case of finite
$K$.
\begin{lem}
\label{lem:LS-EMC-finite-K-M-invertible} Let $M\in\mathbb{R}^{K\times K}$,
where the set of indices is partitioned according to $K=K_{W}+K_{B}$,
$K_{W}\neq\emptyset$, and $|K|<\infty$, whose diagonal elements
have following properties:

\begin{equation}
|M_{kk}|=\sum_{m\in K\setminus\{k\}}|M_{km}|,\qquad\forall k\in K_{B}\label{eq:LS-EMC-inverse-3}
\end{equation}
\begin{equation}
|M_{kk}|>\sum_{m\in K\setminus\{k\}}|M_{km}|,\qquad\forall k\in K_{W}\label{eq:LS-EMC-inverse-1}
\end{equation}
and it holds the \underline{flow condition} \index{flow condition}
\begin{equation}
\forall\tilde{K}_{B}\subset K_{B},~\tilde{K}_{B}\neq\emptyset:\quad\exists~~~k\in\tilde{K}_{B},~~m\in\tilde{K}_{B}^{c}:~~M_{km}\neq0\,.\label{eq:LS-EMC-inverse-flow-condition}
\end{equation}

Then $M$ is invertible. \end{lem}
\begin{rem*}
The Lemma \ref{lem:LS-EMC-finite-K-M-invertible} does not require
the matrix to be irreducible. Since we are interested in systems with
reducible matrices $\V$ which appear in inventory models (see Propositions
\prettyref{prop:LS-EMC-mm1inf-r-S} and \prettyref{prop:LS-EMC-mm1inf-r-Q}),
we have to modify the proof for irreducible matrices which can be
found e.g. in \cite[Lemma 4.12]{kanzow:05}. \end{rem*}
\begin{rem}
\label{rem:LS-EMC-inverse-flow-graph-interpretation}

Consider the directed transition graph of $M$, with vertices $K$
and edges ${\cal E}$ defined by $km\in{\cal E}\Longleftrightarrow K_{km}>0.$
Then the condition \eqref{eq:LS-EMC-inverse-flow-condition} guarantees
the existence of a path from any vertex in $K_{B}$ to a some vertex
in $K_{W}$.\end{rem}
\begin{proof}
We prove the lemma by contradiction, and let $x=(x_{k}:k\in K)$ be
a vector with 
\begin{equation}
Mx=0\text{ with }x\neq0.\label{eq:LS-EMC-inverse-x-assumption}
\end{equation}

The property $Mx=0$ leads for all $k\in K$ to 
\begin{eqnarray}
-M_{kk}x_{k} & = & \sum_{m\in K\setminus\{k\}}M_{km}x_{m},\nonumber \\
\Longrightarrow|M_{kk}||x_{k}| & \leq & \sum_{m\in K\setminus\{k\}}|M_{km}||x_{m}|,\nonumber \\
\Longrightarrow|M_{kk}|\frac{|x_{k}|}{||x||_{\infty}} & \leq & \sum_{m\in K\setminus\{k\}}|M_{km}|\underbrace{\frac{|x_{m}|}{||x||_{\infty}}}_{\leq1}\leq\sum_{m\in K\setminus\{k\}}|M_{km}|,\label{eq:LS-EMC-inverse-2}
\end{eqnarray}

We denote by $J$ the set of indices of elements $x_{k}$ of $x$
with the largest absolute value 
\[
J:=\{k\in K|\,\,|x_{k}|=||x||_{\infty}\}\,.
\]
Because of $x\neq0$ and $|K|<\infty$ the set $J$ is non empty.

First we show that 
\begin{equation}
\forall k\in K_{W}:|x_{k}|<||x||_{\infty}\label{eq:LS-EMC-inverse-max-x-not-in-K-w}
\end{equation}
holds, which implies 
\begin{equation}
K_{W}\subset J^{c}\,.\label{eq:LS-EMC-inverse-max-x-not-in-K-w-1}
\end{equation}

For $K_{B}=\emptyset$ the proof is complete because we have 
\[
K=K_{W}\subseteq J^{c}\subsetneqq K\,,
\]
and so we proceed with the proof for $K_{B}\neq\emptyset$.

From \prettyref{eq:LS-EMC-inverse-2} and \prettyref{eq:LS-EMC-inverse-1}
it follows for all $k\in K_{W}$ 
\begin{eqnarray}
|M_{kk}|\frac{|x_{k}|}{||x||_{\infty}} & \leq & \sum_{m\in K\setminus\{k\}}|M_{km}|<|M_{kk}|,\nonumber \\
\Longrightarrow|M_{kk}|\frac{|x_{k}|}{||x||_{\infty}} & < & |M_{kk}|,\label{eq:LS-EMC-inverse-5}
\end{eqnarray}

The inequality \prettyref{eq:LS-EMC-inverse-5} is valid if and only
if $\frac{|x_{k}|}{||x||_{\infty}}$ is strictly less than $1$, which
implies $|x_{k}|<||x||_{\infty}$ and therefore \prettyref{eq:LS-EMC-inverse-max-x-not-in-K-w-1}.\\

\begin{figure}[h]
\centering{}    \begin{tikzpicture}
      \fill[black!10] (-2,-1) rectangle (0,1);
      \draw (-4,-1) rectangle (0,1);
      \draw(-2,1)--(-2,-1);
      \draw(-1.5,0)node (KB) {$K_B$};
      \draw(-0.5,0)node (J) {$J$};
      \draw(-0.5,0)node (J)[shape = circle,draw] {$J$};
      \draw(-3,0)node (KB) {$K_W$};
    \end{tikzpicture} \caption{\label{fig:LS-EMC-inverse-set} Sets in \prettyref{lem:LS-EMC-finite-K-M-invertible}.
The set $K_{B}$ is gray.}
\end{figure}

Next, we analyze the set $J\subset K_{B}$. For $k\in J$ we examine
the $k$th row of the equation $Mx=0$.

For all $k\in J$ it follows from \prettyref{eq:LS-EMC-inverse-2}
\begin{eqnarray}
|M_{kk}| & \leq & \sum_{m\in K\setminus\{k\}}|M_{km}|\underbrace{\frac{|x_{m}|}{||x||_{\infty}}}_{\leq1}\leq\sum_{m\in K\setminus\{k\}}|M_{km}|\leq|M_{kk}|,\nonumber \\
\Longrightarrow\sum_{m\in K\setminus\{k\}}|M_{km}|\underbrace{\frac{|x_{m}|}{||x||_{\infty}}}_{\leq1} & = & \sum_{m\in K\setminus\{k\}}|M_{km}|\label{ineq:LS-EMC-inverse-4}
\end{eqnarray}
Because $\frac{|x_{m}|}{||x||_{\infty}}$ is strictly less than $1$
for all $m\in J^{c}$, the inequality \eqref{ineq:LS-EMC-inverse-4}
yields 
\[
M_{km}=0,\qquad\forall k\in J,m\in J^{c}
\]

Since $K_{W}\subset J^{c}$ we have a contradiction to the existence
of a path of positive values $M_{km}$ from $k\in J\subset K_{B}$
to $K_{W}$ which is guaranteed by \eqref{eq:LS-EMC-inverse-flow-condition}. \end{proof}
\begin{example}
\label{exmp:LCS-chain-matrix-inverse-proof} This example provides
a matrix $M$ which fulfills the requirements of \prettyref{lem:LS-EMC-finite-K-M-invertible}
and is therefore invertible It is neither irreducible nor strictly
diagonal dominant. We set $\lambda,\v{(2,3}),\v{(3,2)},\v{(4,3)},\v{(4,6)},\v{(5,4)},\v{(6,3)}>0$,
all other entries are zero. \prettyref{fig:LS-EMC-chain-positive-nu-rate-graph-example}
shows the resulting flow graph according to \prettyref{rem:LS-EMC-inverse-flow-graph-interpretation}. 

$M=$ 
\[
\left(\begin{array}{c|cccccc}
 & 1\in K_{W} & 2\in K_{W} & 3\in K_{B} & 4\in K_{B} & 5\in K_{B} & 6\in K_{B}\\
\hline 1\in K_{W} & -\lambda\\
2\in K_{W} &  & -(\lambda+\v{(2,3)}) & \v{(2,3)}\\
3\in K_{B} &  & v{(3,2)} & -\v{(3,2)}\\
4\in K_{B} &  &  & \v{(4,3)} & -(\v{(4,3)}+\v{(4,6)}) &  & \v{(4,6)}\\
5\in K_{B} &  &  &  & \v{(5,4)} & -\v{(5,4)}\\
6\in K_{B} &  &  & \v{(6,3)} &  &  & -\v{(6,3)}
\end{array}\right)
\]
Note, that this matrix is of the form $M=\lambda I_{W}+\V$ with $\V=$
\[
\left(\begin{array}{c|cccccc}
 & 1\in K_{W} & 2\in K_{W} & 3\in K_{B} & 4\in K_{B} & 5\in K_{B} & 6\in K_{B}\\
\hline 1\in K_{W} & 0\\
2\in K_{W} &  & -\v{(2,3)} & \v{(2,3)}\\
3\in K_{B} &  & \v{(3,2)} & -\v{(3,2)}\\
4\in K_{B} &  &  & \v{(4,3)} & -(\v{(4,3)}+\v{(4,6)}) &  & \v{(4,6)}\\
5\in K_{B} &  &  &  & \v{(5,4)} & -\v{(5,4)}\\
6\in K_{B} &  &  & \v{(6,3)} &  &  & -\v{(6,3)}
\end{array}\right)\,,
\]
and fits therefore exactly into the realm of our investigations of
loss systems in a random environment. 

\begin{figure}[h]
\centering{} \begin{tikzpicture}
      \fill[black!10] (1,0.5) rectangle (4,-3.5);
      \draw (-1,0.5) rectangle (1,-3.5);
      \draw (-1,0.5) rectangle (4,-3.5);
      \draw (0, -3) node (K_W) {$K_W$};
      \draw (3, -3) node (K_B) {$K_B$};
      \path 	      (0.0, -1.0) node (Y1)[shape=rectangle,draw] {$1$}
	      (0.0, -2.0) node (Y2)[shape=rectangle,draw] {$2$}
	      (2.0, 0.0) node (Y3)[shape=rectangle,draw] {$3$}
	      (2.0, -1.0) node (Y4)[shape=rectangle,draw] {$4$}
	      (2.0, -2.0) node (Y5)[shape=rectangle,draw] {$5$}
	      (2.0, -3.0) node (Y6)[shape=rectangle,draw] {$6$}
      ;
      \draw (Y2){} edge[->, bend left=10] node[auto]{}(Y3);
      \draw (Y3){} edge[->, bend left=10] node[auto]{}(Y2);
      \draw (Y4){} edge[->] node[auto]{}(Y3);
      \draw (Y4){} edge[->,bend left=30] node[auto]{}(Y6);
      \draw (Y5){} edge[->] node[auto]{}(Y4);
      \draw (Y6){} edge[->,bend left=30] node[auto]{}(Y3);
\end{tikzpicture}\caption{Graph from example \label{fig:LS-EMC-chain-positive-nu-rate-graph-example}
according to the remark of \prettyref{lem:LS-EMC-non-zero-path}.}
\end{figure}

\end{example}
For infinite $K$ we have the following results.
\begin{prop}
\label{prop:LS-EMC-inverse-diagonal-very-dominant} Let $M\in\mathbb{R}^{K\times K}$,
be a linear operator on $\ell_{\infty}(\mathbb{R}^{K})$. If for all
$k\in K$ holds $|M_{kk}|\geq\sum_{m\in K\setminus\{k\}}|M_{km}|+\varepsilon$
for some $\varepsilon>0$ and $\sup_{k\in K}|M_{kk}|<\infty$ ,then
$M$ is invertible. \end{prop}
\begin{proof}
\textbf{(1)} Assume $M_{kk}>0$ for all $k\in K$. Define $\beta:=\frac{1}{\sup_{k\in K}M_{kk}}$,
then it holds

\begin{eqnarray}
||I-\beta M||_{\infty} & = & \sup_{k\in K}\left(|1-\underbrace{\beta M_{kk}}_{\leq1}|+\beta\underbrace{\sum_{m\in K\setminus\{k\}}|M_{km}}_{\leq M_{kk}-\varepsilon}|\right)\\
 & \leq & \sup_{k\in K}\left(1-\beta M_{kk}+\beta(M_{kk}-\varepsilon)\right)<1
\end{eqnarray}
Thus $M$ is invertible and it holds 
\[
M^{-1}=\beta\sum_{n=0}^{\infty}(I-\beta M)^{n}
\]

\textbf{(2)} We define a matrix $S$ with 
\begin{equation}
S_{km}=\begin{cases}
1 & k=m,\ M_{kk}>0\\
-1 & k=m,\ M_{kk}<0\\
0 & \text{otherwise}
\end{cases}
\end{equation}
Then $S$ is a bounded invertible operator with $S^{-1}=S$. According
to \textbf{(1)} $SM$ is invertible and it holds $M^{-1}=(SSM)^{-1}=(SM)^{-1}S^{-1}=(SM)^{-1}S$ \end{proof}
\begin{lem}
\label{lem:LS-EMC-M-invertible-inf} Let $M\in\mathbb{R}^{K\times K}$,
be a linear operator on $\ell_{\infty}(\mathbb{R}^{K})$ where the
set of indices is partitioned according to $K=K_{W}+K_{B}$, $K_{W}\neq\emptyset$,
and $|K_{B}|<\infty$, with the following properties:\\

\underline{{\em Flow condition}}:~~ Define a directed graph $(K,{\cal E})$
by 
\[
(k,m)\in{\cal E}:\Leftrightarrow M(k,m)\neq0\,.
\]
Then for any ~$k\in K_{B}$ there exists some $m=m(k)\in K_{W}$
such that there exists a directed path of finite length in $(K,{\cal E})$
from $k$ to $m$. 
\begin{equation}
\text{The sequence}~~~|M_{mm}|,m\in K,~~~\text{is bounded.}\label{eq:LS-EMC-inverse-inf-4}
\end{equation}

\begin{equation}
|M_{kk}|=\sum_{m\in K\setminus\{k\}}|M_{km}|,\qquad\forall k\in K_{B}\,.\label{eq:LS-EMC-inverse-inf-1}
\end{equation}
\begin{equation}
\sup_{k\in K_{W}}\sum_{m\in K\setminus\{k\}}|M_{km}|=:ND(K_{W})<\infty\,.\label{eq:LS-EMC-inverse-inf-2}
\end{equation}
There exists some $\varepsilon(K_{W})>0$ such that 
\begin{equation}
\inf_{m\in K_{W}}|M_{mm}|=ND(K_{W})+\varepsilon(K_{W})\label{eq:LS-EMC-inverse-inf-3}
\end{equation}
holds.

Then $M$ is injective. \end{lem}
\begin{rem*}
The sequence $|M_{mm}|,m\in K,$ needs not be bounded. \end{rem*}
\begin{proof}
In the case $K_{B}=\emptyset$ the matrix $M$ is strictly diagonal
dominant and thus invertible according to \prettyref{prop:LS-EMC-inverse-diagonal-very-dominant}.

Let $x=(x_{k}:k\in K)\in\ell_{\infty}(\mathbb{R}^{K})$ be any vector
with 
\begin{equation}
Mx=0\text{ with }x\neq0\label{eq:LS-EMC-inverse-x-assumption-inf}
\end{equation}
\textbf{(a)} To show that 
\begin{equation}
\forall k\in K_{W}:|x_{k}|<||x||_{\infty}\label{eq:LS-EMC-inverse-max-x-not-in-K-w-inf}
\end{equation}
holds, is a word-by-word analogue of that property in the proof of
\prettyref{lem:LS-EMC-finite-K-M-invertible}.

\noindent \textbf{(b)} We show: $\{|x_{k}|:k\in K_{W}\}$ is uniformly
bounded away from $||x||_{\infty}$ from below.

The property $Mx=0$ leads for all $k\in K$ to 
\begin{eqnarray*}
-M_{kk}x_{k} & = & \sum_{m\in K\setminus\{k\}}M_{km}x_{m}\Longrightarrow\\
|M_{kk}||x_{k}| & \leq & \sum_{m\in K\setminus\{k\}}|M_{km}||x_{m}|\leq||x||_{\infty}\sum_{m\in K\setminus\{k\}}|M_{km}|\leq||x||_{\infty}ND(K_{W})\,,
\end{eqnarray*}
and therefore

\begin{eqnarray*}
|x_{k}|\inf_{m\in K_{W}}|M_{mm}| & \leq & ||x||_{\infty}ND(K_{W})~~\Longrightarrow\\
|x_{k}|\leq\frac{ND(K_{W})}{\inf_{m\in K_{W}}|M_{mm}|}||x||_{\infty} & = & \left(1-\underbrace{\frac{\varepsilon(K_{W})}{ND(K_{W})+\varepsilon(K_{W})}}_{\in(0,1)}\right)||x||_{\infty}
\end{eqnarray*}

\noindent \textbf{(c)} We show: $J:=\{k\in K:|x_{k}|=||x||_{\infty}\}\neq\emptyset$
and $K_{W}\subset J^{c}$.

The second property follows from \textbf{(b)}, while the first property
holds, because the set $\{|x_{k}|:k\in K_{W}\}$ is uniformly bounded
away from $||x||_{\infty}$ from below and $K_{B}$ is finite, so
there must exist some $k(0)\in K_{B}$ where $|x_{k(0)}|=||x||_{\infty}$
is attained.

\noindent \textbf{(d)} To show that 
\[
M_{km}=0,\qquad\forall k\in J,m\in J^{c}
\]
holds, is a word-by-word analogue of that property in the proof of
\prettyref{lem:LS-EMC-finite-K-M-invertible}. Therefore the \underline{flow condition}
is violated and we have proved the theorem. \end{proof}

\part{References and Indexes}

\appendix
\bibliographystyle{alpha}
\bibliography{../bib/dr,../bib/bib-b,../bib/bib-d,../bib/bib-e,../bib/bib-f,../bib/bib-g,../bib/bib-h,../bib/bib-j,../bib/bib-k,../bib/bib-l,../bib/bib-n,../bib/bib-o,../bib/bib-p,../bib/bib-s,../bib/bib-t,../bib/bib-v,../bib/bib-w,../bib/bib-z}

\printindex{}
\end{document}